\RequirePackage{tikz}
\documentclass[sn-mathphys]{sn-jnl}% Math and Physical Sciences Reference Style
\usepackage{cancel}
\usepackage{multicol}
\usepackage{caption}
\usepackage{bussproofs}
\usepackage{proof}

\jyear{2021}%

\theoremstyle{thmstyleone}%
\newtheorem{theorem}{Theorem}%  meant for continuous numbers
\newtheorem{proposition}[theorem]{Proposition}% 

\theoremstyle{thmstyletwo}%
\newtheorem{example}{Example}%
\newtheorem{lemma}{Lemma}

\newtheorem{corollary}{Corollary}
\theoremstyle{thmstylethree}%

\renewcommand{\restriction}{\mathord{\upharpoonright}}

\newcommand{\bit}{\begin{itemize}}
	\newcommand{\eit}{\end{itemize}}
\newcommand{\ben}{\begin{enumerate}}
	\newcommand{\een}{\end{enumerate}}
\newcommand{\bds}{\begin{description}}
	\newcommand{\eds}{\end{description}}

\newcounter{romc}

\newcounter{alphc}
\newtheorem{definition}{Definition}%
\newcommand{\blr}{\begin{list}{~(\roman{romc})~} {\usecounter{romc}
			\setlength{\topsep}{0pt} \setlength{\itemsep}{0pt}}}
	\newcommand{\elr}{\end{list}}
\newcommand{\bla}{\begin{list}{~(\alph{alphc})~} {\usecounter{alphc}
			\setlength{\topsep}{0pt} \setlength{\itemsep}{0pt}}}
	\newcommand{\ela}{\end{list}}

\begin{document}

\title[Kripke contexts, dBao and corresponding modal systems]{Kripke Contexts, Double Boolean Algebras with Operators and Corresponding Modal Systems}

%%=============================================================%%
%% Prefix	-> \pfx{Dr}
%% GivenName	-> \fnm{Joergen W.}
%% Particle	-> \spfx{van der} -> surname prefix
%% FamilyName	-> \sur{Ploeg}
%% Suffix	-> \sfx{IV}
%% NatureName	-> \tanm{Poet Laureate} -> Title after name
%% Degrees	-> \dgr{MSc, PhD}
%% \author*[1,2]{\pfx{Dr} \fnm{Joergen W.} \spfx{van der} \sur{Ploeg} \sfx{IV} \tanm{Poet Laureate} 
%%                 \dgr{MSc, PhD}}\email{iauthor@gmail.com}
%%=============================================================%%
\author[1]{\fnm{Prosenjit} \sur{Howlader}}\email{prosen@iitk.ac.in}

\author*[1]{\fnm{Mohua} \sur{Banerjee}}\email{mohua@iitk.ac.in}
%\equalcont{These authors contributed equally to this work.}

%\author[1,2]{\fnm{Third} \sur{Author}}\email{iiiauthor@gmail.com}
%\equalcont{These authors contributed equally to this work.}

\affil[1]{\orgdiv{Department of Mathematics and Statistics}, 
		             \orgname{Indian Institute of Technology  Kanpur},
		             %\orgaddress{ \city{Kanpur},
		               %\postcode{208016},
		               %\state{Uttar Pradesh},
		               \country{India}}

%\affil[2]{\orgdiv{Department}, \orgname{Organization}, \orgaddress{\street{Street}, \city{City}, \postcode{10587}, \state{State}, \country{Country}}}

%\affil[3]{\orgdiv{Department}, \orgname{Organization}, \orgaddress{\street{Street}, \city{City}, \postcode{610101}, \state{State}, \country{Country}}}

%%==================================%%
%% sample for unstructured abstract %%
%%==================================%%

\abstract{
%1. The formalism of a context of formal concept analysis and the approximation space of rough set theory are combined in this study. We introduce  the notion of a Kripke context, which includes relation between objects and relation between properties. A Kripke context induces a full complex algebra, and a double Boolean algebra with operators is defined on the abstraction of a full complex algebra.  For an adequate Kripke context, representation results for these algebras are established in terms of complex algebras.
%The article provides a proof system for these class of algebras, demonstrating their soundness and completeness in terms of algebraic structures. The representation theorems for the algebras result in another semantics for these logics.
The notion of 
%The formalism of 
a context in formal concept analysis and that of an approximation space in rough set theory are unified in this study to define a Kripke context. 
For any context (G,M,I), a  relation on the set G of objects and a relation on the set M of properties are included, giving a structure of the form ((G,R), (M,S), I). 
%The relations R and S in a Kripke context ((G,R), (M,S),I) are observed to induce respectively, interior-type and closure-type  operators on the sets of protoconcepts and semiconcepts of the underlying context (G,M,I). 
%Protoconcepts of any context form a contextual double Boolean algebra (dBa) and semiconcepts form a pure dBa. The interior and closure-type operators obtained here, motivate definitions of double Boolean algebra with operators (dBao), contextual and pure dBao. The case when R and S are reflexive and transitive lead to the definition of topological dBas.  
%%and topological pure dBa. 
%based on context, which includes relation between objects and relation between properties. 
%Depending  on the properties satisfied by the relations between objects and properties, reflexive, symmetric and transitive Kripke contexts are introduced. 
A Kripke context gives rise to  
complex algebras based on the collections of protoconcepts  and semiconcepts of the underlying context. On abstraction,  double Boolean algebras (dBas) with operators and topological dBas are defined. 
%on the abstraction of a  complex algebra. 
%The complex algebra based on protoconcepts 
%forms a contextual dBa, while the complex algebra of reflexive and transitive Kripke context form a topological contextual dBas. 
Representation results for these algebras are established in terms of the complex algebras of an appropriate Kripke context.  As a natural next step, logics corresponding
to  classes of these algebras are formulated. A sequent calculus is proposed for contextual dBas, modal extensions of which give logics for contextual dBas with operators and
topological contextual dBas. The representation theorems for the algebras result in a protoconcept-based  semantics for these logics.
%Protoconcepts of any context form a contextual double Boolean algebra (dBa) and semiconcepts form a pure dBa. The interior and closure-type operators obtained here, motivate the definition of a double Boolean algebra with operators (dBao). The case when R and S are reflexive and transitive leads to the definition of topological dBas. 
%Representation results for these algebras are established in terms of the dBaos of protoconcepts and semiconcepts, for an appropriate Kripke context. As a natural next step, logics corresponding to the algebras are formulated. A sequent calculus is proposed for contextual dBas, modal extensions of which give logics for contextual dBaos and topological contextual dBas.  
}


%%\pacs[JEL Classification]{D8, H51}

\section*{Declarations:}
\textbf{Funding:} The research of Mr. P. Howlader is supported by the \emph{Council of Scientific and Industrial Research} (CSIR) India - Research Grant No. 09/092(0950)/2016-EMR-I.
\vskip 3pt
\noindent \textbf{Conflicts of interest/Competing interests:} Not applicable.
\vskip 3pt
\noindent\textbf{Availability of data and material:} Not applicable
\vskip 3pt
\noindent\textbf{Code availability:} Not applicable.
\maketitle

\section{Introduction}
\label{intro}
Formal concept analysis (FCA) \cite{wille1982restructuring} and rough set theory \cite{pawlak1982rough} are both well-established areas of study with applications in several domains including knowledge representation and data analysis. There has also been a lot of study connecting and comparing the two areas, e.g. in \cite{yao2004comparative,aclmc,CACLkeyun,FCAARDT,RSAFCAyao,RCA,howlader2018algebras,howlader2020,saquer2001concept}, and the  work presented here is motivated by such studies from the perspective of algebra and logic. 

%has been established as a useful tool for data analysis. On the other hand, another well-established area of study  that has been  extensively used in data analysis is Rough set theory, founded by Pawlak \cite{pawlak1998rough}. There have been several studies linking and comparing the two theories, e.g. in \cite{yao2004comparative,aclmc,CACLkeyun,FCAARDT,RSAFCAyao,RCA,howlader2018algebras,howlader2020,saquer2001concept}. 

The central objects  of FCA  are contexts and concepts of a context \cite{ganter2012formal}. A {\it context} is a triple $\mathbb{K}:=(G, M, I)$, where $G$ is the set of {\it objects}, and $M$ is the set of {\it attributes} and $I\subseteq G\times M$. For any $A\subseteq G, B\subseteq M$, the following sets are defined:
$A^{\prime}:=\{m \in M:\mbox{for all} ~g\in G(g\in A\implies gRm)\}$, and
$ B^{\prime}:=\{g\in G:\mbox{for all} ~m\in M(m\in M\implies gRm)\}$. 
A pair $(A,B)$ is called a {\it concept} of $\mathbb{K}$, if $A^{\prime}=B$ and $B^{\prime}=A$. For a concept $(A,B)$, $A$ is its {\it extent} and $B$ its {\it intent}. $\mathcal{B}(\mathbb{K})$ denotes the set of all concepts of $\mathbb{K}$.
%of $\mathbb{K}$, $A:=ext((A,B))$ is its {\it extent} and $B:=int((A,B))$ is its {\it intent}. The set of concept of a context $\mathbb{K}$ is denoted by $\mathfrak{B}(\mathbb{K})$. 
An order relation $\leq$ is obtained on $\mathcal{B}(\mathbb{K})$ as follows:  $\mbox{for}~(A_{1},B_{1}),(A_{2},B_{2})\in\mathcal{B}(\mathbb{K}), (A_{1},B_{1}) \leq (A_{2},B_{2})~\mbox{if and only if}~ A_{1}\subseteq A_{2}~(\mbox{equivalent to}~ B_{2}\subseteq B_{1})$.

The notion of a concept was generalized to that of  semiconcepts and  protoconcepts in \cite{lpwille,wille}. A pair $(A,B)$ is called a {\it  semiconcept} of $\mathbb{K}$,  if $A^{\prime}=B$ or $B^{\prime}=A$. $(A,B)$ is called a {\it protoconcept} of $\mathbb{K}$,  if $A^{\prime\prime}=B^{\prime}$ (equivalently $A^{\prime}=B^{\prime\prime}$). $\mathfrak{H}(\mathbb{K})$ and $\mathfrak{P}(\mathbb{K})$ denote the sets of all semiconcepts and  protoconcepts of $\mathbb{K}$ respectively.  It is observed that $\mathcal{B}(\mathbb{K})\subseteq\mathfrak{H}(\mathbb{K})\subseteq\mathfrak{P}(\mathbb{K})$.
%A  concept of a  context represents a name, a definition \cite{Stumme2009} , a category \cite{HILSWLTS}  with in the context and the partial order represents the hierarchical order of concepts.
%Wille in \cite{lpwille} defined a mathematical model to represent conceptual knowledge, and
%to increase the  possibilities  of inferences,  introduced negations into the study.
  The partial order $\leq$ on $\mathcal{B}(\mathbb{K})$ is extended to the set  $\mathfrak{P}(\mathbb{K})$ as: for any $(A,B),(C,D)\in \mathfrak{P}(\mathbb{K})$,
%(\mathfrak{H}(\mathbb{K}))$, 
$(A,B)\sqsubseteq (C,D)$  if and only if $A\subseteq C $ and $D\subseteq B$. 
		 
\noindent The following operations are  defined on  $\mathfrak{P}(\mathbb{K})$. For $(A_{1},B_{1}),(A_{2},B_{2})$ in $\mathfrak{P}(\mathbb{K})$,
\begin{align*}
	(A_{1},B_{1})\sqcap(A_{2},B_{2}) & :=(A_{1}\cap A_{2}, (A_{1}\cap A_{2})^{'}),\\
	(A_{1},B_{1})\sqcup(A_{2},B_{2}) &:=((B_{1}\cap B_{2})^{'}, B_{1}\cap B_{2}),\\
	\neg(A,B) &:=(G\setminus A,(G\setminus A)^{'}),\\
	\lrcorner(A,B) &:=((M\setminus B)^{'}, M\setminus B),\\
	\top &:=(G,\emptyset),\\
	\bot &:=(\emptyset,M).\\
\end{align*}
\noindent With these operations, the protoconcepts of any context  form an  algebraic structure called {\it double Boolean algebra} (dBa) \cite{wille}. 
%which is moreover, {\it contextual}. 
The structure of a dBa  is such that there are two negation operators in it, which result in two Boolean algebras being derived from it -- justifying the name. The set of  semiconcepts, with the same operations as above, forms a subalgebra of the algebra of protoconcepts.   In this work, our interest lies  in {\it contextual}  and {\it pure} dBas \cite{vormbrock2005semiconcept,wille},  the structures formed by protoconcepts and semiconcepts  respectively.

There may be circumstances in which the objects and  properties defining a context are indistinguishable with respect to certain attributes. For example,  two diseases may be indistinguishable by the symptoms available. 
%One can think of diseases as objects and symptoms as attributes in this case. 
Indistinguishability of objects and properties have motivated  authors \cite{RCA,kent1996rough, saquer2001concept, CACLkeyun}  to study  ``indiscernibility'' relations on the set of objects and the set of properties. Rough  set-theoretic notions  of approximation  spaces and approximation operators   \cite{pawlak1982rough,pawlak2012rough} are then introduced in FCA.

 A {\it Pawlakian approximation space} is a pair $(W, E)$, where $W$ is a set and $E$  is an equivalence relation on $W$. This is generalised to a pair $(W, E)$ with $E$   any binary relation on $W$, and called a {\it generalised approximation space} \cite{yao1996generalization}.
%In this article, we are interested in the work of  Saquer and Deogun \cite{saquer2001concept}.
%An approximation space in a generalized rough set model \cite{yao1996generalization}  is a pair $(W, E)$, where $W$ is a set and $E$  is a binary relation on $W$. If $E$ is equivalence relation, $(W, E)$ is called  approximation space in  rough set model. 
For  $x\in W$,  $E(x):=\{y\in W~:~ xRy\}$.
The {\it lower} and {\it upper approximations} of any $A (\subseteq W)$ are defined respectively as $\underline{A}_{E}:=\{x\in W~:~E(x)\subseteq A\}$, and $\overline{A}^{E}:=\{x\in W~:~E(x)\cap A\neq\emptyset\}$.
%\blr
%\item $\underline{A}_{E}:=\{x\in W~:~E(x)\subseteq A\}$.
%\item $\overline{A}^{E}:=\{x\in W~:~E(x)\cap A\neq\emptyset\}$.
%\elr
%\noindent 
%If the relation is clear from the context, we shall omit the subscript and denote $\underline{A}_{E}$ by $\underline{A}$, $\overline{A}^{E}$ by $\overline{A}$.  
%\begin{proposition}
%	\label{pra}
%	{\rm 
%		\noindent \blr  \item[{\bf I.}] For an approximation space $(W,E)$, the following are hold.
%		\item $\overline{A}=(\underline{(A^{c})})^{c}, \underline{A}=(\overline{(A^{c})})^{c}$.
%		\item $\underline{W}= W$.
%		\item $\underline{A\cap B}=\underline{A}\cap\underline{B}, \overline{A\cup B}=\overline{A}\cup\overline{B}$.
%		\item $A\subseteq B$ implies that $\underline{A} \subseteq \underline{B}, \overline{A} \subseteq \overline{B}$.
%		
%		\item[{\bf II.}] Moreover if $E$ is a reflexive and transitive relation then following hold.
%		\item $\underline{A}\subseteq A$ and $A\subseteq \overline{A}$.
%		\item $\underline{(\underline{A})}=\underline{A}$ and $\overline{(\overline{A})}=\overline{A}$.
%		\elr}
%\end{proposition}
%In a rough set model $A\subset W$ is called definable if it is union of some equivalence classes. For example $\underline{A}$ and $\overline{A}$ are definable.
Kent introduced the notion of approximation space into FCA \cite{RCA,kent1996rough}, and defined lower and upper approximations of contexts and concepts.
%Let $\mathbb{K}:=(G,M, I)$ be a context and $(G, E)$ be an approximation space. Based on the  approximation space $(G, E)$, Kent defined {\it lower and upper approximations of the context} $\mathbb{K}$. Then lower and upper approximations {\it of a concept} of $\mathbb{K}$ are defined, in terms of concepts of the  lower and upper approximations of $\mathbb{K}$ respectively. 
%%Kent's  work has been applied, for instance, in the study of the basic Kent algebra and abstract Kent algebra  \cite{LRCA}. 
%In \cite{RCA, kent1996rough} the context $\mathbb{K}$ is called definable if $\{m\}_{I}^{\prime}$ is definable in $(G, E)$ for all $m\in M$. The following lower and upper approximation of $\mathbb{K}$ is defined in \cite{RCA, kent1996rough}.The lower approximation of $\mathbb{K}$ is a context $(G, M, \underline{I})$, where $\underline{I}(m)=\underline{I(m)}_{E}$ for all $m \in M$. The upper approximation of $\mathbb{K}$ is a context $(G, M, \overline{I})$, where $\overline{I}(m)=\overline{I(m)}_{E}$ for all $m \in M$.of $\mathbb{K}$, Kent proposed the following  lower and approximations of a concept $(A, B)\in \mathfrak{B}(\mathbb{K})$.The lower approximation of a concept $(A, B)\in \mathfrak{B}(\mathbb{K})$ is denoted by $\underline{(A, B)}_{E}$ and defined as $\underline{(A, B)}_{E}=$
%%In this article, we are interested in 
The work of Saquer and Deogun  \cite{saquer2001concept} differs from that of Kent in choosing the ``indiscernibility'' relations.  Kent considers an indiscernibility relation on the set $G$ of objects which is externally given by some agent, whereas  Saquer and Deogun consider a relation that is determined by the given context. For a given context $\mathbb{K}:= (G, M, I)$, relations $E_{1}, E_{2}$  are defined on the set $G$ of objects and the set $M$ of properties respectively,
%The latter defined  relations $E_{1}, E_{2}$ on the set $G$ of objects and the set $M$ of properties of a given context $\mathbb{K}:=(G,M,I)$, 
as follows.
\bla
\item For $g_{1},g_{2}\in G$, $g_{1}E_{1} g_{2}$ if and only if $I(g_{1})=I(g_{2})$.
\item For $m_{1}, m_{2}\in M$, $m_{1}E_{2}m_{2}$ if and only if  $I^{-1}(m_{1})= I^{-1}(m_{2})$.
\ela
\noindent Furthermore, for $A\subseteq G, B\subseteq M$, 
 lower and upper approximations  are defined in terms of concepts of $\mathbb{K}$, and using these, approximations of any pair $(A,B)$ that is not a concept, are  given.
%The lower and upper approximations of a non-definable concepts has been defined based on there component's lower and upper concept approximations. For more details see J Saquer and J.S.Deogun  \cite{saquer2001concept}. 
 Apart from Saquer and Deogun, Hu et.al. \cite{CACLkeyun} introduce approximation spaces on the sets of objects and properties. In  \cite{CACLkeyun}, for a given context $\mathbb{K}:= (G, M, I)$, relations $J_{1}, J_{2}$ are defined  on $G$  and $M$ respectively, as follows.
 \bla
\item For $g_{1},g_{2}\in G$, $g_{1}J_{1} g_{2}$ if and only if $I(g_{1})\subseteq I(g_{2})$.
\item For $m_{1}, m_{2}\in M$, $m_{1}J_{2}m_{2}$ if and only if  $I^{-1}(m_{1})\subseteq I^{-1}(m_{2})$.
\ela

\noindent The relations $E_{1}, E_{2}$ are equivalence relations \cite{saquer2001concept}, while the relations $J_{1}, J_{2}$ are partial order relations \cite{CACLkeyun}. These observations have motivated us to define the  Kripke context, which unifies within a single framework,  the notions of a context of FCA and approximation space of  rough set theory.
\begin{definition}
	\label{K-cntx}
	{\rm A {\it Kripke context} based on a context $\mathbb{K}:=(G,M,I)$ is a triple $\mathbb{KC}:=((G,R),(M,S),I)$, where $R,S$ are relations on $G$ and $M$ respectively.}
\end{definition}

\noindent So a Kripke context consists of a context of FCA and two Kripke frames, which in the terminology of rough set theory, are generalised approximation spaces.
Note that for a context $\mathbb{K}:=(G, M, I)$, we get a Kripke context $\mathbb{KC}_{DS}:=((G, E_{1}), (M, E_{2}), I)$. Moreover, $\mathbb{KC}_{DS}$ is an example such that the relations $E_{1}$ and $E_{2}$ are reflexive, symmetric and transitive. This observation has led us to  define reflexive, symmetric or transitive Kripke contexts, where 
%A Kripke context $\mathbb{KC}$ is called   {\it reflexive, symmetric and transitive Kripke context} if 
the relations $R$ and $S$ are reflexive, symmetric or transitive. 

It is shown that, using 
%The following are then proved.
%\begin{itemize}
   % \item[-] 
    the lower  and upper approximation operators induced by the approximation space  $(G,R)$, $(M,S)$ in a Kripke context $\mathbb{KC}:=((G,R),(M,S),I)$, one  can define unary operators $f_{R}$  and $f_{S}$ on the set $\mathfrak{P}(\mathbb{K})$ of protoconcepts of the underlying context $\mathbb{K}:=(G,M,I)$ such that 
  $f_{R}$  is an interior-type operator, while    $f_{S}$ is  a closure-type operator.
   % \item[-] The lower  approximation operators induced by the approximation spaces  $(M,S)$ in a Kripke context $\mathbb{KC}:=((G,R),(M,S),I)$, one  can define a closure-type operator $f_{S}$ on the set $\mathfrak{P}(\mathbb{K})$ of protoconcepts of the underlying context $\mathbb{K}:=(G,M,I)$. 
%\end{itemize}
%We show  that for a Kripke frame $\mathfrak{K}:=(X, R)$ there is a Kripke context $\mathbb{KC}_{0}:=((X, R), (X, S), \neq)$ based on $\mathbb{K}_{0}:=(X, X, \neq )$, where $S=R$. Now we consider the structure $(\mathfrak{P}(\mathbb{K}_{0}), \sqcup, \sqcap, \neg, \lrcorner, \top, \bot, f_{1}, f_{2})$, where $f_{R}=f_{1}$ and $f_{2}=f_{S}$. Then  $(\mathfrak{P}(\mathbb{K}_{0}), \sqcup, \sqcap, \neg, \lrcorner, \top, \bot, f_{1})$ is a Boa and $f_{1}^{\delta}=f_{2}$. Moreover $(\mathfrak{P}(\mathbb{K}_{0}), \sqcup, \sqcap, \neg, \lrcorner, \top, \bot, f_{1})$ is isomorphic to the complex algebra $\mathfrak{K}^{+}$ of $\mathfrak{K}$. This observation has led us to introduce complex algebra for a Kripke context. 
%distributive over $\sqcap$,  and $f_{S}$ is distributive over $\sqcup$. 
%So for a Kripke context $\mathbb{KC}$ the algebra  of protoconcept  $\underline{\mathfrak{P}}(\mathbb{K})$ is an instant of  algebra  with  operators of additive type($f_{S}$) and dual of  additive type($f_{R}$) . For a Kripke context $\mathbb{KC}$, 
The Kripke context thus leads to complex algebras. The algebra of protoconcepts with the operators $f_{R}$ and $f_{S}$, is called the {\it full complex algebra of $\mathbb{KC}$}. Any subalgebra of the full complex algebra of $\mathbb{KC}$ is called a complex algebra. For a Kripke context $\mathbb{KC}$, 
the algebra of semiconcepts $\mathfrak{H}(\mathbb{K})$ with operators $f_{R}\restriction{\mathfrak{H}(\mathbb{K})}$ and $f_{S}\restriction{\mathfrak{H}(\mathbb{K})}$  is an instance of a complex algebra of $\mathbb{KC}$.
We show how, in terms of approximation spaces and operators $f_{E_{1}}$ and $f_{E_{2}}$, the full complex algebra of the Kripke context  $\mathbb{KC}_{DS}$ can be utilised to compute all  the approximation operators defined in the work of Saquer and Deogun \cite{saquer2001concept}. 
%When generalised approximation spaces are taken over a set of objects and a set of properties, concept approximation can be defined, according to our discussion. 

To understand the equational theory of the full complex algebra of protoconcepts  and the complex algebra of semiconcepts, abstractions of these structures are defined: these are the {\it double Boolean algebras with operators} (dBao) and {\it topological  dBas} respectively. An immediate example of  a dBao is a Boolean algebra with operators \cite{blackburn2002moda}; a topological Boolean algebra \cite{Bao-I} gives an instance of a topological  dBa. It is shown that 
%\begin{itemize}
   % \item[-] 
    the full complex algebra of $\mathbb{KC}$ forms a contextual  dBao, while the complex algebra of semiconcepts forms  a pure  dBao.   
  %  \item[-]  
    For a reflexive and transitive  Kripke context, the full complex algebra  forms a topological contextual  dBa and the complex algebra of semiconcepts forms  a  topological pure  dBa.
%\end{itemize}
%Apart from the above class of examples, we give another class of examples of  dBa with operators and topological  dBas. Indeed, we prove the following.
%\begin{itemize}
%	\item  For a Kripke context $\mathbb{KC}$, the full complex algebra of $\mathbb{KC}$ forms contextual  dBa with operators.
%	\item The complex algebra based on the set of semiconcepts  $\mathfrak{H}(\mathbb{K})$ forms  a pure  dBa with operators.
%	\item For a reflexive symmetric and transitive  Kripke context $\mathbb{KC}$, the full complex algebra of $\mathbb{KC}$ forms topological contextual  dBa.
%	\item The complex algebra of a reflexive symmetric and transitive  Kripke context $\mathbb{KC}$, based on the set of semiconcepts  $\mathfrak{H}(\mathbb{K})$ forms  a  topological pure  dBa.
%\end{itemize}
%As soon as we came across these two different types of classes of algebras, we became interested in proving the 
%**The  Kripke context  constructed based on standared context \cite{wille} $\mathbb{K}(\textbf{D})$ of the underlying dBa $\textbf{D}$ of a dBao $\mathfrak{O}$ is denoted as $\mathbb{KC}(\mathfrak{O})$.**
Representation theorems for these classes of algebras are then proved, in terms of the complex algebras of protoconcepts and semiconcepts of an appropriate Kripke context.  The results are based on the representations obtained for dBas by 
%The representation theorems for these two classes of algebra is proved using the representation theorems of dBa and pure dBa. Former one proved by 
Wille \cite{wille} and  Balbiani \cite{BALBIANI2012260}.  
%**
%To prove representation theorem for dBa $\textbf{D}$ Wille construct a standard context $\mathbb{K}(\textbf{D})$. A standard context  $\mathbb{K}(\textbf{D})$ consist of set of primary filter and set of primary ideal of $\textbf{D}$, as a set of objects and set of properties respectively and a primary filter $F$ is related to a primary ideal $I$ if and only if they have empty intersection, $F\cap I=\emptyset$. 
%For a  dBa with operators $\mathfrak{O}$, whose underlying dBa is $\textbf{D}$, we construct a Kripke context $\mathbb{KC}(\mathfrak{O})$ base on  the standard context $\mathbb{K}(\textbf{D})$. Then the following representation theorems are proved.
%\begin{itemize}
%	\item[-] Any  dBa with operators $\mathfrak{O}$ is quasi-embeddable into the full complex algebra of $\mathbb{KC}(\mathfrak{O})$. In case if $\mathfrak{O}$  is contextual then the quasi-embedding becomes an embedding.  For  pure dBao $\mathfrak{O}$, it is embeddable into the complex algebra based on  the set of semiconcepts $\mathfrak{H}(\mathbb{K}(\textbf{D}))$.
%	\item[-]  If $\mathfrak{O}$ is a topological dBa, $\mathbb{KC}(\mathfrak{O})$ is a reflexive and transitive Kripke context. Moreover, the above representation theorems are also hold in this case.
%\end{itemize}

As a natural next step,  logics corresponding to dBaos are formulated. 
%**We proposed a sequent claclus \textbf{DBL} and extend to modal versions \textbf{MDBL} and \textbf{MDBL4}. We prove that  \textbf{DBL} is sound and complete with respect to class of contextual dBa and \textbf{MDBL}, \textbf{MDBL4} is sound and complete with respect to dBao and topological dBa.**
A sequent calculus, denoted \textbf{CDBL}, is proposed for contextual dBas.  \textbf{CDBL}  is  extended to \textbf{MCDBL} and \textbf{MCDBL4}  for the contextual dBaos and topological contextual dBas respectively. %We prove that \textbf{MPDBL}, \textbf{MPDBL4} are sound and complete with respect to the class of pure dBao and topological pure dBa respectively.
Due to the representation theorems for the algebras, one is able to get another semantics for these logics, based on protoconcepts of contexts.
%**We denote the class of all context by $\mathcal{K}$. 
%A model $\mathbb{M}$ consists of a context $\mathbb{K}$ and an assingment $v$ that associates each formula with a portoconcept of $\mathbb{K}$.  A sequent of \textbf{CDBL} is intrepreted as an inequality in $\mathfrak{P}(\mathbb{K})$.  The satisfaction and validity of a sequent are then defined in a standard way and prove the following  soundness and completeness result.** 
%\begin{itemize}
%    \item[-] \textbf{CDBL} is sound and complete with respect to the class $\mathcal{K}$ of all contexts.
%     \end{itemize}
% 
% \noindent  We also discuss the sound and completeness result for the logics  \textbf{MCDBL} and \textbf{MCDBL4}  with respect to the class of Kripke context.
%Logics for contextual dBas, contextual dBaos and topological contextual dBas are noted to be derivable from the formulations  of \textbf{PDBL} and its extensions. For these logics, one obtains another semantics which is based on protoconcepts of contexts.
%**pure/contextual to be inserted.

Section  \ref{preli} gives the preliminaries required for this work.  Kripke contexts, their examples  and the related complex algebras are studied
 %and the definition of reflexive symmetric and transitive Kripke context are introduced 
 in Section \ref{Kripke context}. In particular, we indicate  in Section \ref{applictioncomalgebra} how the various approximations defined in  \cite{saquer2001concept} can be expressed using terms of the full complex algebra of $\mathbb{KC}_{DS}$.
 %We define two operators on the set of protoconcepts of a underlying context $\mathbb{K}$ of a Kripke context $\mathbb{KC}$ and study there properties in Section \ref{Kripke context}. The definition of full complex algebra and complex algebra are also given in Section \ref{Kripke context}.  
 The dBaos and the topological dBa along with the representation results are presented in Section \ref{thealgebra}. 
 %dBas with operators  is defined in section \ref{DBA with operators}. The definition of topological dBa is given in Section \ref{TDBA}. 
 %and the representation theorems are proved in Sections \ref{RtdBaos} and
 %while the represntation theorems for topologica dBas is proved in section 
 %\ref{TDBA}.
In Section \ref{LCA}, the logics corresponding to  the algebras are studied.  \textbf{CDBL} for the class of contextual dBas is discussed in Section \ref{pdbl}; in Section \ref{mpdbl}, \textbf{CDBL} is extended  to \textbf{MCDBL} and \textbf{MCDBL4}. 
%for the class of topological pure dBas. 
In Section \ref{semisemantics}, the protoconcept-based semantics for the logics is given. 
%\textbf{PDBL} is also sound and complete for a semantics based on  semiconcepts. 
%For sound and completeness of \textbf{MPDBL} and \textbf{MPDBL4} for sementics based on the semiconcept, we give a sketch of proof. 
%Logics for the contextual algebras  are outlined in Section \ref{cdbl}. 
Section \ref{candfd} concludes the article. 
%our work and discuss some possible future directions.

In our presentation, the symbols   $\Rightarrow$, $\Leftrightarrow$, {\it and}, {\it or} and $not$ will be used with the usual meanings in the metalanguage. Throughout, for a map $f$ on $X$, $f\restriction{A}$ denotes the restriction of the map $f$  on $A\subseteq X$,  $\mathcal{P}(X)$ denotes the power set of any set $X$, and the complement of $A\subseteq X$ in a set $X$ is denoted $A^{c}$. For basic notions  on universal
algebra and lattices,  we refer  to \cite{ sHP,davey2002introduction}.

\section{Preliminaries}
\label{preli}
In the following subsections, we present basic notions and results related to dBas, Boolean algebras with operators and  approximation operators. Our primary references are \cite{ganter2012formal,wille,BALBIANI2012260, blackburn2002moda,Bao-I,saquer2001concept}.

\subsection{Double Boolean algebra}

A double Boolean algebra is defined as follows. 
\begin{definition}
	\label{DBA}
	{\rm \cite{wille}  An  algebra $ \textbf{D}:= (D,\sqcup,\sqcap, \neg,\lrcorner,\top,\bot)$,  satisfying the following properties is called a {\it double Boolean algebra} (dBa). For any $x,y,z \in D$,
		$\begin{array}{ll}
			(1a)~  (x \sqcap x ) \sqcap  y = x \sqcap  y  &
			(1b)~  (x \sqcup x)\sqcup  y = x \sqcup y \\
			(2a)~ x\sqcap y = y\sqcap  x  &
			(2b)~  x \sqcup   y = y\sqcup   x  \\
			(3a)~  x \sqcap ( y \sqcap  z) = (x \sqcap  y) \sqcap  z  &
			(3b)~  x \sqcup (y \sqcup  z) = (x \sqcup  y)\sqcup  z \\
			(4a)~ \neg (x \sqcap  x) = \neg  x  &
			(4b)~  \lrcorner(x \sqcup   x )= \lrcorner x \\
			(5a)~  x  \sqcap (x \sqcup y)=x \sqcap  x  &
			(5b)~  x \sqcup  (x \sqcap y) = x \sqcup   x \\
			(6a)~ x \sqcap  (y \vee z ) = (x\sqcap  y)\vee (x \sqcap  z) &
			(6b)~  x \sqcup  (y \wedge z) = (x \sqcup  y) \wedge  (x \sqcup  z) \\
			(7a)~  x \sqcap (x\vee y)= x \sqcap  x  &
			(7b)~  x\sqcup  (x \wedge  y) =x \sqcup   x \\
			(8a)~  \neg \neg (x \sqcap  y)= x \sqcap  y &
			(8b)~  \lrcorner\lrcorner(x \sqcup  y) = x\sqcup  y \\
			(9a)~  x  \sqcap \neg  x= \bot &
			(9b)~  x \sqcup \lrcorner x = \top  \\
			(10a)~ \neg \bot = \top \sqcap   \top  &
			(10b)~ \lrcorner\top =\bot \sqcup  \bot \\
			(11a)~  \neg \top = \bot  &
			(11b)~ \lrcorner\bot =\top \\
			(12)~  (x \sqcap  x) \sqcup (x \sqcap x) = (x \sqcup x) \sqcap (x \sqcup x), &
		\end{array}$
		
		\noindent where $ x\vee y := \neg(\neg x \sqcap\neg y)$, and 
		$ x \wedge y :=\lrcorner(\lrcorner x \sqcup \lrcorner y)$.
		A quasi-order relation $\sqsubseteq$ on $\textbf{D}$ is defined as follows:  
		$x\sqsubseteq y ~\mbox{if and only if}~  x\sqcap y=x\sqcap x~\mbox{and}~x\sqcup y=y\sqcup y,$ for any $x,y \in D$.}
\end{definition}
\noindent A dBa $\textbf{D}$ is called {\it contextual} if $\sqsubseteq$ is a partial order. A contextual dBa is also  known as a regular dBa \cite{breckner2019topological}. $\textbf{D}$ is  {\it pure} if for all $x\in D$, either $x\sqcap x=x$ or $x\sqcup x=x$. 
%Protoconcepts of a context $\mathbb{K}$ form a contextual dBa and semiconcepts of  $\mathbb{K}$ form a pure dBa.
%\begin{notation} {\rm 
		%For any dBa $\textbf{ D}:= (D,\sqcup,\sqcap,\neg,\lrcorner,\top,\bot)$, 
		In the following, let $ \textbf{\textit{D}}:=(D,\sqcup,\sqcap,\neg,\lrcorner,\top,\bot)$ be a dBa.
Let us give some notations that shall be used:\\
		$D_{\sqcap}:=\{x\in D~:~x\sqcap x=x\}$,  $D_{\sqcup}:=\{x\in D~:~x\sqcup x=x\}$, $D_{p}:=D_{\sqcap}\cup D_{\sqcup}$.\\
		For $x\in D$, $x_{\sqcap}:=x\sqcap x$ and $x_{\sqcup}:=x\sqcup x$. 
%} \end{notation}

%\noindent The above features of the Boolean algebras $\underline{\mathfrak{P}}(\mathbb{K})_{\sqcap}$ and $\underline{\mathfrak{P}}(\mathbb{K})_{\sqcup}$ play a role in our work, as we shall see in the sequel.

\begin{proposition}
	\label{puresub}
	{\rm  \cite{wille} $\textbf{D}_{p}:=(D_{p},\sqcup,\sqcap,\neg,\lrcorner,\top,\bot)$ is the largest pure subalgebra of $\textbf{D}$. Moreover, if $\textbf{D}$ is pure, $\textbf{D}_{p}=\textbf{D}$. }
\end{proposition}

\begin{proposition} 
	\label{order pure}
	{\rm \cite{BALBIANI2012260} Every pure dBa $\textbf{D}$ is contextual.} 
\end{proposition}
\begin{proposition} 
	\label{pro1}
	{\rm \cite{vormbrock} \noindent 
		%Let $\textbf{ D}= (D,\sqcup,\sqcap,\neg,\lrcorner,\top,\bot)$ be a dBa and $ \sqsubseteq $ is the quasi-order on $ D $. Then
		\begin{enumerate}
			\item $\textbf{D}_{\sqcap}:=(D_{\sqcap},\sqcap,\vee,\neg,\bot,\neg\bot)$ is a {\rm Boolean algebra} whose order relation is the restriction of $\sqsubseteq$ to $D_{\sqcap}$ and is denoted by $\sqsubseteq_{\sqcap}$.
			\item $\textbf{D}_{\sqcup}:=(D_{\sqcup},\sqcup,\wedge,\lrcorner,\top,\lrcorner\top)$ is a {\rm Boolean algebra} whose order relation is the restriction of $\sqsubseteq$ to $D_{\sqcup}$ and it is denoted by $\sqsubseteq_{\sqcup}$.
			\item $x\sqsubseteq y$ if and only if $x\sqcap x\sqsubseteq y\sqcap y$ and $x\sqcup x\sqsubseteq y\sqcup y$ for $x,y\in D$, that is, $x_{\sqcap}\sqsubseteq_{\sqcap} y_{\sqcap}$ and $x_{\sqcup}\sqsubseteq y_{\sqcup}$.
	\end{enumerate}}
\end{proposition}

\begin{proposition}
	\label{pro1.5}
	{\rm \cite{kwuida2007prime}
		Let 
		%$\textbf{ D}= (D,\sqcup,\sqcap,\neg,\lrcorner,\top,\bot)$ be a dBa and 
		$x,y,a\in D$. Then the following hold.
		\begin{enumerate}
			\item $x\sqcap\bot=\bot$ and $x\sqcup\bot=x\sqcup x$ that is $\bot\sqsubseteq x$.
			\item $x\sqcup \top=\top$ and $x\sqcap \top=x\sqcap x$ that is $x\sqsubseteq \top$.
			\item $x=y$ implies that $x\sqsubseteq y$ and $y\sqsubseteq x$.
			\item $x\sqsubseteq y$ and $y \sqsubseteq x$ if and only if $x\sqcap x=y\sqcap y$ and $x\sqcup x = y\sqcup y$.
			\item $x\sqcap y\sqsubseteq x,y\sqsubseteq x\sqcup y, x\sqcap y\sqsubseteq y,x\sqsubseteq x\sqcup y$.
			\item $x\sqsubseteq y$ implies $x\sqcap a\sqsubseteq y\sqcap a$ and $x\sqcup a\sqsubseteq y\sqcup a$. 
	\end{enumerate}}
\end{proposition}

\begin{proposition}
	\label{pro2}
	{\rm \cite{howlader3} For any $ x,y\in D$, the following  hold.
		\begin{enumerate}
			\item $\neg x\sqcap \neg x=\neg x$ and $\lrcorner x\sqcup \lrcorner x=\lrcorner x$, that is, $\neg x=(\neg x)_{\sqcap}\in D_{\sqcap}$, $\lrcorner x=(\lrcorner x)_{\sqcup}\in D_{\sqcup}$.
			\item $x\sqsubseteq y$  if and only if $ \neg y\sqsubseteq \neg x $ and $ \lrcorner y\sqsubseteq \lrcorner x $.
			\item $\neg\neg x=x\sqcap x$ and $\lrcorner\lrcorner x=x\sqcup x$.
			\item $x\vee y\in D_{\sqcap}, x\wedge y\in D_{\sqcup}$.
			\item $\neg\neg\bot=\bot$, and $\lrcorner\lrcorner \top=\top$.
			\item $\neg (x\sqcap y)=\neg x\vee \neg y$ and $\lrcorner(x\sqcup y)=\lrcorner x\wedge \lrcorner y$.
	\end{enumerate}}
\end{proposition}

%	\begin{theorem}
%	\label{generalization of boolean}
%	{\rm Any Boolean algebra $(D,\sqcap,\sqcup,\neg,\top,\bot)$ forms a contextual as well as pure dBa $\textbf{D}:=(D,\sqcap,\sqcup,\neg,\lrcorner,\top,\bot),$ where for all $a\in D$, $\lrcorner a:=\neg a$. On the other hand, a dBa $\textbf{D}:=(D,\sqcap,\sqcup,\neg,\lrcorner,\top,\bot)$ forms a Boolean algebra $(D,\sqcap,\sqcup,\neg,\top,\bot),$ if for all $a\in D$, $\neg a=\lrcorner a$ and $\neg\neg a=a$.  }
%	\end{theorem}

%\noindent Let  $\textbf{D}:=(D,\sqcup,\sqcap,\neg,\lrcorner,\top,\bot)$ be a dBa. 

\begin{definition}
	{\rm A subset $F$ of $D$ is a  {\it filter} in  $\textbf{D}$ if and only if  $x\sqcap y\in F$ for all $x,y \in F$, and for all $z\in D$ and $ x \in F, x\sqsubseteq z$ implies that $z\in F$. An {\it ideal} in a dBa is defined dually.\\
		A filter $F$ (ideal $I$) is  {\it proper} if and only if  $F\neq D$ ($I\neq D$).  A proper filter $F$ (ideal $I$) is called {\it primary} if and only if $x\in F~ \mbox{or} ~\neg x \in F ~(x\in I ~\mbox{or}~ \lrcorner x\in I)$, for all $x\in D$.\\
		The set of  primary filters is denoted by $\mathcal{F}_{pr}(\textbf{D})$;  the set of all primary ideals is denoted by $\mathcal{I}_{pr}(\textbf{D})$.\\
		A {\it base} $F_{0}$ for a filter  $F$ is a subset of $D$ such that $F=\{x\in D~:~ z\sqsubseteq x~\mbox{for some}~ z\in F_{0}\}$. A {\it base} for an ideal is defined similarly. \\
		For a subset $X$ of $D$, $F(X)$ and $I(X)$ denote the filter and ideal generated by $X$ respectively.}
\end{definition}

\begin{lemma}
	\label{gene-filtida}
	{\rm \cite{kwuida2007prime}	Let $F$ be a filter and $I$  an ideal  of  $\textbf{D}$. Then for any element $x\in D$, 
		\begin{enumerate}
			\item $F(F\cup \{x\})=\{a\in D~:~ x\sqcap w\sqsubseteq a~\mbox{for some}~ w\in F\}$.
			\item $I(I\cup \{x\})=\{a\in D~:~a\sqsubseteq x\sqcup w~\mbox{for some}~ w\in I\}$.
	\end{enumerate}}
\end{lemma}

\noindent The following  are introduced in  \cite{wille} to prove representation theorems for dBas. \\
$\mathcal{F}_{p}(\textbf{D}):=\{F \subseteq D : F ~\mbox{is a  filter of}~\textbf{D}~\mbox{and}~F\cap D_{\sqcap} ~\mbox{is a prime filter in}~ \textbf{D}_{\sqcap}\}$. \\
$\mathcal{I}_{p}(\textbf{D}):=\{I \subseteq D : I ~\mbox{is an ideal of}~\textbf{D}~\mbox{and}~I\cap D_{\sqcup}~\mbox{ is a prime ideal in}~ \textbf{D}_{\sqcup}\}$. 

\begin{proposition}
	\label{compar-ideal}
	{\rm \cite{howlader3} %Let \textbf{D} be a dBa. Then 
	$\mathcal{F}_{p}(\textbf{D})=\mathcal{F}_{pr}(\textbf{D})$ and $\mathcal{I}_{p}(\textbf{D})=\mathcal{I}_{pr}(\textbf{D})$.}
\end{proposition}
\begin{lemma}
	\label{lema1}
	{\rm \cite{wille} %Let $\textbf{D}$ be  a dBa. 
		\begin{enumerate}
			\item For any filter $F$ of $\textbf{D}$, $F\cap D_{\sqcap}$ and $F\cap D_{\sqcup}$ are filters of the Boolean algebras $\textbf{D}_{\sqcap}$, $\textbf{D}_{\sqcup}$ respectively.
			\item Each filter $F_{0}$ of the Boolean algebra $\textbf{D}_{\sqcap}$ is the base of some filter $F$ of $\textbf{D}$ such that $F_{0}=F\cap D_{\sqcap}$. Moreover if $F_{0}$ is prime, $F\in\mathcal{F}_{p}(\textbf{D})$. 
	\end{enumerate}
	It is straightforward to show that similar results hold for
ideals of dBas. }
\end{lemma}

%\vskip 2pt
%\begin{lemma}
%	\label{lema2}
%	{\rm Let $I$ be a ideal of a dBa $\textbf{D}$. Then,
%		\begin{enumerate}
%			\item $I\cap D_{\sqcap}$ and $I\cap D_{\sqcup}$ are ideal of the Boolean algebra $\textbf{D}_{\sqcap}$, $\textbf{D}_{\sqcup}$ respectively.
%			\item Each ideal $I_{0}$ of the Boolean algebra $\textbf{D}_{\sqcup}$ is the base of some ideal $I$ of $\textbf{D}$ such that $I_{0}=I\cap D_{\sqcup}$. Moreover if $I_{0}$  is prime , $I\in \mathcal{I}_{p}(\textbf{D})$. 
%	\end{enumerate}}
%\end{lemma}

For a context $\mathbb{K}:=(G, M, I)$ and sets  $A \subseteq G,B,\subseteq M$, recall the sets $A^\prime, B^\prime$  and   the operations on protoconcepts of $\mathbb{K}$ defined in Section \ref{intro}. 

\begin{lemma} \cite{davey2002introduction}
	\label{proty-prime}
	{\rm 	\noindent 
	%Let $\mathbb{K}:=(G, M, I)$ be a context and $A, X\subseteq G$ and $B, Y\subseteq M$. Then the following hold.
		\begin{enumerate}
			\item $A\subseteq A^{\prime\prime}$ and $B\subseteq B^{\prime\prime}$.
			\item $A\subseteq X$ implies that $X^{\prime}\subseteq A^{\prime}$,  $B\subseteq Y$ implies that $Y^{\prime}\subseteq B^{\prime}$, for any $X\subseteq G$ and $Y\subseteq M$.
	\end{enumerate}}
\end{lemma}
%\begin{definition}
%\label{DBAHOM}
%{\rm Let $\textbf{D}$ and $\textbf{M}$ be two dBas. A map  $h:M \rightarrow D$ is called a {\it homomorphism} if  $h$ preserves the operations in the algebras.\\
%$h$ is called {\it quasi-injective}, when $x\sqsubseteq y$ if and only if $h(x)\sqsubseteq h(y)$, for all $x,y\in M$. A  quasi-injective and surjective homomorphism is called  a {\it quasi-isomorphism} and a bijective homomorphism  is called an {\it isomorphism}. A order reversing and a bijective homomorphism is called {\it anti-isomorphism}.
%}
%\end{definition}

% protoconcepts and semiconcepts of  a context $\mathbb{K}$ form dBas.    
\begin{theorem}
	\label{protconcept algebra}
	{\rm \cite{wille} \noindent 
	%Let $\mathbb{K}$ be a context. 
		\begin{enumerate}
			\item  $\underline{\mathfrak{P}}(\mathbb{K}):= (\mathfrak{P}(\mathbb{K}), \sqcap,\sqcup,\neg,\lrcorner,\top,\bot)$ is a  contextual dBa.
			\item $\underline{\mathfrak{H}}(\mathbb{K}):=(\mathfrak{H}(\mathbb{K}), \sqcap,\sqcup,\neg,\lrcorner,\top,\bot)$ is a  pure dBa. Moreover, $\underline{\mathfrak{H}}(\mathbb{K})=\underline{\mathfrak{P}}(\mathbb{K})_{p}$.
		\end{enumerate}
	}
\end{theorem} 

\begin{theorem}
	{\rm \cite{wille}
		\label{protosemialgebra2}
		\begin{enumerate}
			\item The power set Boolean algebra $(\mathcal{P}(G), \cap, \cup,^{c}, G, \emptyset )$ is isomorphic to the Boolean algebra  $\underline{\mathfrak{P}}(\mathbb{K})_{\sqcap}:= (\mathfrak{P}(\mathbb{K})_{\sqcap}, \sqcap,\vee,\neg,\bot,\neg\bot)$, where any $A (\subseteq G)$ is mapped to $(A,A^{\prime}) \in \mathfrak{P}(\mathbb{K})_{\sqcap}$.
			\item The power set Boolean algebra $(\mathcal{P}(M), \cup, \cap, ^{c}, M, \emptyset )$ is anti-isomorphic to  the Boolean algebra $\underline{\mathfrak{P}}(\mathbb{K})_{\sqcup}:= (\mathfrak{P}(\mathbb{K})_{\sqcup}, \sqcup,\wedge, \lrcorner,\top,\lrcorner\top)$, where any $B (\subseteq M)$ is mapped to $(B^{\prime},B) \in \mathfrak{P}(\mathbb{K})_{\sqcup}$.
	\end{enumerate}}
\end{theorem}

Let us now move to  representation theorems for dBas. The following notations and results are needed.
%, the following result has been proved. To recall them, first we recall some notations:   
Let $\textbf{D}$ be a dBa. For any $x \in D$, \\$F_{x}:=\{F\in\mathcal{F}_{p}(\textbf{D})~:~x\in F\}$  and $I_{x}:=\{I\in\mathcal{I}_{p}(\textbf{D})~:~x\in I\}$.
\begin{lemma}
	\label{complement of Fx}
	{\rm \cite{wille,howlader2020} Let  $x\in D$. Then the following hold.
		\begin{enumerate}
			\item $(F_{x})^{c}=F_{\neg x}$ and $(I_{x})^{c}=I_{\lrcorner x}$.
			\item $F_{x\sqcap y}=F_{x}\cap F_{y}$ and $I_{x\sqcup y}=I_{x}\cap I_{y}$.
	\end{enumerate}}
\end{lemma}
To prove the representation theorem, Wille uses the {\it standard context} corresponding to the dBa $\textbf{D}$, defined as $\mathbb{K}(\textbf{D}):=(\mathcal{F}_{p}(\textbf{D}),\mathcal{I}_{p}(\textbf{D}),\Delta)$, where for all $F\in \mathcal{F}_{p}(\textbf{D}) $ and $I\in \mathcal{I}_{p}(\textbf{D})$, $F\Delta I$  if and only if $F\cap I=\emptyset$. Then we have 
\begin{lemma}
	\label{derivation}
	{\rm \cite{wille} For all $x\in \textbf{D}$, $F_{x}^{\prime}=I_{x_{\sqcap\sqcup}}$ and $I_{x}^{\prime}=F_{x_{\sqcup\sqcap}}$.}
\end{lemma}

\begin{theorem} 
	\label{protoembedding}
	{\rm \cite{wille} %Let $\textbf{D}$ be a dBa. 
	The map $h:\textbf{D}\rightarrow \underline{\mathfrak{P}}(\mathbb{K}(\textbf{D})) $ defined by $h(x):=(F_{x},I_{x})$ for all $x\in \textbf{D}$ is a quasi-embedding.}
\end{theorem} 
As a consequence of the above theorem, we have
\begin{corollary}
\label{repcdBa}
    {\rm For a contextual dBa $\textbf{D}$, the map $h:\textbf{D}\rightarrow \underline{\mathfrak{P}}(\mathbb{K}(\textbf{D})) $ defined by $h(x):=(F_{x},I_{x})$ for all $x\in \textbf{D}$ is an embedding.}
\end{corollary}
%\noindent In \cite{BALBIANI2012260}, Balbiani shows that this map is an embedding into $\mathfrak{H}(\mathbb{K}(\textbf{D}))$ for Pure dBa $\textbf{D}$. 
\begin{theorem} 
	\label{semiconceptembedding}
	{\rm \cite{BALBIANI2012260} Let $\textbf{D}$ be a  pure dBa. 
	The map $h:\textbf{D}\rightarrow \underline{\mathfrak{H}}(\mathbb{K}(\textbf{D})) $ defined by $h(x):=(F_{x},I_{x})$ for all $x\in \textbf{D}$ is  an embedding.}
\end{theorem}

\subsection{Boolean algebras with operators }

In the literature, there are several definitions of Boolean algebras with additional operators. 
 In this section, we mention the ones to be used in this work.  
%The following definition, we consider only a single unary operator. 
\begin{definition}{\rm \cite{blackburn2002moda}}
	\label{Bao}
	{\rm {\it A Boolean algebra with operators} (Bao) is an algebra $\mathfrak{A}:=(B, \vee, \wedge, \neg, 0, f)$ such that $(B, \vee, \wedge, \neg, 0)$ is a Boolean algebra and $f:B \rightarrow B$ satisfies the following.
	
		$\begin{array}{ll}
			\mbox{{\it Normality}:}~ f(0)=0, & \mbox{{\it Additivity}:}~ f(x\vee y)= f(x)\vee f(y).
		\end{array}$
	} 
\end{definition}
\noindent Note that  \cite{blackburn2002moda} gives a general definition of Baos with more than one operator. In \cite{Bao-I}, a Boolean algebra $(B, \vee, \wedge, \neg, 0)$ with only an additive operator $f$ is taken as the definition of Bao. 
%In this paper we work with the Definition \ref{Bao}.

\begin{definition}{\rm \cite{Bao-I}
		An algebra  $\mathfrak{A}:=(B, \vee, \wedge, \neg, 0, f )$ is called a {\it closure algebra} if $(B, \vee, \wedge, \neg, 0)$ is a Boolean algebra and for all $x, y\in B$, $f:B \rightarrow B$ satisfies the following conditions.	
			
		$\begin{array}{ll}
			1.~ f(0)=0.& 2.~ f(x\vee y)= f(x)\vee f(y).\\
			3.~ ff(x)=f(x).& 4.~ x\leq f(x).
		\end{array}$
		
	}
\end{definition}
\noindent Note that for a closure algebra $\mathfrak{A}:=(B, \vee, \wedge, \neg, 0, f )$, one can define an operator $g$ on B as:   $g(x):=\neg f(\neg x)$, for all $x\in B$. Then for all $x, y\in B$, 
%$g$ satisfies %the following.

$\begin{array}{ll}
	1^{\prime}.~ g(1)=1.& 2^{\prime}.~ g(x\wedge y)= g(x)\wedge g(y).\\
	3^{\prime}.~ gg(x)=g(x).& 4^{\prime}.~g(x) \leq x.
\end{array}$

\noindent An algebra $\mathfrak{A}:=(B, \vee, \wedge, \neg, 0, g )$, where $(B, \vee, \wedge, \neg, 0)$ is a Boolean algebra and $g$ satisfies $1^{\prime},2^{\prime}, 3^{\prime},4^{\prime}$ is called  a {\it topological Boolean algebra} in  \cite{Rasiowa}. Moreover, for a topological Boolean algebra  $\mathfrak{A}:=(B, \vee, \wedge, \neg, 0, g )$, one can define an operator $g^{\delta}(x):=\neg g(\neg x)$, for all $x\in D$ such that  $\mathfrak{A}:=(B, \vee, \wedge, \neg, 0, g^{\delta})$ is a closure algebra. In other words, a closure algebra and a topological Boolean algebra of  \cite{Rasiowa} are dual to each other and one can be obtained from the other. In this work, by a topological Boolean algebra, we shall mean a closure algebra.

\subsection{Approximation operators}	
\label{Appropefca}

%\noindent 
Recall the definitions of lower and upper approximation operators in an approximation space given in Section  \ref{intro}. If the relation is clear from the context, we shall omit the subscript and denote $\underline{A}_{E}$ by $\underline{A}$, $\overline{A}^{E}$ by $\overline{A}$.  
\begin{proposition}
	\label{pra}
	{\rm \cite{yao1996generalization}
		\noindent \blr  \item[{\bf I.}] For an approximation space $(W,E)$, the following  hold.
		\item $\overline{A}=(\underline{(A^{c})})^{c}, \underline{A}=(\overline{(A^{c})})^{c}$.
		\item $\underline{W}= W$.
		\item $\underline{A\cap B}=\underline{A}\cap\underline{B}, \overline{A\cup B}=\overline{A}\cup\overline{B}$.
		\item $A\subseteq B$ implies that $\underline{A} \subseteq \underline{B}, \overline{A} \subseteq \overline{B}$.
		
		\item[{\bf II.}] Moreover if $E$ is a reflexive and transitive relation then the following hold.
		\item $\underline{A}\subseteq A$ and $A\subseteq \overline{A}$.
		\item $\underline{(\underline{A})}=\underline{A}$ and $\overline{(\overline{A})}=\overline{A}$.
		\elr}
\end{proposition}
%In a rough set model $A\subset W$ is called definable if it is union of some equivalence classes. For example $\underline{A}$ and $\overline{A}$ are definable.**necessary?**

%Many authors define approximations operators in FCA. 
%In this section, we recall the work of  Saquer and Deogun  \cite{saquer2001concept}. 
Let $\mathbb{K}:=(G,M, I)$ be a context and recall the approximation spaces $(G, E_{1})$ and $(M, E_{2})$ mentioned in Section \ref{intro}. In \cite{saquer2001concept}, $A\subseteq G$ and $B\subseteq M$ are called {\it feasible} if $A^{\prime\prime}=A$ and $B^{\prime\prime}=B$.  Then the concept approximation(s) of $A$ are defined as follows.

%For a set of objects $A\subseteq G$, t
%The concept approximation of $A$ is a concept of $\mathbb{K}$, whose extent approximates  $A$ and defined as follows.
\begin{itemize}
	\item[-]  If  $A$ is feasible, the concept approximation of $A$ is $(A, A^{\prime})$.
	\item[-] If $A$ is not feasible,  $A$ is considered as s rough set of  the approximation space $(G, E_{1})$, and its  concept approximations are defined with the help of its lower approximation $\underline{A}_{E_{1}}$ and  upper approximation $\overline{A}^{E_{1}}$. The  {\it lower concept approximation}  of $A$ is the pair $((\underline{A}_{E_{1}})^{\prime\prime}, (\underline{A}_{E_{1}})^{\prime})$, while its {\it upper concept approximation} is $((\overline{A}^{E_{1}})^{\prime\prime}, (\overline{A}^{E_{1}})^{\prime})$. 
	%of $A$ are defined with the help of the lower approximation $\underline{A}_{E_{1}}$ and the upper approximation $\overline{A}^{E_{1}}$ of $A$.
\end{itemize}
For $B\subseteq M$:
%the concept approximations  are defined as follows.
\begin{itemize}
	\item[-] if $B$ is feasible, the concept approximation of $B$ is $(B^{\prime}, B)$;
	\item[-] if $B$ is non-feasible, 
	the lower and upper concept approximations of $B$ are defined by    $((\overline{B}^{E_{2}})^{\prime}, (\overline{B}^{E_{2}})^{\prime\prime})$ and $((\underline{B}_{E_{2}})^{\prime}, (\underline{B}_{E_{2}})^{\prime\prime})$ respectively.
	\end{itemize}

A pair $(A, B)$ is called a {\it non-definable} concept, if it is not a concept of the context $\mathbb{K}$. A concept is said to approximate such a pair $(A, B)$, if its  extent approximates A and intent approximates B.  The  four possible cases for $A,B$ are considered:  %in \cite{saquer2001concept}.
%\begin{enumerate}
	(i) both $A$ and $B$ are feasible, 
	(ii) $A$ is feasible and $B$ is not, 
	(iii)  $B$ is feasible and $A$ is not, and 
	(iv) both $A$ and $B$ are not feasible.
%\end{enumerate}
In case both  $A$ and $B$ are feasible and $A^{\prime}=B$ then the pair $(A, B)$ itself constitutes a concept and no approximations are needed. For the other cases, the lower  approximation of $(A, B)$ is obtained in terms of the meet of the lower concept approximations of its individual components, while the upper approximation of $(A, B)$ is obtained in terms of the join of the upper concept approximations of its individual components. For example, 
%we discuss the case 4.  Let $(A, B)$ be a non-definable concept, where both 
consider case (iv), when both  $A$ and $B$ are not feasible.

\noindent The {\it lower  approximation} of $(A, B)$ is defined by $\underline{(A, B)}:=\\((\underline{A}_{E_{1}})^{\prime\prime}, (\underline{A}_{E_{1}})^{\prime})\sqcap((\overline{B}^{E_{2}})^{\prime}, (\overline{B}^{E_{2}})^{\prime\prime}) =((\underline{A}_{E_{1}})^{\prime\prime}\cap(\overline{B}^{E_{2}})^{\prime}, ((\underline{A}_{E_{1}})^{\prime\prime}\cap(\overline{B}^{E_{2}})^{\prime})^{\prime} )$.

\noindent The {\it upper  approximation} of $(A, B)$ is defined by $\overline{(A, B)}:=\\((\overline{A}^{E_{1}})^{\prime\prime}, (\overline{A}^{E_{1}})^{\prime})\sqcup ((\underline{B}_{E_{2}})^{\prime}, (\underline{B}_{E_{2}})^{\prime\prime}) =(((\overline{A}^{E_{1}})^{\prime}\cap (\underline{B}_{E_{2}})^{\prime\prime})^{\prime}, (\overline{A}^{E_{1}})^{\prime}\cap (\underline{B}_{E_{2}})^{\prime\prime} )$.
\vskip 3pt

\noindent Let us illustrate these notions by an example. The following context $(G,M,I)$ is a subcontext of a context given by Wille \cite{ganter2012formal} with some modifications. $G:=\{Leech, Bream,Frog,Dog, Cat\}$ and $M:=\{a,b,c,g\}$, where a:= needs water to live, b:= lives in water, c:= lives on land, g:=can move around.
$I$ is given by Table \ref{context-1}, where * as an entry corresponding to object $x$ and property $y$ means $xIy$ holds.

% **Give caption of Table and label the table.
\begin{table}[ht]
	\begin{center}
		\caption {Context $\mathbb{K}$} \label{context-1}
		\begin{tabular}{ | l | l | l | l | l |}
			\hline	
			& a & b& c & g\\ \hline
			Leech & * &*  &  & * \\ \hline
			Bream & * & * &  & * \\ \hline
			Frog & * & * & * & * \\ \hline
			Dog & * &  & * & * \\ \hline
			Cat & * &  &  * & *  \\ \hline
		\end{tabular}
		
	\end{center}
\end{table}
\noindent Observe that the properties a and g are indiscernible by objects, while Leech and Bream as well as Dog and Cat are indiscernible by properties. The  induced approximation spaces are $(G, \{\{Leech, Bream\},\{Frog\}, \{Dog, Cat\}\})$ and \\$(M, \{\{a,\mbox{g}\}, \{b\},\{c\}\})$.

%, where $G:=\{Leech, Bream,Frog,Dog, Cat\}$ and $M:=\{a,b,c,g\}$.
Let $A:=\{Leech, Bream , Dog\}$ and $B:=\{a,c\}$. $A$ is not feasible,  as $A^{\prime\prime}\neq A$.  $B$ is also non-feasible. The upper and lower concept approximations of $A$ are $(G,\{a, g\})$ and  $(\{ Leech, Bream, Frog\}, \{a,b,\mbox{g}\})$,  respectively. The upper and lower  concept approximations of $B$  are both given by $(\{Frog,Dog,Cat\}, \{a,\mbox{g},c\})$. 
%and $(\{Frog, Dog, Cat\}, \{a,c,g\})$, respectively. \\
Moreover, $(A, B)$ is a non-definable concept. The lower approximation  of $(A, B)$ is $(\{Frog\}, M)$  and the  upper  approximation is $(G,\{a, g\})$. 
%**Similar to the above one can obtain the others concept approximations.**
% \vskip 2pt

\section{Kripke context}
\label{Kripke context}

%In this section, we recall the definition of the Kripke context(Definition \ref{K-cntx}) defined in Section \ref{intro} and  give definitions of full complex algebra and complex algebra of a Kripke context.   Next, we define reflexive, symmetric, and transitive Kripke context and give  examples of Kripke contexts.

As given by Definition \ref{K-cntx} in Section \ref{intro}, a   Kripke context based on a context $\mathbb{K}:=(G,M,I)$ is a triple $\mathbb{KC}:=((G,R),(M,S),I)$, where $R,S$ are binary relations on $G$ and $M$ respectively. 
Let us give a couple of examples of Kripke contexts. 
The first example is based on Pawlakian approximation spaces.

\begin{example}
	\label{rst-kcxt}
	{\rm 	$\mathbb{KC}:=((G, R), (M, S), I)$, where $G:=\{D_1,D_2,D_3,D_4\}$ represents a collection of diseases and  $M:=\{S_1,S_2,S_3,S_4,S_5\}$ a collection  of symptoms. $D_iIS_j$ holds if disease $D_i$ has symptom $S_j$, and $I$ is given by Table \ref{context-2}. Equivalence relations $R$ on $G$ and $S$ on $M$ are then induced as follows, relating respectively, the diseases that have the same set of symptoms, and the symptoms that apply to the same set of diseases:\\
		$D_i R D_j$, if and only if $I(D_i)= I(D_j)$, $i,j \in \{1,2,3,4\}$ and   $S_i R S_j$, if and only if $I^{-1}(S_i)= I^{-1}(S_j)$, $i,j \in \{1,2,3,4,5\}$.\\
		One thus gets the approximation spaces $(G, R)$ and $(M, S)$.
		%$\mathbb{KC}:=((G, \{\{De-1, De-4\},\{De-2\}, \{De-3\}\}), (M, \{\{Sy-1, Sy-2\}, \{Sy-3, Sy-5\},\{Sy-4\}\}), I)$ is based on the context $\mathbb{K}:=(G, M, I)$, which is given by the following table.
		%**Give caption of Table and label the table.
		\begin{table}[ht]
			\begin{center}
				\caption {Context $\mathbb{K}$} \label{context-2}
				
				\begin{tabular}{ | l | l | l | l | l |l| }
					\hline	
					& $S_1$ & $S_2$ & $S_3$ & $S_4$& $S_5$\\ \hline
					$D_1$& * &*  &  & * &\\ \hline
					$D_2$ &  &  & * &  &*\\ \hline
					$D_3$ &  &  & * & * &*\\ \hline
					$D_4$ & * & * &  & * &\\ \hline
				\end{tabular}
				
			\end{center}
		\end{table}
		% where the set of objects $G:=\{De-1, De-2, De-3, De-4\}$ consist some  instant of  diseases and  the set of properties  $M:=\{Sy-1, Sy-2, Sy-3, Sy-4, Sy-5\}$ consist some symptoms. A diseases $De-i$  and a symptoms $Sy-j$ is related $De-i I Sy-j$ if a patient has  the disease $De-i$ then that patient will have the symptoms $Sy-j$.  The relation $I$ between diseases and symptoms is represented by * in the above table.  Now, it may be the case when two diseases have similar symptoms, for example  Influenza (Flu) and COVID-19. In this case two diseases are indiscernible by symptoms. So the above context induced the approximation space $(G, \{\{De-1, De-4\},\{De-2\}, \{De-3\}\})$, where two diseases $De-i$ and $De-j$ are indiscernible if they have similar symptoms that is $I(De-i)= I(De-j)$. Similarly, we say two Symptoms $Sy-i$ and $Sy-j$ are indiscernible if $I^{-1}(Sy-i)= I^{-1}(Sy-j)$, which induced the approximation space $(M, \{\{Sy-1, Sy-2\}, \{Sy-3, Sy-5\},\{Sy-4\}\})$.   
	}	
\end{example}

Our next example is motivated by the notion of bisimulation between Kripke frames \cite{blackburn2002moda}. It gives a Kripke context  $\mathbb{KC}:=((G, R), (M, S), I)$ such that the relation $I$ is in fact, a bisimulation between the Kripke frames $(G, R)$ and $(M, S)$, that is, it satisfies the back and forth conditions:  for all $g\in G$ and $m\in M$, \\
%$\mathbb{M}_{1}=((G, R), v_{1})$ and $\mathbb{M}_{2}=((M, S), v_{2})$. A relation $Z\subseteq G\times M$ is called a bisimulation if 
%\blr 
%\item For all $g\in G$ and $m\in M$, $gZm$ if and only if 
%\subitem[a] They satisfied same propositional variables.
%\subitem[b]  
for all $g_{1}\in G$ $(gRg_{1}$ and $gIm\implies \mbox{there exists}~ m_{1}\in M(mSm_{1}~\mbox{and}~g_{1}Im_{1}))$;\\
%\subitem[c]  F
for all $m_{1}\in M$ $(mSm_{1}$ and $gIm \implies \mbox{there exists}~ g_{1}\in M(gRg_{1}~\mbox{and}~g_{1}Im_{1}))$.
%\elr  
%
%As of now , we are interested on Kripke frame, we drop the condition [a]. 
%We say a Kripke context $\mathbb{KC}:=((G, R), (M, S), I)$ is bisimilar if $I$ satisfies [b] and [c]. Our next example is a Kripke context of this kind.
%and reformulate the above definition. We say a relation  $Z\subseteq G\times M$ is called bisimulation between two Kripke frame $(G, R)$ and $(M, S)$ if $Z$ satisfies [b] and [c]. Now observed that two 

\begin{example}
	\label{rt-s-kcxt}
	{\rm   $\mathbb{KC}:=((G, R), (M, S), I)$, where $G:=\{c, d, e\}$,  $M:=\{a, b\}$,  $R:=\{  (d, e), ( c, d)\}$ and  $ S:=\{(a, b), (b, a)\}$.  
		%The Kripke context $\mathbb{KC}$ is based on the context $\mathbb{K}:=(G, M, I)$ and the context 
		$I$ is given by Table \ref{context-3}.  Figure \ref{bisi-Kc} depicts the objects, properties and the three relations $R,S,I$.  Each circular node represents an object and each rectangular node a property. Two circular nodes are connected by an arrow if they are related by $R$. Similarly for the rectangular nodes. The dotted arrow represents the relation $I$. From the figure it is clear  that $I$ satisfies the back and forth conditions.
		%The relation $I$ between object and property is represented by * in Table**.
		
		\begin{table}[ht]
			\begin{center}
				\caption {Context $\mathbb{K}$} \label{context-3}
				\begin{tabular}{ | c | c | c | c | }
					\hline	
					& a & b  \\ \hline
					c & * &  \\ \hline
					d &  & * \\ \hline
					e & * &  \\ \hline
				\end{tabular}
			\end{center}
		\end{table}	
		
		%** Correct the arrowheads in the dotted arrows
		\begin{figure}[h]
			
			\begin{center}
				
				\begin{tikzpicture}
					\node (waiting 1) at ( 2,1) [shape=rectangle,draw]  {a};
					\node (critical 1) at ( 2,-1)[shape=rectangle,draw]  {b};
					\node (semaphore) at ( -1,-3)[circle,draw]  {e};
					\node (leave critical) at ( -1,-1)[circle,draw]  {d};
					\node (enter critical) at (-1,1)[circle,draw]  {c};
					\draw [dashed,<-] (critical 1.west) -- (leave critical.east);
					\draw [dashed,<-] (waiting 1.west) -- (enter critical.east);
					\draw [->] (leave critical.south)--(semaphore.north);
					%\draw [->] (enter critical.west)--(semaphore.east);
					\draw [->] (enter critical.south) -- (leave critical.north);
					\draw [<->] (waiting 1.south) -- (critical 1.north);
					\draw [dashed,<-] (waiting 1.west) -- (semaphore.east);
					%\draw[->] (semaphore) to [ loop above]  (semaphore);
					%	\draw[->] (leave critical) to [loop below]  (leave critical);
					%	\draw[->] (enter critical) to [ loop above]  (enter critical);
					
				\end{tikzpicture}
				\caption{Kripke Context $\mathbb{KC}$}
				\label{bisi-Kc}
			\end{center}
		\end{figure}
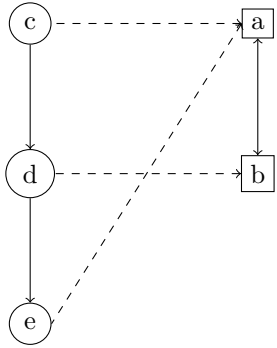
	}
\end{example}

In a Kripke context $\mathbb{KC}:=((G,R),(M,S),I)$, if $(G,R)$ is a Pawlakian approximation space,  one gets an interior  operator $-_{R}:\mathcal{P}(G)\rightarrow \mathcal{P}(G) $ defined as  $-_{R}(A):=\underline{A}_{R}$ for all $A\in \mathcal{P}(G)$ (Proposition \ref{pra}). 
%it follows that $-_{R}$ is an interior operator. %which makes the dual $\mathcal{P}(G)^{\delta}$ of 
%$(\mathcal{P}(G), \cap, \cup,^{c}, G, \emptyset, -_{R} )$,  a Boolean algebra with operators.   
Similarly, one has the interior operator $-_{S}:\mathcal{P}(M)\rightarrow \mathcal{P}(M) $   defined by $-_{M}(B):=\underline{B}_{S}$ for all $B\in \mathcal{P}(M)$, if $(M,S)$ is a  Pawlakian approximation space.
%$(\mathcal{P}(M), \cup, \cap,^{c}, M, \emptyset, -_{S} )$ is a Boolean algebra with operators, where $-_{S}:\mathcal{P}(M)\rightarrow \mathcal{P}(M) $  is defined by $-_{M}(B)=\underline{B}_{S}$ for all $B\in \mathcal{P}(M)$. 
Now from Theorem \ref{protosemialgebra2}, we get
%$\underline{\mathfrak{P}}(\mathbb{K})_{\sqcap}$ is 
the isomorphism  
%to $(\mathcal{P}(G), \cap, \cup,^{c}, G, \emptyset )$ and the isomorphism 
$f:\mathcal{P}(G)\rightarrow  \mathfrak{P}(\mathbb{K})_{\sqcap}$  given by $f(A):=(A, A^{\prime})$ for all $A\in \mathcal{P}(G)$ and the 
%whereas $\underline{\mathfrak{P}}(\mathbb{K})_{\sqcup}$ is antiisomorphic to $(\mathcal{P}(M), \cup, \cap,^{c}, M, \emptyset )$ and the  
anti-isomorphism $g:\mathcal{P}(M)\rightarrow  \mathfrak{P}(\mathbb{K})_{\sqcup}$  given by $g(B):=(B^{\prime}, B)$ for all $B\in \mathcal{P}(M)$. 
Taking a cue from the compositions of $f,-_{R}$ and $g,-_{S}$, we can define two unary operators $f_R$ and $f_S$ on $ \mathfrak{P}(\mathbb{K})$ as given below. It will be seen in Theorem \ref{oprator-proto} that $f_R$ is an interior-type operator on $ \mathfrak{P}(\mathbb{K})$, while $g_S$ is a closure-type operator on $ \mathfrak{P}(\mathbb{K})$.
%Utilizing  $f,-_{R}$ on the one hand, one gets the unary operator $f_{R}$ below, which acts as an interior-type operator on $ \mathfrak{P}(\mathbb{K})$. On the other hand, using   $g$ and  $-_{S}$, we get  the unary operator $f_{S}$ which acts as a closure-type operator on $ \mathfrak{P}(\mathbb{K})$.
%a couple of operators 
%we can lift the interior type operator $-_{R}$ to an interior type operator $i$ on $\mathfrak{P}(\mathbb{K})_{\sqcap}$ as follows, $i((A, A^{\prime}))=(\underline{A}_{R},(\underline{A}_{R})^{\prime})$. Dually using the anti-isomorphism$g$, we can lift the interior type operator $-_{S}$
%\noindent to a closure type operator $c$ on $\mathfrak{P}(\mathbb{K})_{\sqcup}$ as follows $c(B^{\prime}, B)=((\underline{B}_{S})^{\prime},\underline{B}_{S})$ for all $(B^{\prime}, B)\in \mathfrak{P}(\mathbb{K})_{\sqcup}$. Moreover, the interior type operator $i$ and the closure type operator $c$ can be extend to an  interior type operator $f_{R}$ and a closure type operator $f_{S}$ on $\mathfrak{P}(\mathbb{K})$. This motivate us the following studies.**\\ 
%Let $\mathbb{KC}:=((G,R),(M,S),I)$ be a Kripke context based on the context $\mathbb{K}:=(G,M,I)$. 
%$\mathfrak{H}(\mathbb{K})$.
%Based on the relations $R$ and $S$, 
%the unary operators $f_{R}, f_{S}$ 
%on the set $ \mathfrak{P}(\mathbb{K})$ are defined 
%as follows. 
For any $(A,B)\in \mathfrak{P}(\mathbb{K})$,
%$\mathfrak{H}(\mathbb{K})$ as follows.
%\blr
%\item $f_{R}:\mathfrak{H}(\mathbb{K})\rightarrow\mathfrak{H}(\mathbb{K})$ defined by $f_{R}((A,B)):=(\underline{A}_{R},(\underline{A}_{R})^{\prime})$, for all $(A,B)\in \mathfrak{H}(\mathbb{K})$.
%\item $f_{S}:\mathfrak{H}(\mathbb{K})\rightarrow\mathfrak{H}(\mathbb{K})$ defined by $f_{S}((A,B)):=((\underline{B}_{S})^{\prime},\underline{B}_{S})$, for all $(A,B)\in \mathfrak{H}(\mathbb{K})$.
%\elr
\bit
\item 
%$f_{R}: \mathfrak{P}(\mathbb{K}) \rightarrow\mathfrak{P}(\mathbb{K})$ defined by 
$f_{R}((A,B)):=(\underline{A}_{R},(\underline{A}_{R})^{\prime})$,
\item 
%$f_{S}:\mathfrak{P}(\mathbb{K})\rightarrow\mathfrak{P}(\mathbb{K})$ defined by 
$f_{S}((A,B)):=((\underline{B}_{S})^{\prime},\underline{B}_{S})$.
\eit

\noindent $f_{R}, f_{S}$ are well-defined, as $(\underline{A}_{R}, (\underline{A}_{R})^{\prime})$ and $((\underline{B}_{S})^{\prime},\underline{B}_{S})$ are both   semiconcepts and hence protoconcepts of $\mathbb{K}$. This implies that the set $\mathfrak{P}(\mathbb{K})$ of protoconcepts  is closed under the operators $f_{R}, f_{S}$. We have

\begin{definition}
	\label{full-complexalg}
	{\rm Let $\mathbb{KC}:=((G,R),(M,S),I)$ be a Kripke context. The {\it full complex algebra} of $\mathbb{KC}$, $\underline{\mathfrak{P}}^{+}(\mathbb{KC}):=(\mathfrak{P}(\mathbb{K}),\sqcup,\sqcap,\neg,\lrcorner,\top,\bot,f_{R},f_{S})$, is the expansion  of the algebra $\underline{\mathfrak{P}}(\mathbb{K})$ of protoconcepts with the operators $f_{R}$ and $f_{S}$. \\
		%The full complex algebra of $\mathbb{KC}$ is denoted by . \\
		Any subalgebra of  $\underline{\mathfrak{P}}^{+}(\mathbb{KC})$ is called a {\it complex algebra} of $\mathbb{KC}$. }
\end{definition}

Let $f_{R}^{\delta}, f_{S}^{\delta}$ denote the operators on $ \mathfrak{P}(\mathbb{K})$ that are {\it dual} to $f_{R}, f_{S}$ respectively. In other words, for each $x:=(A,B)\in \mathfrak{P}(\mathbb{K})$, \\$f_{R}^{\delta}(x):=\neg f_{R}(\neg x)=\neg f_{R}((A^{c}, A^{c\prime}))=\neg (\underline{A}^{c}_{R},(\underline{A}^{c}_{R})^{\prime})=((\underline{A}^{c}_{R})^{c},(\underline{A}^{c}_{R})^{c\prime})=(\overline{A}^{R},(\overline{A}^{R})^{\prime})$, by Proposition \ref{pra}(i).\\
Similarly  $f_{S}^{\delta}(x):=\lrcorner f_{S}(\lrcorner x)=((\overline{B}^{S})^{\prime}, \overline{B}^{S})$. \\
Again, note that $f_{R}^{\delta}(x) = (\overline{A}^{R}, (\overline{A}^{R})^{\prime})$ and $f_{S}^{\delta}(x) = ((\overline{B}^{S})^{\prime},\overline{B}^{S})$ are   semiconcepts of $\mathbb{K}$.
%In the following let $\mathbb{KC}:=((G,R),(M,S),I)$ be any Kripke context. 
%A primary goal of this work is to investigate $\underline{\mathfrak{P}}^{+}(\mathbb{KC})$  from the algebraic point of view. 
Let us now list 
% To study the algebraic structure  $\underline{\mathfrak{P}}^{+}(\mathbb{KC})$ first, we have to study 
some properties of  $f_{R}$ and $f_{S}$.

\begin{theorem}
	\label{oprator-proto}
	{\rm %Let $\mathbb{KC}:= ((G,R), (M,S), I)$ be a Kripke context. Then 
		For all $x, y\in \mathfrak{P}(\mathbb{K})$, the following hold.
		\begin{enumerate}
			\item $f_{R}(x\sqcap y)=f_{R}(x)\sqcap f_{R}(y)$ and $f_{S}(x\sqcup y)=f_{S}(x)\sqcup f_{S}(y)$.
			\item $f_{R}(x\sqcap x)=f_{R}(x)$ and $f_{S}(x\sqcup x)=f_{S}(x)$.
			\item $f_{R}(\neg\bot)=\neg\bot$ and $f_{S}(\lrcorner\top)=\lrcorner\top$.
			\item $f_{R}(\neg x)=\neg f_{R}^{\delta}(x)$ and $f_{S}(\lrcorner x)=\lrcorner f_{S}^{\delta}(x)$. 
	\end{enumerate} }
	
\end{theorem}

\begin{proof}
	Let $x:=(A,B)$ and $y:=(C,D)$. \\
	1.	We use Proposition \ref{pra}(iii) in the following equations.
	$f_{R}((A,B)\sqcap(C,D))=f_{R}(A\cap C, (A\cap C)^{\prime})=(\underline{A\cap C}_{R},(\underline{A\cap C}_{R})^{\prime} )=(\underline{A}_{R}\cap \underline{C}_{R},(\underline{A}_{R}\cap \underline{C}_{R})^{\prime})=(\underline{A}_{R},(\underline{A}_{R})^{\prime})\sqcap (\underline{C}_{R}, (\underline{C}_{R})^{\prime})=f_{R}((A,B))\sqcap f_{R}((C,D))$.\\ 
	$f_{S}((A,B)\sqcup (C,D))=f_{S}((B\cap D)^{\prime},B\cap D)=((\underline{B\cap D}_{S})^{\prime},\underline{B\cap D}_{S})=f_{S}((A,B))\sqcup f_{S}((C,D))$.\\
	2. %Let $(A,B)\in \mathfrak{P}(\mathbb{K})$. Then 
	$f_{R}((A,B)\sqcap (A,B))=f_{R}((A, A^{\prime}))=(\underline{A}_{R},(\underline{A}_{R})^{\prime})= f_{R}((A,B))$. Similarly, one can show that  $f_{S}((A,B)\sqcup (A,B))=((\underline{B}_{S})^{\prime}, \underline{B}_{S})$.\\ 
	3.  $f_{R}(\neg\bot)=f_{R}((G,G^{\prime}))=(\underline{G}_{R},(\underline{G}_{R})^{\prime})=(G,G^{\prime})=\neg\bot$, by Proposition \ref{pra}(ii). Similarly, one gets  $f_{S}(\lrcorner \top)=\lrcorner\top$.\\
	4. $f_{R}(\neg(A,B))=f_{R}(A^{c}, A^{c\prime})=(\underline{A^{c}}_{R}, (\underline{A^{c}}_{R})^{\prime})=((\overline{A}^{R})^{c}, (\overline{A}^{R})^{c\prime})$ by Proposition \ref{pra}(i). So $f_{R}(\neg(A, B))=\neg (\overline{A}^{R}, (\overline{A}^{R})^{\prime} )=\neg f_{R}^{\delta}((A, B))$. Similarly, one can show that $f_{S}(\lrcorner (A, B))=\lrcorner f_{S}^{\delta}((A, B))$.  
\end{proof}
Using Theorem \ref{oprator-proto}(1,3,4), one obtains
\begin{corollary}
	\label{dual-ope}
	{\rm For all $x, y\in \mathfrak{P}(\mathbb{K})$,
		\begin{enumerate}
			\item  $f_{R}^{\delta}(x\vee y)=f_{R}^{\delta}(x)\vee f_{R}^{\delta}(y)$ and $f_{S}^{\delta}(x\wedge y)=f_{S}^{\delta}(x)\wedge f_{S}^{\delta}(y)$.
			\item $f_{R}^{\delta}(\bot)=\bot$ and $f_{S}^{\delta}(\top)=\top$.
	\end{enumerate}}
\end{corollary}

%Let $\mathbb{KC}$ be a Kripke context and $\underline{\mathfrak{P}}(\mathbb{K})$ be the corresponding  algebra of protoconcepts. Now, we recall the Boolean algebras $\underline{\mathfrak{P}}(\mathbb{K})_{\sqcap}$ and $\underline{\mathfrak{P}}(\mathbb{K})_{\sqcup}$. 

Consider the restriction maps $f_{R}\restriction{\mathfrak{P}(\mathbb{K})_{\sqcap}}$ 
%be the restriction of $f_{R}$ on $\mathfrak{P}(\mathbb{K})_{\sqcap}$ 
and  $f_{S}\restriction{\mathfrak{P}(\mathbb{K})_{\sqcup}}$. 
%be the restriction of $f_{S}$ on $\mathfrak{P}(\mathbb{K})_{\sqcup}$. Now, f
From Theorem \ref{oprator-proto}(2),  it follows that $\mathfrak{P}(\mathbb{K})_{\sqcap}$ and $\mathfrak{P}(\mathbb{K})_{\sqcup}$  are closed under $f_{R}\restriction{\mathfrak{P}(\mathbb{K})_{\sqcap}}$ and $f_{S}\restriction{\mathfrak{P}(\mathbb{K})_{\sqcup}}$ respectively. Using   Theorem \ref{oprator-proto}(1,3) and  Corollary \ref{dual-ope}, we get 

\begin{corollary}
	{\rm 
		%\noindent  \begin{enumerate}
		%\item 
		$\underline{\mathfrak{P}}(\mathbb{KC})^{+}_{\sqcap}:=(\mathfrak{P}(\mathbb{K})_{\sqcap}, \sqcap, \vee, \neg, \bot, f_{R}^{\delta}\restriction{\mathfrak{P}(\mathbb{K})_{\sqcap}})$ and 
		$\underline{\mathfrak{P}}(\mathbb{KC})^{+}_{\sqcup}:=\\(\mathfrak{P}(\mathbb{K})_{\sqcup}, \sqcup, \wedge, \lrcorner, \top, f_{S}\restriction{\mathfrak{P}(\mathbb{K})_{\sqcup}})$ are Baos.
		%\end{enumerate}
	}
\end{corollary}

%Then   by 1 and 3  of Theorem \ref{oprator-proto}. On the other hand,  Corollary \ref{dual-ope} implies that  
%So the full complex algebra $\underline{\mathfrak{P}}(\mathbb{KC})^{+}$ 
%of a Kripke context $\mathbb{KC}$ is an expansion of algebra of protoconcept $\underline{\mathfrak{P}}(\mathbb{K})$ of $\mathbb{K}$ with operators $f_{R}, f_{S}$ 
%such that $\mathfrak{P}(\mathbb{K})_{\sqcap}$ and $\mathfrak{P}(\mathbb{K})_{\sqcup}$ form Bao  with respect to $f_{R}^{\delta}\restriction{\mathfrak{P}(\mathbb{K})_{\sqcap}}$ and $f_{S}\restriction{\mathfrak{P}(\mathbb{K})_{\sqcup}}t$.

%The full complex algebra $\underline{\mathfrak{P}}(\mathbb{KC})^{+}$  becomes more interesting if we impose various conditions on the relation R and S. 
%**Equivalence relations on the sets of objects and properties play an important role in the literature of unifying contexts of FCA and approximation spaces of rough set theory. Apart from   Saquer and Deogun  \cite{saquer2001concept}, Kent \cite{kent1996rough}  considers equivalence relations on the set of objects. These  motivate us to 

We next consider a Kripke context $\mathbb{KC}:=((G,R),(M,S),I)$ where the relations $R,S$ satisfy certain properties that are of particular relevance here.

\begin{definition}
	{\rm \noindent 
		\begin{enumerate} 
			\item $\mathbb{KC}$ is {\it reflexive  from the left}, if $R$ is reflexive. 
			\item $\mathbb{KC}$ is {\it reflexive from the right}, if $S$ is reflexive. 
			\item $\mathbb{KC}$ is {\it reflexive}, if it is reflexive from both left and right.
			%Let $\mathbb{KC}=((G,R),(M,S),I)$ be a Kripke context.
			%		\begin{enumerate}
			%			\item $\mathbb{KC}$ is reflexive from the left if $R$ is reflexive and is reflexive from the right if $S$ is reflexive. $\mathbb{KC}$ is reflexive if it is reflexive from both left and right.
			%			\item $\mathbb{KC}$ is symmetric from the left if $R$ is symmetric and is symmetric from the right if $S$ is symmetric. $\mathbb{KC}$ is symmetric if it is symmetric from both left and right.
			%			\item $\mathbb{KC}$ is transitive from left if $R$ is transitive and is transitive from right if $S$ is transitive. $\mathbb{KC}$ is transitive if it is transitive from both left and right.			
			%\item $\mathbb{KC}$ is right-unbounded form left if $R$ is right-unbounded and is right-unbounded form right if $S$ is right-unbounded. $\mathbb{KC}$ is right-unbounded if it is right-unbounded from both left and right.			
			%\item $\mathbb{KC}$ is no branching to the right form left if $R$ is no branching to the right and is no branching to the right form right if $S$ is no branching to the right. $\mathbb{KC}$ is no branching to the right if it is no branching to the right from both left and right.			
		\end{enumerate}
		The cases for {\it symmetry} and {\it transitivity} of $\mathbb{KC}$ are similarly defined.}
\end{definition}
\noindent Observe that  the Kripke context in Example \ref{rt-s-kcxt} is symmetric from the right.
\begin{theorem}
	\label{ope-rel-and-tran-kcxt}
	{\rm  Let $\mathbb{KC}:= ((G,R), (M,S), I)$ be a reflexive and transitive Kripke context. Then for all $x\in \mathfrak{P}(\mathbb{K})$, the following hold. 
		\begin{enumerate}
			\item $f_{R}(x)\sqsubseteq x$ and $x\sqsubseteq f_{S}(x)$.
			\item $f_{R}f_{R}(x)=f_{R}(x)$ and $f_{S}f_{S}(x)=f_{S}(x)$.
	\end{enumerate} }
\end{theorem}
\begin{proof}
	%**Introduce suffixes**
	1. Let $(A, B)\in \mathfrak{P}(\mathbb{K})$. By Proposition \ref{pra}(v) $\underline{A}_{R}\subseteq A$ and $\underline{B}_{S}\subseteq B$, which implies that $A^{\prime}\subseteq (\underline{A}_{R})^{\prime}$ and $B^{\prime}\subseteq (\underline{B}_{S})^{\prime}$. Now $ B^{\prime\prime}=A^{\prime}$ and $ A^{\prime\prime}=B^{\prime}$, as $(A, B)\in \mathfrak{P}(\mathbb{K})$. By Lemma \ref{proty-prime}, $A\subseteq A^{\prime\prime}$ and  $B\subseteq B^{\prime\prime}$. So $A\subseteq B^{\prime}$ and $B\subseteq A^{\prime}$, which implies that $B\subseteq (\underline{A}_{R})^{\prime} $ and $A\subseteq (\underline{B}_{S})^{\prime}$. Therefore $f_{R}((A,B))=(\underline{A}_{R},(\underline{A}_{R})^{\prime})\sqsubseteq(A,B)$ and  $(A,B)\sqsubseteq f_{S}((A,B))=((\underline{B}_{S})^{\prime}, \underline{B}_{S})$.\\ 
	2. $f_{R}f_{R}((A,B))=f_{R}((\underline{A}_{R},(\underline{A}_{R})^{\prime}))=(\underline{(\underline{A}_{R})}_{R}, (\underline{(\underline{A}_{R})}_{R})^{\prime})=(\underline{A}_{R},(\underline{A}_{R})^{\prime})= f_{R}((A,B))$, by Proposition \ref{pra}(vi). Similarly, one can show that $f_{S}f_{S}((A,B))=f_{S}((A,B))$.	 
\end{proof}

Theorems \ref{ope-rel-and-tran-kcxt}   and  \ref{oprator-proto}(4) give
\begin{corollary}
	\label{dul-topo}
	{\rm For all $x\in \mathfrak{P}(\mathbb{K})$,
		$x\sqsubseteq f_{R}^{\delta}(x)$ and $f_{R}^{\delta}f_{R}^{\delta}(x)=f_{R}^{\delta}(x)$.
	}
\end{corollary}
Further, using  Theorems \ref{oprator-proto},  \ref{ope-rel-and-tran-kcxt} and Corollaries \ref{dual-ope},  \ref{dul-topo}, we get  
\begin{corollary}
	{\rm 
		% \noindent  \begin{enumerate}
		%	\item 
		$\underline{\mathfrak{P}}(\mathbb{KC})^{+}_{\sqcap}:=(\mathfrak{P}(\mathbb{K})_{\sqcap}, \sqcap, \vee, \neg, \bot, f_{R}^{\delta}\restriction{\mathfrak{P}(\mathbb{K})_{\sqcap}})$ and $\underline{\mathfrak{P}}(\mathbb{KC})^{+}_{\sqcup}:=\\(\mathfrak{P}(\mathbb{K})_{\sqcup}, \sqcup, \wedge, \lrcorner, \top, f_{S}\restriction{\mathfrak{P}(\mathbb{K})_{\sqcup}})$ are topological  Boolean algebras.
		%\end{enumerate}
	}
\end{corollary}
% \footnote{Note that the definition of Boolean algebra with operators is used here is actually dual to the definition of Boolean algebra with operators defined in \cite{Rasiowa}}  and  

\subsection{Complex algebra to concept approximation}
\label{applictioncomalgebra}
Recall the Kripke context $\mathbb{KC}_{DS}:=((G, E_{1}), (M, E_{2}), I)$ defined in Section \ref{intro}, where $(G, E_{1}), (M, E_{2})$ are Pawlakian approximation spaces. We observe that  terms of the full complex algebra $\underline{\mathfrak{P}}^{+}(\mathbb{KC}_{DS})$ are able to express the various notions of concept approximations mentioned in Section \ref{Appropefca}. Indeed, for $\mathbb{KC}_{DS}$, 
%Then the approximation spaces $(G, E_{1})$ and $(M, E_{2})$ induced the following approximation operators on power set $\mathcal{P}(G)$ and $\mathcal{P}(M)$.
%$-_{E_{1}}:\mathcal{P}(G)\rightarrow \mathcal{P}(G) $ defined as  $-_{E_{1}}(A):=\underline{A}_{E_{1}}$ for all $A\in \mathcal{P}(G)$.
%$-_{E_{2}}:\mathcal{P}(M)\rightarrow \mathcal{P}(M) $ defined as  $-_{E_{2}}(B):=\underline{B}_{E_{1}}$ for all $B\in \mathcal{P}(M)$.
%From Proposition \ref{pra} it follows that $-_{E_{1}}$ and $-_{E_{2}}$ are modal operators on  $\mathcal{P}(G)$ and $\mathcal{P}(M)$ respectively.
%Using the operators  $-_{E_{1}}$ and $-_{E_{2}}$, 
we get the operators $f_{E_{1}},f_{E_{2}}:\mathfrak{P}(\mathbb{K}) \rightarrow \mathfrak{P}(\mathbb{K})$ as  above, that is,
%and shown to be modal operators.
$f_{E_{1}}((A,B)):=(\underline{A}_{E_{1}},(\underline{A}_{E_{1}})^{\prime})$, and
$f_{E_{2}}((A,B)):=((\underline{B}_{E_{2}})^{\prime},\underline{B}_{E_{2}})$ for any $(A,B)\in \mathfrak{P}(\mathbb{K})$. Moreover, $f_{E_{1}}^{\delta}((A,B)) = (\overline{A}^{E_{1}}, (\overline{A}^{E_{1}})^{\prime})$ and $f_{E_{2}}^{\delta}((A,B)) = ((\overline{B}^{E_{2}})^{\prime},\overline{B}^{E_{2}})$.
Let $A\subseteq G$ and $B\subseteq M$.\\
If  $A$ and $B$ are feasible then the concept approximations of $A$ and $B$ are $(A, A^{\prime})$ and $(B^{\prime}, B)$ respectively and these are elements of  $\mathfrak{P}(\mathbb{K})$.\\
%$(A, A^{\prime}), (B^{\prime}, B)\in \mathfrak{P}(\mathbb{K})$.** So?**\\
Suppose $A$ and $B$ are both non-feasible sets.
%We observe that the lower concept
%approximations of A and B are expressible using the operators $f_{E_1}$ and $f_{E_2}$ respectively.
Let $x, y\in \mathfrak{P}(\mathbb{K})$ be such that the extent of $x$ is $A$ and intent of $y$ is $B$. Then we have the following.

\noindent The lower concept approximation of $A$, $ ((\underline{A}_{E_{1}})^{\prime\prime}, (\underline{A}_{E_{1}})^{\prime})=(\underline{A}_{E_{1}}, (\underline{A}_{E_{1}})^{\prime})\sqcup (\underline{A}_{E_{1}}, (\underline{A}_{E_{1}})^{\prime})=f_{E_{1}}(x)\sqcup f_{E_{1}}(x)$. 

\noindent The upper concept approximation of $A$, $((\overline{A}^{E_{1}})^{\prime\prime}, (\overline{A}^{E_{1}})^{\prime})= (\overline{A}^{E_{1}}, (\overline{A}^{E_{1}})^{\prime})\sqcup (\overline{A}^{E_{1}}, (\overline{A}^{E_{1}})^{\prime})=f_{E_{1}}^{\delta}(x)\sqcup f_{E_{1}}^{\delta}(x)$.

\noindent The lower concept approximation of $B$, $((\overline{B}^{E_{2}})^{\prime}, (\overline{B}^{E_{2}})^{\prime\prime})=((\overline{B}^{E_{2}})^{\prime},\overline{B}^{E_{2}})\sqcap  ((\overline{B}^{E_{2}})^{\prime},\overline{B}^{E_{2}})=f^{\delta}_{E_{2}}(y)\sqcap f^{\delta}_{E_{2}}(y)$.

\noindent  The  upper concept approximation of $B$, 
$ ((\underline{B}_{E_{2}})^{\prime}, (\underline{B}_{E_{2}})^{\prime\prime} )=((\underline{B}_{E_{2}})^{\prime},\underline{B}_{E_{2}})\sqcap ((\underline{B}_{E_{2}})^{\prime},\underline{B}_{E_{2}})  =f_{E_{2}}(y)\sqcap f_{E_{2}}(y)$. 

%Now observed that for $A\subseteq G$ and $B\subseteq M$ the concept approximations of $A$ and $B$ are expressible by the terms of the full complex algebra $\underline{\mathfrak{P}}^{+}(\mathbb{KC}_{DS})$. 
\noindent Now by definition, approximations of any pair $(A, B)$ are obtained using the concept approximations of $A$ and $B$. As shown above, the latter are all expressible  by the terms of the full complex algebra, and hence we have the observation. For instance, suppose, $(A, B)$ is a non-definable concept of $\mathbb{K}$ with $A$ and $B$  non-feasible. 
%The  lower and upper concept approximations of $(A, B)$ are also expressible using $f_{E_1},f_{E_2}$ and their duals. Indeed, The above equations in $\underline{\mathfrak{P}}^{+}(\mathbb{KC}_{DS})$ give \\

\noindent The lower  approximation of $(A, B)$, $((\underline{A}_{E_{1}})^{\prime\prime}\cap (\overline{B}^{E_{2}})^{\prime}, ((\underline{A}_{E_{1}})^{\prime\prime}\cap (\overline{B}^{E_{2}})^{\prime})^{\prime})\\=(f_{E_{1}}(x)\sqcup f_{E_{1}}(x))\sqcap (f^{\delta}_{E_{2}}(y)\sqcap f^{\delta}_{E_{2}}(y))=(f_{E_{1}}(x)\sqcup f_{E_{1}}(x))\sqcap f^{\delta}_{E_{2}}(y)$.

\noindent The upper approximation of $(A, B)$, $(((\overline{A}^{E_{1}})^{\prime}\cap (\underline{B}_{E_{2}})^{\prime\prime})^{\prime}, (\overline{A}^{E_{1}})^{\prime}\cap (\underline{B}_{E_{2}})^{\prime\prime} ) =(f^{\delta}_{E_{1}}(x)\sqcup f^{\delta}_{E_{1}}(x))\sqcup (f_{E_{2}}(y)\sqcap f_{E_{2}}(y))=(f_{E_{2}}(y)\sqcap f_{E_{2}}(y))\sqcup f^{\delta}_{E_{1}}(x)$. 
%\vskip 2pt
%\noindent Hence the observation.
%Now if  the pair $(A, B)$ is a non-definable concepts, from the above equations 
%It follows that $(f_{E_{1}}(x)\sqcup f_{E_{1}}(x))\sqcap f^{\delta}_{E_{2}}(y)$ is the lower concept approximation of $(A, B)$, and $(f_{E_{2}}(y)\sqcap f_{E_{2}}(y))\sqcup f^{\delta}_{E_{1}}(x)$ is 

%***In \cite{saquer2001concept}  the lower concept approximation of an arbitrary non-definable concept $(A, B)$  is obtain by taking meet($\sqcap$)  of the concept approximations of $A$ and $B$ and on the other hand  the upper concept approximation of an arbitrary non-definable concept $(A, B)$  is obtain by taking join($\sqcup$)  of the concept approximations of $A$ and $B$. It can be shown that all concept approximations of $A$ and $B$ is expressible using the terms of the algebra $\underline{\mathfrak{P}}^{+}(\mathbb{KC}_{DS})$.
%So, we can express all the concept approximations of a non-definable concept using the terms of the full complex algebra $\underline{\mathfrak{P}}^{+}(\mathbb{KC}_{DS})$. ***
\vskip 4pt

\section{The algebras}
\label{thealgebra}
In this section, we study abstractions of the algebraic structure  $\underline{\mathfrak{P}}^{+}(\mathbb{KC})$ obtained in Section \ref{Kripke context}.
%we investigate the algebraic structure $\underline{\mathfrak{P}}^{+}(\mathbb{KC})$ and in detail. For that, we propose 
These are   dBas with operators (Definition \ref{DBA with operators}), and topological  dBas (Definition \ref{tdBa}). 
%We prove the representation theorems for the classes of algebras.
\subsection{Double Boolean algebras with operators }

\begin{definition}
	\label{DBA with operators}
	{\rm A  structure $\mathfrak{O}:=(D,\sqcup,\sqcap,\neg,\lrcorner,\top,\bot, \textbf{I},\textbf{C})$ is a {\it dBa with operators} (dBao) provided\\
		1.  $(D,\sqcup,\sqcap,\neg,\lrcorner,\top,\bot)$ is a dBa and\\ 
		2. $\textbf{I},\textbf{C}$ are monotonic operators on $D$ satisfying the following for any $x,y\in D$.
	\vskip 2pt	
		$\begin{array}{ll}
			1a~ \textbf{I}(x \sqcap y) = \textbf{I}(x) \sqcap \textbf{I}(y) &
			1b~ \textbf{C}(x \sqcup y) = \textbf{C}(x)\sqcup \textbf{C}(y)\\ 
			2a~ \textbf{I}(\neg\bot)= \neg\bot  &
		2b~\textbf{C}(\lrcorner\top)=\lrcorner\top\\
			3a~ \textbf{I}(x\sqcap x)=\textbf{I}(x) &
			3b~ \textbf{C}(x\sqcup x)=\textbf{C}(x)
			 	
		\end{array}$
		\vskip 2pt	
		\noindent	A {\it  contextual dBao}  is a dBao in which the underlying dBa is  contextual. If the underlying dBa is pure,  the dBao is called a {\it pure dBao}.\\\noindent The duals of $\textbf{I}$ and $\textbf{C}$ with respect to $\neg,\lrcorner$ are defined as
		$\textbf{I}^{\delta}(a):=\neg \textbf{I}(\neg a)$ and $\textbf{C}^{\delta}(a):=\lrcorner \textbf{C}(\lrcorner a)$ for all $a\in D$.
		}
\end{definition}
%\begin{notation}
	%{\rm  
	
		%}
%\end{notation}

\noindent Any Bao  provides a trivial example of a contextual and pure dBao. Indeed, in a Bao  $(B,\sqcap,\sqcup,\neg,\top,\bot,f)$, setting $\lrcorner=\neg$, $\textbf{C}:=f$ and $\textbf{I}:=f^\delta$, one obtains the dBao $(B,\sqcap,\sqcup,\neg,\lrcorner,\top,\bot,\textbf{I},\textbf{C})$. Due to the idempotence of the operators $\sqcap,\sqcup$ in the Boolean algebra $(B,\sqcap,\sqcup,\neg,\top,\bot)$, the dBa $(B,\sqcap,\sqcup,\neg,\lrcorner,\top,\bot)$ is pure; 
as $B_\sqcap= B_\sqcup = B$, the dBa is contextual as well.
\vskip 2pt
An immediate consequence  is the following.
\begin{theorem}
	\label{largestpuredBao}
	{\rm Let $\mathfrak{O}:=(D,\sqcup,\sqcap,\neg,\lrcorner,\top,\bot,\textbf{I},\textbf{C})$ be a dBao. Then
		\begin{enumerate}
			\item $\mathfrak{O}_{p}:=(D_{p}, \sqcup,\sqcap,\neg,\lrcorner,\top,\bot, \textbf{I}\restriction{D_{p}},\textbf{C}\restriction{D_{p}} )$ is the largest pure subalgebra of $\mathfrak{O}$.
			\item If $\mathfrak{O}$ is  pure, it is  contextual and moreover, $\mathfrak{O}=\mathfrak{O}_{p}.$
	\end{enumerate}}
\end{theorem}
\begin{proof}
	
	1. From Proposition \ref{puresub} it follows that $(D_{p}, \sqcup,\sqcap,\neg,\lrcorner,\top,\bot)$ is the largest pure subalgebra of $\textbf{D}$. To complete the proof it is sufficient to  show that $D_{p}$ is closed under $\textbf{I}$ and $\textbf{C}$, which follows from Definition \ref{DBA with operators}(1a, 3a, 1b, 3b).\\
	2. Proposition \ref{order pure} gives the first part. 
	%$\mathfrak{O}=\mathfrak{O}_{p}$, as 
	For any pure dBa, $D=D_{p}$.  
\end{proof}

%Indeed, the set of protoconcept $\mathfrak{P}(\mathbb{K})$ of a Kripke context $\mathbb{KC}$ forms a dBao and the set of semiconcepts $\mathfrak{H}(\mathbb{K})$ of  a Kripke context $\mathbb{KC}$ forms a pure dBao.
As intended, the sets of protoconcepts and semiconcepts of a context provide examples of dBaos:
\begin{theorem}
	\label{complex algebra}
	{\rm Let $\mathbb{KC}:=((G,R),(M,S),I)$ be a Kripke context based on the context $\mathbb{K}:=(G,M,I)$. Then the following hold.
		\begin{enumerate} 
			\item$\underline{\mathfrak{P}}^{+}(\mathbb{KC}):=(\mathfrak{P}(\mathbb{K}),\sqcup,\sqcap,\neg,\lrcorner,\top,\bot,f_{R},f_{S})$ is a contextual dBao.
			
			\item $\underline{\mathfrak{H}}^{+}(\mathbb{KC}):=(\mathfrak{H}(\mathbb{K}),\sqcup,\sqcap,\neg,\lrcorner,\top,\bot,f_{R}\restriction{\mathfrak{H}(\mathbb{K})},f_{S}\restriction{\mathfrak{H}(\mathbb{K})})$  is a pure dBao. It is the largest pure subalgebra of $\underline{\mathfrak{P}}^{+}(\mathbb{KC})$, that is, $\underline{\mathfrak{P}}^{+}(\mathbb{KC})_{p}=\underline{\mathfrak{H}}^{+}(\mathbb{KC})$. 
	\end{enumerate} }
\end{theorem}
\begin{proof}
	1. From Theorem \ref{protconcept algebra} it follows that  $(\mathfrak{P}(\mathbb{K}),\sqcup,\sqcap,\neg,\lrcorner,\top,\bot)$ is a dBa.  To show monotonicity of  $f_{R},f_{S}$, let  $(A, B), (C, D)\in \mathfrak{P}(\mathbb{K})$ and $(A,B)\sqsubseteq (C,D)$. Then, by definition of  $\sqsubseteq$, $A\subseteq C$ and $D\subseteq B$, and by using Proposition \ref{pra}(iv), $\underline{A}_{R}\subseteq \underline{C}_{R}$,  which  implies $(\underline{C}_{R})^{\prime}\subseteq (\underline{A}_{R})^{\prime}$. Hence $f_{R}((A,B))\sqsubseteq f_{R}((C,D))$. Similar to the above, we can show the monotonicity of $f_{S}$. Rest of the proof follows from Theorem \ref{oprator-proto}.
	\\
	2. From Theorem \ref{protconcept algebra}, it follows that $\underline{\mathfrak{P}}(\mathbb{K})_{p}=\underline{\mathfrak{H}}(\mathbb{K})$. By Theorem \ref{largestpuredBao}(2), $\underline{\mathfrak{P}}^{+}(\mathbb{KC})_{p}=\underline{\mathfrak{H}}^{+}(\mathbb{KC})$ is the largest pure subalgebra of $\underline{\mathfrak{P}}^{+}(\mathbb{KC})$.  
\end{proof}

The following lists some basic properties of the operators $\textbf{I},\textbf{C}$ and their duals in a dBao.
\begin{lemma}
	\label{dual operators}
	{\rm Let $\mathfrak{O}:=(D,\sqcup,\sqcap,\neg,\lrcorner,\top,\bot, \textbf{I},\textbf{C})$ be a dBao. Then  the following hold  for any $a,x,y\in D$.\\
		1. $\neg \textbf{I}^{\delta}(\neg a)= \textbf{I}a$ and $\lrcorner\textbf{C}^{\delta}(\lrcorner a)=\textbf{C}(a)$.\\
		2. $\textbf{I}(\neg a)=\neg \textbf{I}^{\delta}(a)$ and $\textbf{I}^{\delta}(\neg a)=\neg \textbf{I}(a)$.\\
		3. $\textbf{C}(\lrcorner a)=\lrcorner \textbf{C}^{\delta}(a)$ and $\textbf{C}^{\delta}(\lrcorner a)=\lrcorner \textbf{C}(a)$.\\
		4. $\textbf{I}^{\delta}$ and $\textbf{C}^{\delta}$ both are monotonic.\\
		5. $\textbf{I}^{\delta}(a\sqcap a)=\textbf{I}^{\delta}(a)$ and $\textbf{C}^{\delta}(a\sqcup a)=\textbf{C}^{\delta}(a)$.\\
		6. $\textbf{I}^{\delta}(x\vee y)=\textbf{I}^{\delta}(x)\vee \textbf{I}^{\delta}(y)$ and $\textbf{C}^{\delta}(x\wedge y)=\textbf{C}^{\delta}(x)\wedge \textbf{C}^{\delta}(y)$.\\
		7. $\textbf{I}^{\delta}(\bot)=\bot$ and $\textbf{C}^{\delta}(\top)=\top$.\\
		8. $\textbf{I}^{\delta}(x)\sqcap \textbf{I}^{\delta}(x)=\textbf{I}^{\delta}(x)$ and $\textbf{C}^{\delta}(x)\sqcup \textbf{C}^{\delta}(x)=\textbf{C}^{\delta}(x)$. }
\end{lemma}
\begin{proof}
The	proof is obtained  in a straightforward manner. We use 1, 2, 3 and 5 of Proposition \ref{pro2}, (8a), (8b) of Definition \ref{DBA} and $3a, 3b$ of  Definition \ref{DBA with operators}. 
\end{proof}
%**Move this proposition to Appendix also.** The proof  of the following result is straightforward and given in the Appendix.
%\begin{proposition}
%	\label{meet join}
%	{\rm  Let $\textbf{D}:=(D,\sqcup,\sqcap,\neg,\lrcorner,\top,\bot)$ be a dBa. For any  $x,y\in D$, the following hold.
%		\begin{enumerate}
%			\item $x\sqcap y\sqsubseteq x\vee y\sqsubseteq x\sqcup y$.
%			\item $x\sqcap y\sqsubseteq x\wedge y\sqsubseteq x\sqcup y$.
%	\end{enumerate}}
%\end{proposition}

We noted earlier that a Bao provides an example of a dBao. The converse question is addressed in Theorems \ref{gen of dBa with oper} and \ref{relation with Booleanalgebra with operators} below.
%\begin{proof}
%	 
%	of 2 is dual to the proof of 1., and 
%	We only  prove 1. For any $x,y\in D$, $\neg x\sqcap \neg y\sqsubseteq \neg x$ and $\neg x\sqcap \neg y\sqsubseteq \neg y.$ So by Proposition \ref{pro2}(2), $\neg\neg x\sqsubseteq \neg(\neg x\sqcap\neg y)$ and $\neg\neg y\sqsubseteq \neg(\neg x\sqcap\neg y).$ Then Proposition \ref{pro1.5}(6) gives $\neg\neg x\sqcap \neg\neg y\sqsubseteq \neg(\neg x\sqcap\neg y)\sqcap\neg\neg y$ and $\neg\neg y\sqcap \neg(\neg x\sqcap\neg y)\sqsubseteq\neg(\neg x\sqcap\neg y)\sqcap \neg(\neg x\sqcap\neg y).$ Therefore $\neg\neg x\sqcap \neg\neg y\sqsubseteq\neg(\neg x\sqcap\neg y)\sqcap \neg(\neg x\sqcap\neg y).$  By Proposition \ref{pro2}(1), $\neg\neg x\sqcap \neg\neg y\sqsubseteq\neg(\neg x\sqcap\neg y)$,  that is $(x\sqcap x)\sqcap (y\sqcap y)\sqsubseteq x\vee y.$ By axiom $(1a)$ and $(3a)$, $x\sqcap y\sqsubseteq x\vee y$.
%	
%	\noindent We know that $x,y\sqsubseteq x\sqcup y.$ Proposition \ref{pro2}(2) gives $\neg(x\sqcup y)\sqsubseteq \neg x,\neg y.$ Therefore by  Proposition \ref{pro1.5}(6), $\neg (x\sqcup y)\sqcap\neg y\sqsubseteq \neg x\sqcap\neg y$ and $\neg (x\sqcup y)\sqcap\neg (x\sqcup y)\sqsubseteq \neg (x\sqcup y)\sqcap\neg y.$ So $\neg (x\sqcup y)\sqcap \neg (x\sqcup y)\sqsubseteq \neg x\sqcap\neg y.$ By  Proposition \ref{pro2}(1),  $\neg(x\sqcup y)\sqsubseteq \neg x\sqcap\neg y,$ and by Proposition \ref{pro2}(2),  $\neg(\neg x\sqcap\neg y)\sqsubseteq \neg\neg (x\sqcup y)=(x\sqcup y)\sqcap (x\sqcup y)\sqsubseteq x\sqcup y.$ Hence $x\vee y\sqsubseteq x\sqcup y.$
%\end{proof}

\begin{theorem}
	\label{gen of dBa with oper}
	{\rm Let  $\mathfrak{O}:=(D,\sqcup,\sqcap,\neg,\lrcorner,\top,\bot,\textbf{I},\textbf{C})$ be a dBao such that for all $a\in D$ $\neg a=\lrcorner a$, $\neg\neg a=a$. Then 
		%the following  hold.
		%\begin{enumerate}
		%\item 
		$(D,\sqcup,\sqcap,\neg,\top,\bot,\textbf{C})$ and $(D,\sqcup,\sqcap,\neg,\top,\bot,\textbf{I}^{\delta})$  are Baos.
		% Boolean algebra with operators.
		%\item  $(D,\sqcup,\sqcap,\neg,\lrcorner,\top,\bot,\textbf{I}^{\delta})$ is a Boolean algebra with operators.
		%\end{enumerate
	} 
\end{theorem} 
\begin{proof}
	%let $\mathfrak{O}$ be a dBao such that for all $a\in D$, $\neg a=\lrcorner a$ and $\neg\neg a=a$. 
	That $(D,\sqcup,\sqcap,\neg,\top,\bot)$ forms a Boolean algebra is not difficult to prove, and the proof  is given in the Appendix. In particular, one can show that $y\sqcup z=y\vee z$ and $y\sqcap z=y\wedge z$ for any $y,z \in D$. It is then easy to verify that $\textbf{C}$  and $\textbf{I}^{\delta}$ are additive and normal. Indeed,
	Definition \ref{DBA with operators}(1b) implies that $\textbf{C}$ is  additive. As $\lrcorner \top= \bot$,  
	%$\textbf{C}(\bot)=\textbf{C}(\lrcorner \top)= \lrcorner\top$ 
	by  Definition \ref{DBA with operators}(3b), it is normal. On the other hand, 
	%which implies that
	%$\textbf{C}(\bot)=\bot$. So $\textbf{C}$ is normal, which implies that $(D,\sqcup,\sqcap,\neg,\lrcorner,\top,\bot,\textbf{C})$ is a Boolean algebra with operators.\\
	as $y\sqcup z=y\vee z$ for all $y,z \in D$, from Lemma \ref{dual operators}(6) it follows that $\textbf{I}^{\delta}(x\sqcup y)=\textbf{I}^{\delta}(x\vee y)=\textbf{I}^{\delta}(x)\vee \textbf{I}^{\delta}(y)=\textbf{I}^{\delta}(x)\sqcup \textbf{I}^{\delta}(y)$. $\textbf{I}^{\delta}(\bot)=\bot$ by Lemma \ref{dual operators}(7).  
	%So  $(D,\sqcup,\sqcap,\neg,\lrcorner,\top,\bot,\textbf{I}^{\delta})$ is a Boolean algebra with operators.
\end{proof}

%Now, we prove some structural results regarding  the dBao.

\begin{theorem}
	\label{relation with Booleanalgebra with operators}
	{\rm Let $\mathfrak{O}:=(D,\sqcup,\sqcap,\neg,\lrcorner,\top,\bot, \textbf{I},\textbf{C})$ be a dBao. Then $\mathfrak{O}_{\sqcap}:=(D_{\sqcap},\sqcap,\vee,\neg,\bot, \textbf{I}^{\delta}\restriction{D_{\sqcap}})$  and $\mathfrak{O}_{\sqcup}:=(D_{\sqcup},\sqcup,\wedge,\lrcorner,\top, \textbf{C}\restriction{D_{\sqcup}})$ are Baos. }
\end{theorem}

\begin{proof}
	%Let $\mathfrak{O}$ be a dBao. Then 
	By Proposition \ref{pro1}, $\textbf{D}_{\sqcap}$ and $\textbf{D}_{\sqcup}$ are Boolean algebras. Let $x\in D_{\sqcap}$. Then   $\textbf{I}^{\delta}\restriction{D_{\sqcap}}(x)\sqcap \textbf{I}^{\delta}\restriction{D_{\sqcap}}(x)=\textbf{I}^{\delta}(x)\sqcap \textbf{I}^{\delta}(x)=\textbf{I}^{\delta}(x)=\textbf{I}^{\delta}\restriction{D_{\sqcap}}(x)$, by Lemma \ref{dual operators}(8).  So $D_{\sqcap}$ is closed under $\textbf{I}^{\delta}\restriction{D_{\sqcap}}$.
	Similarly, $D_{\sqcup}$ is closed under $\textbf{C}\restriction{D_{\sqcup}}$. That both $\textbf{I}^{\delta}\restriction{D_{\sqcap}}$ and $\textbf{C}\restriction{D_{\sqcup}}$ are additive and normal follows from Lemma \ref{dual operators}(6,7) and Definition \ref{DBA with operators}.
	 	%  it follows that $\textbf{C}$ is a normal additive operator on $D_{\sqcup}$ and  $\textbf{I}^{\delta}\restriction{D_{\sqcap}}$ is a normal additive operator on $D_{\sqcap}$. Hence $\mathfrak{O}_{\sqcap}$ and $\mathfrak{O}_{\sqcup}$ are Baos.\\
\end{proof}
The following result addresses the converse of Theorem \ref{relation with Booleanalgebra with operators}.
\begin{theorem}
	{\rm Let $\textbf{D}:=(D,\sqcup,\sqcap,\neg,\lrcorner,\top,\bot,) $ be a dBa such that $\mathfrak{O}_{\sqcap}:=(D_{\sqcap},\sqcap,\vee,\neg,\bot, \overline{\textbf{I}})$  and $\mathfrak{O}_{\sqcup}:=(D_{\sqcup},\sqcup,\wedge,\lrcorner,\top, \overline{\textbf{C}})$ are Baos. Then $\mathfrak{O}:= (D,\sqcup,\sqcap,\\\neg,\lrcorner,\top,\bot, \textbf{I},\textbf{C})$ is a dBao, where $\textbf{I}(x):=\neg \overline{\textbf{I}}(\neg x)$ and $\textbf{C}(x):= \overline{\textbf{C}}(x\sqcup x)$ for all $x\in D$. }
\end{theorem}
\begin{proof}
	%Let $\textbf{C}(x\sqcup x)=\textbf{C}(x)$ for all $x\in D$ and  $\mathfrak{O}_{\sqcap}:=(D_{\sqcap},\sqcap,\vee,\neg,\bot, \textbf{I}^{\delta}\restriction{D_{\sqcap}})$  and $\mathfrak{O}_{\sqcup}:=(D_{\sqcup},\sqcup,\wedge,\lrcorner,\top, \textbf{C}\restriction{D_{\sqcup}})$ are two Boolean algebra with operators. 
	Let $x,y\in D$. Using  Proposition \ref{pro2}(6), $\textbf{I}(x\sqcap y)=\neg \overline{\textbf{I}}(\neg(x\sqcap y))=\neg \overline{\textbf{I}}(\neg x \vee \neg y)=\neg(\overline{\textbf{I}}(\neg x)\vee\overline{\textbf{I}}(\neg y))$, as $\neg x,\neg y\in D_{\sqcap}$ by Proposition \ref{pro2}(1). As $\overline{\textbf{I}}(\neg x), \overline{\textbf{I}}(\neg y) \in D_{\sqcap}$, using definition of $\vee$ we have  $ \textbf{I}(x\sqcap y)=\neg \overline{\textbf{I}}(\neg x)\sqcap \neg \overline{\textbf{I}}(\neg y)=\textbf{I}(x)\sqcap \textbf{I}(y)$. Using  Proposition \ref{pro2}(5), $\textbf{I}(\neg \bot)=\neg \overline{\textbf{I}} (\neg\neg \bot)=\neg \overline{\textbf{I}}(\bot)=\neg \bot $. By Definition \ref{DBA}(4a), $\textbf{I}(x\sqcap x)=\neg \overline{\textbf{I}}(\neg (x\sqcap x))=\neg\overline{\textbf{I}}(\neg x)=\textbf{I}(x).$
	
	\noindent $\textbf{C}(\lrcorner\top)=\overline{\textbf{C}}(\lrcorner\top\sqcup\lrcorner \top)=\overline{\textbf{C}}(\lrcorner \top)=\lrcorner \top$, as $\top\in D_{\sqcup}$. That  $\textbf{C}(x\sqcup x)= \textbf{C}(x)$ is immediate from Definition \ref{DBA}. Finally, one shows that $\textbf{C}(x\sqcup y)=\textbf{C}(x)\sqcup \textbf{C}(y)$ for all $x, y\in D$. Let $x, y\in D$. Using commutativity and associativity of $\sqcup$ and Definition \ref{DBA}(1b), additivity of $\overline{\textbf{C}}$ and the fact that $x\sqcup x, y \sqcup y \in D_{\sqcup}$, we have the following equalities. $\textbf{C}(x\sqcup y)=\overline{\textbf{C}}((x\sqcup y)\sqcup (x\sqcup y))=\overline{\textbf{C}}((x\sqcup x)\sqcup (y\sqcup y))=\overline{\textbf{C}}(x\sqcup x)\sqcup \overline{\textbf{C}}(y\sqcup y)=\textbf{C}(x)\sqcup \textbf{C}(y)$. So $\mathfrak{O}$ is a dBao. 
	\end{proof}

%\begin{proof}
%Proof follows from the Theorem \ref{generalization of boolean} and Definition \ref{DBA with operators}.
%\end{proof}
%Therefore we can say that a dBa with operators differ from a Boolean algebra by its negations, interior and closure operator. So if $\textbf{D}$ be a double Boolean algebra with operator such that for all $a\in \textbf{D}$ $a\in \textbf{D}$ $\neg a=\lrcorner a$  $\neg\neg a=a$ and $C^{\delta}(a)=I(a)$ then $\textbf{D}=(D,\sqcup,\sqcap,\neg,\lrcorner,\top,\bot,\textbf{C})$ is a Boolean algebra with operator.

We end this part by noting a close connection between the full complex algebra of a Kripke frame and that of a corresponding Kripke context. Let $(W,R)$ be a Kripke  frame and $\mathfrak{F}^{+}:=(\mathcal{P}(W),\cap,\cup,^c,W,\emptyset,m_{R})$ be the full complex algebra \cite{blackburn2002moda}, where for all $A\in \mathcal{P}(W)$, $m_{R}(A):=\{w\in W:R(w)\cap A\neq\emptyset\}=\overline{A}^{R}$. This is a Bao, and as observed earlier, yields the dBao 
%Now if we set  $\textbf{C}:=m_{R}$,  $\textbf{I}:=m^{\delta}_{R}$ and both negations $\neg$, $\lrcorner$ as set complement $^c$,  then  
$(\mathcal{P}(W),\cap,\cup,^c,W,\emptyset,m^{\delta}_{R}, m_{R} )$.  For the Kripke  frame $(W,R)$, let us define the Kripke context $\mathbb{KC}_{0}:=((W,R),(W,R),\neq)$. By Definition \ref{full-complexalg}, we have the full complex algebra of $\mathbb{KC}_{0}$ as $\underline{\mathfrak{P}}^{+}(\mathbb{KC}_{0}):=(\mathfrak{P}(\mathbb{K}), \sqcup,\sqcap, \neg, \lrcorner, \top,\bot, f_{1}, f_{2})$, where $f_{1}((A, B)):=(\underline{A}_{R}, (\underline{A}_{R})^{\prime})$, $f_{2}((A, B)):=((\underline{B}_{R})^{\prime}, \underline{B}_{R})$ for all $(A,B)\in \mathfrak{P}(\mathbb{K})$. Then we get
%The  full complex algebra of $\mathbb{KC}_{0}$, $\mathfrak{P}^{+}(\mathbb{KC}_{0}):=(\mathfrak{P}(\mathbb{K}), \sqcup,\sqcap, \neg, \lrcorner, \top,\bot, f_{1}, f_{2})$, where $f_{1}((A, B)):=(\underline{A}_{R}, (\underline{A}_{R})^{\prime})$, $f_{2}((A, B)):=((\underline{B}_{R})^{\prime}, \underline{B}_{R})$ for all $(A,B)\in \mathfrak{P}(\mathbb{K})$. Then, we get the following.
% Theorem \ref{reltion btw complex algb}   below shows that  the Bao $(\mathfrak{P}(\mathbb{K}), \sqcup,\sqcap, \neg, \top,\bot, f_{2})$ induced by the full complex algebra $\mathfrak{P}^{+}(\mathbb{KC}_{0})$ is isomorphic to $\mathfrak{F}^{+}$. Moreover, $f_{1}$ is the dual of $f_{2}$. 
\begin{theorem}
	\label{reltion btw complex algb}
	{\rm For the full complex algebra $\underline{\mathfrak{P}}^{+}(\mathbb{KC}_{0})$, the following hold.
		\begin{enumerate}
			\item $\neg x=\lrcorner x$, $\neg\neg x=x$ and  $f_{1}(x)=\neg f_{2}(\neg x)$ for all $x\in \mathfrak{P}(\mathbb{K}) $. %$(\mathfrak{P}(\mathbb{K}), \sqcup,\sqcap, \neg, \top,\bot)$ is a Boolean algebra.
			\item  $(\mathfrak{P}(\mathbb{K}), \sqcup,\sqcap, \neg, \top,\bot, f_{2})$ is a Bao, which is  isomorphic to $\mathfrak{F}^{+}$.  
		\end{enumerate}
	%	corresponding to the Kripke  frame $(W,R)$, there is a Kripke context  $\mathbb{KC}_{0}:=((W,R),(W,R),\neq)$ such thatLet $\mathfrak{F}:=(W,R)$ be a Kripke frame. Then  to .
}
\end{theorem}
\begin{proof}
	1. Let $A\subseteq W$ and $x\in A^{c}$. Then for all $a\in A$, $x\neq a$, which implies that $x\in A^{\prime}$.  Now let $x\in A^{\prime}$. Then $x\neq a$, for all $a\in A$, which implies that $x\in A^{c}$.
	So $A^{\prime}=A^{c}$, and 
	$A^{\prime\prime}=A^{c\prime}=A^{cc}=A$. Therefore  $(A,B)\in \mathfrak{P}(\mathbb{K}) $ if and only if $A=B^{c}$, which is equivalent to $A^{c}=B$. \\Let $(A,A^{c})\in \mathfrak{P}(\mathbb{K})$. Then $\neg(A,A^{c})=(A^{c},A)=\lrcorner(A,A^{c})$ and $\neg\neg (A,A^{c})=(A,A^{c})$. \\
	%Let $(A,A^{c})\in  \mathfrak{P}(\mathbb{K}) $. \\
$f_{2}((A,A^{c})):=(\underline{(A^{c})}_{R}^{\prime},\underline{(A^{c})}_{R})$, giving \\$\neg f_{2}(\neg(A,A^{c}))=
	%\neg f_{2}((A^{c},A^{c\prime}))=
	\neg f_{2}((A^{c},A))=\neg((\underline{A}_{R})^{\prime},\underline{A}_{R})= ((\underline{A}_{R})^{\prime c}, (\underline{A}_{R})^{\prime c \prime})=(\underline{A}_{R}, (\underline{A}_{R})^{\prime})$. So $f_{1}((A,A^{c}))=(\underline{A}_{R},(\underline{A}_{R})^{\prime})=\neg f_{2}(\neg(A,A^{c}))$.\\
2.	By  Theorem \ref{gen of dBa with oper} it follows that $(\mathfrak{P}(\mathbb{K}), \sqcup,\sqcap, \neg, \top,\bot, f_{2})$ is a Bao. 

	\noindent	Let us define a map $f$ from  $\mathcal{P}(W)$  to $\mathfrak{P}(\mathbb{K})$ by $f(A):=(A,A^{c})$ for all $A\subseteq W$. It is clear that $f$ is  well-defined. To show $f$ is a homomorphism, let $A,B\subseteq W$. $f(A\cap B)=(A\cap B, (A\cap B)^{c})=(A,A^{c})\sqcap (B,B^{c})=f(A)\sqcap f(B)$ and $f(A\cup B)=(A\cup B, (A\cup B)^{c})=(A,A^{c})\sqcup (B,B^{c})=f(A)\sqcup f(B)$. $f(A^{c})=(A^{c}, A)=\neg f(A)=\lrcorner f(A)$ and $f(W)=(W,\emptyset)=\top$ $f(\emptyset)=(\emptyset, W)=\bot$. $f(m_{R}(A))=(\overline{A}^{R}, (\overline{A}^{R})^{c})=((\underline{A^{c}}_{_{R}})^{c},\underline{A^{c}}_{_{R}})=f_{2}((A,A^{c}))=f_{2}(f(A))$. \\ Injectivity and surjectivity  of $f$ follow trivially.   
	%Hence proved.
\end{proof}

\noindent From Theorem \ref{reltion btw complex algb}, we may conclude that the dBao $\underline{\mathfrak{P}}^{+}(\mathbb{KC}_{0})$ is identifiable with the Bao $\mathfrak{F}^{+}$.
%So, the dBao $\mathfrak{P}^{+}(\mathbb{KC}_{0})$  can be identified with the Bao $\mathfrak{F}^{+}$.

%\begin{definition}
%{\rm Let $\mathfrak{M}$ and $\mathfrak{N}$ be two dBao. A map $f$ from $\mathfrak{M}$ to $\mathfrak{N}$ is a dBao homomorphism if it is a dBa homomorphism and  satisfying the following conditions, $f(I(a))=I(f(a))$ and $f(C(a))=C(f(a))$. An isomorphism is bijective homomorphism.}
%\end{definition}It follows from the Theorem \ref{relation with Booleanalgebra with operators} that $\underline{\mathfrak{P}}^{+}(\mathbb{KC})_{\sqcap}:=(\mathfrak{P}(\mathbb{K})_{\sqcap},\sqcap,\vee,\neg,\bot,f^{\delta}_{R})$ and $\underline{\mathfrak{P}}^{+}(\mathbb{KC})_{\sqcup}:=(\mathfrak{P}(\mathbb{K})_{\sqcup},\sqcup,\wedge,\lrcorner,\top,f_{S})$ are Bao. 

\subsubsection{Representation theorems for dBaos}
\label{RtdBaos}
%The representation theorems for dBao are  proved using  Wille's representation theorems for dBa. 

For every dBao $\mathfrak{O}:=(D,\sqcup,\sqcap,\neg,\lrcorner,\top,\bot, \textbf{I}, \textbf{C})$, we construct a Kripke context based on the standard context  $\mathbb{K}(\textbf{D}):=(\mathcal{F}_{p}(\textbf{D}),\mathcal{I}_{p}(\textbf{D}), \Delta)$ corresponding to the underlying dBa \textbf{D}. For that,  relations $R$ and $S$ are defined  on $\mathcal{F}_{p}(\textbf{D})$ and  $\mathcal{I}_{p}(\textbf{D})$ respectively
%. The relations $R$ and $S$ are defined 
as follows.\\
For all $u,u_{1}\in \mathcal{F}_{p}(\textbf{D})$, $uRu_{1}$ if and only if $\textbf{I}^{\delta}(a)\in u$ for all $a\in u_{1}$.\\
For all $v,v_{1}\in \mathcal{I}_{p}(\textbf{D})$, $vSv_{1}$ if and only if $\textbf{C}^{\delta}(a)\in v$ for all $a\in v_{1}$.
\vskip 2pt
\noindent 
%To prove representation theorems for dBao, we will use 
The following results are required to get (Representation) Theorem \ref{rtdBao}.
\begin{lemma}
	\label{req}
	{\rm 
		%For any \textbf{D} dBa, 		
		%\begin{enumerate}
		%\item 
		If $F$ is a primary filter (ideal) of a dBa \textbf{D}, then for any $x\in D$,  exactly one of the elements $x$ and $\neg x$ belongs to $F$.
		%\item If $I$ is a primary ideal then for all $x\in D$, exactly one of the elements $x$ and $\lrcorner x$ belongs to $I$.
		%\end{enumerate}
	}
\end{lemma}
\begin{proof}
	Proof follows from the definition of a primary filter (ideal).  
\end{proof}
\begin{lemma}
	\label{canonical relations}
	{\rm Let $\mathfrak{O}:=(D,\sqcup,\sqcap,\neg,\lrcorner,\top,\bot, \textbf{I}, \textbf{C})$ be a dBao. The following hold.
		\begin{enumerate}
			\item  For all $u,u_{1}\in \mathcal{F}_{p}(\textbf{D})$, $uRu_{1}$ if and only if for all $a\in D$, $\textbf{I}a\in u$ implies that $a\in u_{1}$.
			\item  For all $v,v_{1}\in \mathcal{I}_{p}(\textbf{D})$, $vSv_{1}$ if and only if for all $a\in D$, $\textbf{C}a\in v$ implies that $a\in v_{1}$.
	\end{enumerate}}
\end{lemma}
\begin{proof}
	1.  For all $a\in D$, suppose  $\textbf{I}a\in u$ implies that $a\in u_{1}$. If possible, assume $u\cancel{ R}u_{1}$. Then there exists $a_{1}\in u_{1}$ such that $\textbf{I}^{\delta}(a_{1})\notin u$. So $\neg \textbf{I}^{\delta}(a_{1})\in u$, which implies that $\textbf{I}(\neg a_{1})\in u$  by  Lemma \ref{dual operators}(2). As $a_{1}\in u_{1}$, $\neg a_{1}\notin u_{1}$, which contradicts that $\textbf{I}(\neg a_{1})\in u$. Hence $uRu_{1}$.
	
	Now, we assume that $uRu_{1}$ and let $a_{1}\in D$ such that $\textbf{I}a_{1}\in u$. If possible, suppose $a_{1}\notin u_{1}$. Then $\neg a_{1}\in u_{1}$. So $\textbf{I}^{\delta}(\neg a_{1})\in u$ as $uRu_{1}$.  Therefore by Lemma \ref{dual operators}, $\textbf{I}^{\delta}(\neg a_{1})=\neg \textbf{I}(a_{1})\in u$, which is a contradiction by Lemma \ref{req}. Hence $a_{1}\in u_{1}$.
	
	\noindent Proof of 2 is similar to the above.  
\end{proof}
%**Why proposition and not lemma?**
\begin{lemma}
	\label{botom neg and ordered}
	{\rm Let $\textbf{D}:=(D,\sqcup,\sqcap,\neg,\lrcorner,\top,\bot)$ be a dBa.  For all $a,b\in D$, the following hold.
		\begin{enumerate}
			\item If $a\sqcap b=\bot$ then $a\sqcap a\sqsubseteq \neg b$.
			\item If $a\sqcap a\sqsubseteq \neg b$ then $a\sqcap b\sqsubseteq \bot$.
			\item If $a\sqcup b=\top$ then $\lrcorner b\sqsubseteq a\sqcup a$.
			\item If $\lrcorner b\sqsubseteq a\sqcup a$ then $\top\sqsubseteq a\sqcup b$.
		\end{enumerate}
		In particular, if $\textbf{D}$ is a contextual dBa then $a\sqcap b=\bot$ if and only if $a\sqcap a\sqsubseteq \neg b$, and  $a\sqcup b=\top$ if and only if $\lrcorner b\sqsubseteq a\sqcup a$.}
\end{lemma}
\begin{proof}
	1. Let $a,b\in D$ and $a\sqcap b=\bot$. Then by Definition \ref{DBA}(1a) and the associative law, $\bot= (a\sqcap a)\sqcap (b\sqcap b)$. So $\bot\vee\neg (b\sqcap b)=((a\sqcap a)\sqcap (b\sqcap b))\vee\neg(b\sqcap b)$. By  Definition \ref{DBA}(6a),  $\bot\vee\neg (b\sqcap b)=((a\sqcap a)\vee\neg(b\sqcap b))\sqcap ((b\sqcap b)\vee\neg(b\sqcap b))$. Now $(a\sqcap a)\vee\neg(b\sqcap b)=\neg(\neg(a\sqcap a)\sqcap\neg\neg(b\sqcap b))=\neg (\neg a\sqcap (b\sqcap b))$ by Definition \ref{DBA}(4a) and  Proposition \ref{pro2}(3). So $(a\sqcap a)\vee\neg(b\sqcap b)=\neg(\neg a\sqcap b)$ by  Definition \ref{DBA}(1a). Similarly, we can show that $\bot\vee\neg(b\sqcap b )=\neg(b\sqcap \neg\bot)$.  Therefore $\bot\vee\neg(b\sqcap b )=\neg(b\sqcap \neg\bot)=\neg(b\sqcap(\top\sqcap\top))$ by Definition \ref{DBA}(10a). Using Definition \ref{DBA}(1a) and Proposition \ref{pro1.5}(2), 
	$\bot\vee\neg(b\sqcap b )=\neg(b\sqcap\top)=\neg(b\sqcap b)=\neg b$, where the  last equality follows from Definition \ref{DBA}(4a). This implies that $\neg b=\neg(\neg a\sqcap b)\sqcap\neg\bot=\neg(\neg a\sqcap b)$, as $\neg(\neg a\sqcap b),b\sqcap b, \neg\bot\in D_{\sqcap}$. $\neg\neg a\sqsubseteq \neg(\neg a\sqcap b)$, as $\neg a\sqcap b\sqsubseteq \neg a $. So $a\sqcap a\sqsubseteq \neg(\neg a\sqcap b)=\neg b$.\\
	2. Let $a\sqcap a\sqsubseteq \neg b$.  Then $a\sqcap a\sqcap b\sqsubseteq \neg b\sqcap b$ by Proposition \ref{pro1.5}(6) and by Definition \ref{DBA}(1a), $a\sqcap b\sqsubseteq \bot$.\\
	Now if $\textbf{D}$ is a contextual dBa  then $\sqsubseteq$ becomes a partial order. Therefore from the above it follows that $a\sqcap b=\bot$ if and only if $a\sqcap a\sqsubseteq \neg b$.\\
	The other parts can be proved dually.  
\end{proof}

%\begin{proposition}
%	\label{top neg and ordered}
%	{\rm 
%	%Let $\textbf{D}:=(D,\sqcup,\sqcap,\neg,\lrcorner,\top,\bot)$ be dBa. Then f
%	For all $a,b\in D$ the following hold.
%		\begin{enumerate}
%			\item If $a\sqcup b=\top$ then $\lrcorner b\sqsubseteq a\sqcup a$.
%			\item If $\lrcorner b\sqsubseteq a\sqcup a$ then $\top\sqsubseteq a\sqcup b$.
%		\end{enumerate}
%		In particular if $\textbf{D}$ be a contextual dBa then $a\sqcup b=\top$ if and only if $\lrcorner b\sqsubseteq a\sqcup a$.}
%\end{proposition}
%\begin{proof}
%	Proof is dual to the proof of Lemma \ref{botom neg and ordered}
%\end{proof}
\begin{lemma}
	\label{canonical box and dimon} 
	{\rm Let $\mathfrak{O}$ be a dBao and $\mathbb{KC}(\mathfrak{O}):=((\mathcal{F}_{p}(\textbf{D}),R), (\mathcal{I}_{p}(\textbf{D}),S),\Delta)$. Then for all $a\in D$ the following hold. 
		\begin{enumerate}
			\item  $\overline{F_{a}}^{R}=F_{\textbf{I}^{\delta}(a)}$ and $\underline{F_{a}}_{R}=F_{\textbf{I}(a)}$.
			\item  $\overline{I_{a}}^{S}=I_{\textbf{C}^{\delta}(a)}$ and $\underline{I_{a}}_{S}=I_{\textbf{C}(a)}$.
	\end{enumerate}}
\end{lemma}
\begin{proof}
	1. Let $F\in \overline{F_{a}}^{R}$. Then there exists $F_{1}\in F_{a}$ such that $FRF_{1}$, which implies that $\textbf{I}^{\delta}(a)\in F$, as $a\in F_{1}$. So $\overline{F_{a}}^{R}\subseteq F_{\textbf{I}^{\delta}(a)}$.\\
	Let $F\in F_{\textbf{I}^{\delta}(a)}$ and we show that $F\in\overline{F_{a}}^{R}$. We must then find a primary filter $F_{1}\in F_{a}$ such that $FRF_{1}$. Let $F_{0}:=\{x\in D: \textbf{I}x\in F\}$ and $F_{01}:=\{x\sqcap a: x\in F_{0}\}$. Then $F_{01}$ is closed under $\sqcap$ and $F_{01}\subseteq D_{\sqcap}$. Next we show that $\bot \notin F_{01}$. If possible, suppose $\bot\in F_{01}$. Then there exists  $x_{1}\in F_{0}$ such that $x_{1}\sqcap a=\bot$, which implies that $a\sqcap a\sqsubseteq \neg x_{1}$  by  Lemma \ref{botom neg and ordered}(1). So $\textbf{I}^{\delta}(a\sqcap a)\sqsubseteq \textbf{I}^{\delta}(\neg x_{1})$, whence $\textbf{I}^{\delta}(a)\sqsubseteq \textbf{I}^{\delta}(\neg x_{1})$ by  Lemma \ref{dual operators}(4,5).  $\textbf{I}^{\delta}(\neg x_{1})\in F$, as $\textbf{I}^{\delta}(a)\in F$ and $F$ is a filter, which implies that $\neg \textbf{I}( x_{1})\in F$. So $\textbf{I}(x_{1})\notin F$ which contradicts that $x_{1}\in F_{0}$. Therefore $\bot\notin F_{01}$. Since $\textbf{D}_{\sqcap}$ is a Boolean algebra and $F_{01}\subseteq D_{\sqcap}$, there exists a prime filter $F_{2}$ containing $F_{01}$. So $F_{3}:=\{x\in D: y\sqsubseteq x~\mbox{for some}~ y\in F_{2}\}$ is a primary filter  containing  $F_{2}$ by Lemma \ref{lema1} and Proposition \ref{compar-ideal}.  For all $x\in F_{0}$, $x\sqcap a\in F_{01}\subseteq F_{2}$ and $x\sqcap a\sqsubseteq x, x\sqcap a\sqsubseteq a$, which implies that $F_{0}\subseteq F_{3}$ and  $a\in F_{3}$. By Lemma \ref{canonical relations}(1) it follows that $FRF_{3}$. Therefore $F\in\overline{F_{a}}^{R}$. \\
	%Hence $\overline{F_{a}}^{R}= F_{\textbf{I}^{\delta}(a)}$.\\
	Using Proposition \ref{pra}(i), Lemmas \ref{complement of Fx} and  \ref{dual operators}(1), we get \\
	$ \underline{F_{a}}_{R}=(\overline{(F_{a}^{c})}^{R})^{c}=(\overline{(F_{\neg a})}^{R})^{c}=F_{\textbf{I}^{\delta}(\neg a)}^{c}=F_{\neg \textbf{I}^{\delta}(\neg a) }= F_{\textbf{I}(a)}$.\\
	2 can be proved dually.  
\end{proof}

The Kripke context $\mathbb{KC}(\mathfrak{O})$ of Lemma \ref{canonical box and dimon} is used to obtain the representation theorem.
%we show that $\mathfrak{O}$ is isomorphic to complex algebra of the Kripke context $\mathbb{KC}(\mathfrak{O})=((\mathcal{F}_{p}(\textbf{D}),R),(\mathcal{I}_{p}(\textbf{D}), S), \Delta)$ that is we prove the following representation theorem.
\begin{theorem}[Representation theorem]
	\label{rtdBao}
	{\rm Let $\mathfrak{O}:=(D,\sqcup,\sqcap,\neg,\lrcorner,\top,\bot, \textbf{I}, \textbf{C})$ be a dBao. The following hold.
		\begin{enumerate}
			\item $\mathfrak{O}$ is quasi-embeddable into the full complex algebra $\underline{\mathfrak{P}}^{+}(\mathbb{KC}(\mathfrak{O}))$ of the Kripke context $\mathbb{KC}(\mathfrak{O})$.  $h:D\rightarrow \mathfrak{P}(\mathbb{K}(\textbf{D}))$ defined by $h(x):=(F_{x},I_{x})$ for all $x \in D$,  is the required quasi-embedding.
			\item If  $\mathfrak{O}$ is a contextual dBao then the quasi-embedding $h$ is an embedding.
			\item  $\mathfrak{O}_{p}$ is embeddable into the largest pure subalgebra $\underline{\mathfrak{H}}^{+}(\mathbb{KC}(\mathfrak{O}))$ of $\underline{\mathfrak{P}}^{+}(\mathbb{KC}(\mathfrak{O}))$.
	\end{enumerate}}
\end{theorem}
\begin{proof}
	1. 
	%Let $\mathfrak{O}:=(D,\sqcup,\sqcap,\neg,\lrcorner,\top,\bot, \textbf{I}, \textbf{C})$ be a dBao, where 
	Let $\textbf{D}:=(D,\sqcup,\sqcap,\neg,\lrcorner,\top,\bot)$ be the underlying  dBa. 
	%Now consider the  Kripke context $\mathbb{KC}(\mathfrak{O}):=((\mathcal{F}_{p}(\textbf{D}), R), (\mathcal{I}_{p}(\textbf{D}), S), \Delta)$, where the  underlying context is the **standard context $\mathbb{K}(\textbf{D}):=(\mathcal{F}_{p}(\textbf{D}),\mathcal{I}_{p}(\textbf{D}),\Delta)$. 
	By Theorem \ref{protoembedding}, we know that the map $h:D\rightarrow \mathfrak{P}(\mathbb{K}(\textbf{D}))$ defined by $h(x):=(F_{x},I_{x})$ for all $x \in D$ is a  quasi-embedding. To show $h$ is a dBao homomorphism, we prove that for any  $x \in D$, $h(\textbf{I}x)=f_{R}(h(x))$ and $h(\textbf{C}x)=f_{S}(h(x))$, that is, $(F_{\textbf{I}x}, I_{\textbf{I}x})=(\underline{F_{x}}_{R},(\underline{F_{x}}_{R})^{\prime})$ and $(F_{\textbf{C}x}, I_{\textbf{C}x})=((\underline{I_{x}}_{S})^{\prime}, \underline{I_{x}}_{S})$. By Lemma \ref{canonical box and dimon}(1), $\underline{F_{x}}_{R}=F_{\textbf{I}x}$.  By Lemma \ref{derivation}, 
	%$(\underline{F_{x}}_{R})^{\prime}=
	$F_{\textbf{I}x}^{\prime}=I_{\textbf{I}x_{\sqcap\sqcup}}=I_{(\textbf{I}x\sqcap \textbf{I}x)\sqcup (\textbf{I}x\sqcap \textbf{I}x)}=I_{\textbf{I}x\sqcup \textbf{I}x}=I_{\textbf{I}x}$, the last two equalities hold, as $\textbf{I}x\sqcap \textbf{I}x=\textbf{I}(x\sqcap x)=\textbf{I}x$ and  by Lemma \ref{complement of Fx}(1). So $(\underline{F_{x}}_{R})^{\prime}=I_{\textbf{I}x}$.  \\
	%which implies that $(F_{\textbf{I}x}, I_{\textbf{I}x})=(\underline{F_{x}}_{R},(\underline{F_{x}}_{R})^{\prime})$.\\
	Similar to the above, using Lemma \ref{canonical box and dimon}(2) and Lemma \ref{derivation}, we can show that  $(F_{\textbf{C}x}, I_{\textbf{C}x})=((\underline{I_{x}}_{S})^{\prime}, \underline{I_{x}}_{S})$. Hence $h$ is the required quasi-embedding  from the dBao $\mathfrak{O}$ into $\underline{\mathfrak{P}}^{+}(\mathbb{KC}(\mathfrak{O}))$ .\\
	2. Since $\mathfrak{O}$ is contextual, the quasi-order is a partial order. As a result, $h$ becomes injective.\\
	%** Follows from ** rest unnecessary **If $\mathfrak{O}$ is a contextual dBao then the quasi order becomes a partial order.   $h$ becomes  injective dBao homomorphism in this case.\\
	3. Let $x\in D_{p}$. Then either $x\sqcap x=x$ or $x\sqcup x=x$. If $x\sqcap x=x$, $h(x)=(F_{x}, I_{x})= (F_{x}, F_{x}^{\prime})$, by  Lemmas \ref{complement of Fx} and  \ref{derivation}. Now  if $x\sqcup x=x$, $h(x)=(F_{x}, I_{x})= (I_{x}^{\prime}, I_{x})$, by Lemmas \ref{complement of Fx} and \ref{derivation}. So  $h\restriction{D_{p}}$  is an injective dBao homomorphism  from $\mathfrak{O}_{p}$ to  $\underline{\mathfrak{H}}^{+}(\mathbb{KC}(\mathfrak{O}))$, as $\mathfrak{O}_{p}$ is pure and by Proposition \ref{order pure}.  
	% By Theorem \ref{largestpuredBao}(ii) and from the above it follows that restriction of $h$ on  $\mathfrak{O}_{p}$ is an embedding from $\mathfrak{O}_{p}$ into  $\underline{\mathfrak{H}}^{+}(\mathbb{KC}(\mathfrak{O}))$.
\end{proof}

\begin{corollary}
	\label{rtfpdBa}
	{\rm Let $\mathfrak{O}$ be a pure dBao. Then $\mathfrak{O}$ is embeddable into the complex algebra $\underline{\mathfrak{H}}^{+}(\mathbb{KC}(\mathfrak{O}))$ of the Kripke context $\mathbb{KC}(\mathfrak{O})$. }
\end{corollary}
\begin{proof}
	Proof follows from Theorems \ref{largestpuredBao}(2) and  \ref{rtdBao}(3).  
\end{proof}

\subsection {Topological double Boolean algebras}

\label{TDBA}

\begin{definition}
	\label{tdBa}
	{\rm A dBao  $\mathfrak{O}:=(D,\sqcap,\sqcup,\neg,\lrcorner,\top,\bot,\textbf{I},\textbf{C})$ is called a {\it topological dBa} if the following hold.
		\vskip 2pt
		$\begin{array}{ll}
			4a~  \textbf{I}(x) \sqsubseteq x &  
			4b~ x\sqsubseteq \textbf{C}(x ) \\
			5a~ \textbf{I}\textbf{I}(x)= \textbf{I}(x)&
			5b~ \textbf{C}\textbf{C}(x)=\textbf{C}(x) 	
		\end{array}$
		\vskip 2pt
\noindent 		A {\it topological contextual   dBa}  is a topological dBa in which the underlying dBa is contextual. If the underlying dBa is pure,  the topological dBa is called a {\it   topological pure dBa}.}
\end{definition}
%Let $P$ be a property and $\mathbb{KC}=((W,R),(W,S), I)$ be a Kripke context. We say $\mathbb{KC}$ has the property $P$ form left if $R$ has the property $P$ and we say  $\mathbb{KC}$ has the property $P$ form right if $S$ has the property $P$.

Again, as intended, we obtain a class of examples of topological dBas from the sets of protoconcepts and semiconcepts of contexts.

\begin{theorem}
	\label{pure complex algebra]}
	{\rm Let $\mathbb{KC}:=((G,R),(M,S),I)$ be a reflexive and transitive Kripke context. Then the following hold.
		\begin{enumerate}
			\item $\underline{\mathfrak{P}}^{+}(\mathbb{KC})$  is a  topological contextual dBa.
			\item $\underline{\mathfrak{P}}^{+}(\mathbb{KC})_{p}=\underline{\mathfrak{H}}^{+}(\mathbb{KC})$  is a topological pure dBa.
	\end{enumerate} }
\end{theorem}
\begin{proof}
	1. Proof follows from	Theorems \ref{complex algebra} and  \ref{ope-rel-and-tran-kcxt}.\\
	2. Proof is similar to the proof of Theorem \ref{complex algebra}(2).  
\end{proof}

 Now, we will show that for a topological dBa $\mathfrak{O}$, $\mathbb{KC}(\mathfrak{O})$ is a reflexive and transitive Kripke context. For that,  we first prove the following lemma.
\begin{lemma}
	\label{tdBadual}
	{\rm Let $\mathfrak{D}$ be a topological dBa. Then for all $a\in D$, $\textbf{I}^{\delta}\textbf{I}^{\delta}(a)=\textbf{I}^{\delta}(a)$ and $\textbf{C}^{\delta}\textbf{C}^{\delta}(a)=\textbf{C}^{\delta}(a)$.
		%\begin{enumerate}
		%\item $\textbf{I}^{\delta}\textbf{I}^{\delta}(a)=\textbf{I}^{\delta}(a)$ and $\textbf{C}^{\delta}\textbf{C}^{\delta}(a)=\textbf{C}^{\delta}(a)$.
		%\item  $a\sqcap a\sqsubseteq \textbf{I}^{\delta}(a)$ and $\textbf{C}^{\delta}(a)\sqsubseteq a\sqcup a$.
		%\end{enumerate} 
	}
\end{lemma}
\begin{proof}
	Let $a\in D$. By Definition \ref{tdBa}(5a), $\textbf{I}\textbf{I}(\neg a)=\textbf{I}(\neg a)$, which implies that $\neg \textbf{I}\textbf{I}(\neg a)=\neg \textbf{I}(\neg a)$. By Lemma \ref{dual operators}(2), $\textbf{I}^{\delta}(\neg\textbf{I}\neg a)=\textbf{I}^{\delta}(a)$, whence $\textbf{I}^{\delta}\textbf{I}^{\delta}(a)=\textbf{I}^{\delta}(a)$.
	Similarly, we can show that $\textbf{C}^{\delta}\textbf{C}^{\delta}(a)=\textbf{C}^{\delta}(a)$.  
\end{proof}

%As a topological dBa is also a dBao, the representation  results (Theorem \ref{rtdBao}) of dBaos also hold for topological dBas. Moreover, 
We now have
\begin{theorem}
	\label{rttdBa}
	{\rm  
		%Let $\mathfrak{O}$ be a topological dBa. Then  
		$\mathbb{KC}(\mathfrak{O}):=((\mathcal{F}_{p}(\textbf{D}), R), (\mathcal{I}_{p}(\textbf{D}), S), \Delta)$  is a reflexive  and transitive Kripke context.}
\end{theorem}
\begin{proof}
	%Let $\mathfrak{O}$ be a topological dBa. We show that $\mathbb{KC}(\mathfrak{O}):=((\mathcal{F}_{p}(\textbf{D}), R), (\mathcal{I}_{p}(\textbf{D}), S), \Delta)$  is a reflexive  and transitive Kripke context.  
	To show  $R$ is reflexive, let $F\in \mathcal{F}_{p}(\textbf{D})$ and $\textbf{I}a\in F$ for some $a\in D$. By Definition \ref{tdBa}(4a), $\textbf{I}a\sqsubseteq a$, which implies that $a\in F$, as $F$ is a filter. So $FRF$ by Lemma \ref{canonical relations}.\\
	To show $R$ is transitive, let $F,F_{1}, F_{2}\in \mathcal{F}_{p}(\mathfrak{O})$ such that $FRF_{1}$ and $F_{1}RF_{2}$. We show that $FRF_{2}$. Let $a\in F_{2}$. Then $\textbf{I}^{\delta}(a)\in F_{1}$, as $F_{1}RF_{2}$, which implies that $\textbf{I}^{\delta}\textbf{I}^{\delta}(a)\in F$, as $FRF_{1}$. So $\textbf{I}^{\delta}(a)=\textbf{I}^{\delta}\textbf{I}^{\delta}(a)\in F$, using Lemma \ref{tdBadual}. Thus $FRF_{2}$. \\
	%Therefore $R$ is reflexive and transitive.\\
	Similarly, one can show that $S$ is reflexive and transitive.   
\end{proof}
Combining Theorem \ref{rtdBao}, Corollary \ref{rtfpdBa} and Theorem \ref{rttdBa}, we get the representation results for topological dBas in terms of reflexive  and transitive Kripke contexts.
\begin{theorem}
	\label{rtdBa}
	{\rm A topological dBa $\mathfrak{O}$ is quasi-embeddable into the full complex algebra $\underline{\mathfrak{P}}^{+}(\mathbb{KC}(\mathfrak{O}))$ of the reflexive and transitive Kripke context $\mathbb{KC}(\mathfrak{O})$. \\$\mathfrak{O}_{p}$ is embeddable into the complex algebra   $\underline{\mathfrak{H}}^{+}(\mathbb{KC}(\mathfrak{O}))$ of $\mathbb{KC}(\mathfrak{O})$. Moreover,
		\begin{enumerate}
				\item If  $\mathfrak{O}$ is a topological contextual dBa then  $\mathfrak{O}$ is embeddable into  $\underline{\mathfrak{P}}^{+}(\mathbb{KC}(\mathfrak{O}))$. 
			\item If   $\mathfrak{O}$ is a topological pure dBa  then  $\mathfrak{O}$  is embeddable into the complex algebra  $\underline{\mathfrak{H}}^{+}(\mathbb{KC}(\mathfrak{O}))$ of $\mathbb{KC}(\mathfrak{O})$.
	\end{enumerate}}
\end{theorem}

\section{ Logics corresponding to the algebras}
\label{LCA}
We next  formulate the logic \textbf{CDBL}
for  contextual dBas. The logic \textbf{MCDBL} for the class of contextual dBaos, and its extension  \textbf{MCDBL4} for topological contextual dBas are both  defined with  \textbf{CDBL} as their base. In Section \ref{semisemantics}, it is shown that, apart from the algebraic semantics, the logics can be imparted a protoconcept-based semantics, due  to the representation theorems for the respective classes of algebras obtained in Sections \ref{thealgebra}.
%\textbf{MPDBL} (and hence \textbf{MPDBL4}) is based on the logic \textbf{PDBL}
% for  pure dBas. 
%Finally,  in Section \ref{cdbl}, we note  that the logics for contextual dBas and contextual dBaos  can  be extracted from the formulations of  \textbf{PDBL},  \textbf{MPDBL} and  \textbf{MPDBL4}. 
%So we first discuss  \textbf{PDBL}.
\subsection{\textbf{CDBL}}
\label{pdbl}
The language $\mathfrak{L}$ of   \textbf{CDBL}  consists of a countably infinite set \textbf{PV} of propositional variables, propositional constants  $\bot,\top$, and logical connectives $\sqcup,\sqcap,\neg,\lrcorner$. The set $\mathfrak{F}$ of formulae is given by the following  scheme:  
\[\top\mid\bot\mid p \mid \alpha\sqcup\beta\mid \alpha\sqcap\beta\mid\neg\alpha\mid\lrcorner\alpha,\] 
where $p\in \textbf{PV}$.
$\vee$ and $\wedge$ are definable connectives:
$\alpha\vee\beta:=\neg(\neg\alpha\sqcap\neg\beta)$ and $\alpha\wedge\beta:=\lrcorner(\lrcorner\alpha\sqcup\lrcorner\beta)$  for all $\alpha,\beta\in\mathfrak{F} $.
A {\it sequent} in {\rm \textbf{CDBL}} is a pair of formulae denoted by $\alpha\vdash\beta$ for $\alpha,\beta\in \mathfrak{F}$. 
%For $\alpha,\beta\in \mathfrak{F}$, 
If $\alpha\vdash\beta$ and  $\beta\vdash\alpha$, we use the abbreviation $\alpha\dashv\vdash\beta$.
\vskip 3pt

\noindent The {\it axioms}  of $\textbf{CDBL}$ are given by the following schema.
\vskip 2pt
1 $\alpha\vdash\alpha$.
\vskip 2pt

\noindent {\it Axioms for $\sqcap$ and $\sqcup$}: 

$\begin{array}{ll}

    2a~ \alpha\sqcap\beta\vdash \alpha &
	2b~ \alpha\vdash\alpha\sqcup\beta \\
	
	3a~  \alpha\sqcap\beta\vdash \beta &
	3b~ \beta\vdash\alpha\sqcup\beta \\
	
	4a~ \alpha\sqcap\beta\vdash (\alpha\sqcap\beta)\sqcap(\alpha\sqcap\beta) &
	4b~ (\alpha\sqcup\beta)\sqcup(\alpha\sqcup\beta)\vdash\alpha\sqcup \beta 
	\end{array}$
\vskip 2pt
	
\noindent {\it Axioms for  $\neg$ and $\lrcorner$}:

$\begin{array}{ll}

   5a~\neg(\alpha\sqcap \alpha)\vdash\neg \alpha &
	5b~\lrcorner \alpha\vdash\lrcorner(\alpha\sqcup \alpha)\\
		6a~\alpha\sqcap\neg \alpha\vdash \bot &
	6b~\top\vdash \alpha\sqcup \lrcorner \alpha\\
	7a~\neg\neg(\alpha\sqcap \beta)\dashv\vdash(\alpha\sqcap \beta)&
	7b~\lrcorner\lrcorner(\alpha\sqcup \beta)\dashv\vdash(\alpha\sqcup \beta)
\end{array}$
\vskip 2pt

\noindent {\it Generalization of the law of absorption}:

$\begin{array}{ll}

8a~\alpha\sqcap \alpha\vdash \alpha\sqcap(\alpha\sqcup \beta)&
	8b~\alpha\sqcup (\alpha\sqcap \beta)\vdash \alpha\sqcup \alpha\\
    9a~\alpha\sqcap \alpha\vdash \alpha\sqcap(\alpha\vee \beta)&
	9b~\alpha\sqcup(\alpha\wedge \beta)\vdash \alpha\sqcup \alpha\\
	\end{array}$
\vskip 2pt
	
\noindent  {\it Laws of distribution}:

	$\begin{array}{ll}
	
	10a~\alpha\sqcap(\beta\vee \gamma)\dashv\vdash(\alpha\sqcap \beta)\vee(\alpha\sqcap \gamma) &
	10b~\alpha\sqcup (\beta\wedge \gamma)\dashv\vdash(\alpha\sqcup \beta)\wedge(\alpha\sqcup \gamma)\\
	
\end{array}$
\vskip 2pt

\noindent  {\it Axioms for} $\bot,\top$:
 
$\begin{array}{ll}
	
	11a~  \bot\vdash \alpha &
	11b~ \alpha\vdash \top\\
	12a~\neg\top\vdash\bot &
	12b~\top\vdash\lrcorner\bot\\
	13a~\neg \bot\dashv\vdash \top\sqcap\top &
	13b~\lrcorner\top\dashv\vdash\bot\sqcup\bot\\

\end{array}$	
\vskip 2pt

\noindent {\it The compatibility 	 axiom}:

14  $(\alpha\sqcup \alpha)\sqcap (\alpha\sqcup \alpha)\dashv\vdash (\alpha\sqcap \alpha)\sqcup (\alpha\sqcap \alpha)$

%\noindent The  axiom for $\sqcap$ and $\sqcup$ are used to prove the properties of $\sqcup$, $\sqcap$ like commutative, associative. The generalization of the law of absorption is a general generalization of  the Boolean law of absorption. This means that when a dba is  transformed into a Boolean algebra, the generalisation of the law of absorption coincides with the Boolean law of observation. The compatibility axiom states that join $\sqcup$  and meet $\sqcap$ must be compatible. 
\vskip 5pt
\noindent {\it Rules of inference of {\bf CDBL}} are as follows. 
%$B,C,D,E, F, G, H, X$ are possibly empty s-hypersequents.
%$B,C,D,E, F, G, H, X$ are possibly empty s-hypersequents.
\vskip 2pt

\noindent {\it  For $\sqcap$ and $\sqcup$}:
\begin{center}
	$	\begin{array}{ll}
		\infer[(R1)]{\alpha\sqcap \gamma\vdash \beta\sqcap \gamma}{  \alpha\vdash \beta} &
		\infer[(R1)^{\prime}]{ \gamma\sqcap\alpha \vdash \gamma\sqcap \beta}{ \alpha\vdash \beta}\\
		\infer[(R2)]{  \alpha\sqcup \gamma\vdash \beta\sqcup \gamma}{  \alpha\vdash \beta}&
		\infer[(R2)^{\prime}]{  \gamma\sqcup\alpha \vdash\gamma \sqcup \beta}{  \alpha\vdash \beta}
		\end{array}$
		\end{center}
		\vskip 2pt

\noindent {\it  For $\neg, \lrcorner$}:
		\begin{center}
		    
	$	\begin{array}{ll}
		\infer[(R3)]{  \neg\beta\vdash \neg\alpha}{  \alpha\vdash \beta}&
		\infer[(R3)^{\prime}]{  \lrcorner\beta\vdash \lrcorner\alpha}{  \alpha\vdash \beta}
	\end{array}$
	\end{center}
			\vskip 2pt
{\it Transitivity}:
	\begin{center}
	$\infer[(R4)]{  \alpha\vdash\gamma}{  \alpha\vdash\beta & \beta\vdash\gamma}$
	\end{center}
	\vskip 2pt {\it Order}:
	\begin{center}
	 $\infer[\hspace*{-4pt}(R5)]{   \alpha\vdash\beta  }{  \alpha\sqcap\beta\vdash\alpha\sqcap\alpha & \alpha\sqcap\alpha\vdash\alpha\sqcap\beta &   \alpha\sqcup\beta\vdash\beta\sqcup\beta  &   \beta\sqcup\beta\vdash\alpha\sqcup\beta}$
\end{center}

\noindent $(R5)$ captures the order relation of the  contextual dBas.
%\noindent $\infer[(Sp)]{\alpha\vdash\alpha\sqcap\alpha  \alpha\sqcup\alpha\vdash\alpha}{}$
%\vskip 2pt
%\noindent We shall see in the sequel that $(Sp)$ corresponds to the defining axiom for pure dBas.
%\vskip 3pt
%\noindent {\it External rules of inference}:
%\begin{center}
	%**Use arrays!!**
	%$\begin{array}{ll}
	%	\infer[\mbox{(External contraction-EC)}]{B\mid D\mid  C}{B\mid  D\mid  D\mid  C}&
	%	\infer[\mbox{(External exchange-EE)}]{B\mid  E\mid  D\mid  C}{B\mid  D\mid  E\mid  C}\\
	%	\infer[\mbox{(External weakening-EW)}]{B\mid C}{B}
%	\end{array}$
%\end{center}
		\vskip 4pt

Derivability is defined in the standard manner: a sequent S is {\it derivable} (or
{\it provable}) in \textbf{CDBL}, if there exists a finite sequence of sequents $S_{1},\ldots, S_{m}$
such that $S_{m}$ is the sequent S and for all $k \in\{1,\ldots,m\}$ either $S_{k}$ is an axiom or $S_{k}$ is obtained by applying  rules of  \textbf{CDBL} to elements from $\{S_{1} ,\ldots, S_{k-1}\}$. 
%The following results are easy to obtain.
Let us give a few examples of derived rules and sequents.
%\begin{definition}
%{\rm A  sequent $\alpha\vdash\beta$  is said to be {\it provable} in \textbf{DBL}  if there is a finite sequence of sequents $X_{1},X_{2},. . .,X_{n}$ such that either $X_{i}$ are axiom or obtained from $X_{l},X_{m},X_{r},X_{q}$ using rule of inference where $l, m,r,q\leq i$ and $X_{n}= \alpha\vdash \beta$.
%}\end{definition}

\begin{proposition}
	\label{drive rule1}
	{\rm  The following rules are derivable in  \textbf{CDBL}.
		\begin{multicols}{2}
			\item $\infer[(R6)]{  \alpha\sqcap\alpha\vdash\beta\sqcap\gamma}{ \alpha\vdash\beta &   \alpha\vdash \gamma}$ 
			\item $\infer[(R7)]{ \beta\sqcup \gamma\vdash\alpha\sqcup\alpha}{  \beta\vdash\alpha &  \gamma\vdash \alpha}$
		\end{multicols}
		%\noindent where B,C,D,E are possibly empty s-hypersequents.
		}
\end{proposition}
%**Check R1' , R2' everywhere!!**
\begin{proof}
	$(R6)$ is derived using $(R1), (R1)^{\prime}$ and $(R4)$, while for $(R7)$ one uses $(R2), (R2)^{\prime}$ and $(R4)$.  
	%**end proof box??**$\hbox{}$
	%We will give proof for $(R6)$ using $(R1), (R1^{\prime})$ and $(R4)$. Let $\alpha,\beta,\gamma\in\mathfrak{F}$
	%\begin{prooftree}
	%	\AxiomC{$\alpha\vdash\beta$}
	%	\UnaryInfC{$\alpha\sqcap\alpha\vdash\beta\sqcap\alpha$}
	%	\AxiomC{$\alpha\vdash\gamma$}
	%	\UnaryInfC{$\beta\sqcap\alpha\vdash\beta\sqcap\gamma$}
	%	\BinaryInfC{$\alpha\sqcap\alpha\vdash\beta\sqcap\gamma$}
	%	\end{prooftree}
	%	 Proof of $(R7)$ is dual to the proof $(R6)$, using $(R2), (R2^{\prime})$ and $(R4)$
\end{proof}

\begin{theorem} 
	\label{thempdbl}
	{\rm 
		\label{other axiom of DBA}
		For $\alpha,\beta,\gamma\in\mathfrak{F}$, the following are provable in {\rm \textbf{CDBL}}. 
		
		$	\begin{array}{ll}

			1a~ (\alpha\sqcap \beta)\dashv\vdash(\beta\sqcap \alpha).&
			1b~ \alpha\sqcup \beta\dashv\vdash \beta\sqcup \alpha.\\
			
			2a~ \alpha\sqcap(\beta\sqcap \gamma)\dashv\vdash(\alpha\sqcap \beta)\sqcap \gamma.&
			2b~ \alpha\sqcup(\beta\sqcup \gamma)\dashv\vdash(\alpha\sqcup \beta)\sqcup \gamma.\\
			
			3a~ (\alpha\sqcap \alpha)\sqcap \beta\dashv\vdash (\alpha\sqcap \beta).&
			3b~ (\alpha\sqcup \alpha)\sqcup \beta\dashv\vdash \alpha\sqcup \beta.\\
			
			4a~ \neg \alpha\vdash\neg (\alpha\sqcap \alpha).&
			4b~ \lrcorner(\alpha\sqcup \alpha)\vdash\lrcorner \alpha.\\
			
			5a~ \alpha\sqcap(\alpha\sqcup \beta)\vdash (\alpha\sqcap \alpha).&
			5b~ \alpha\sqcup \alpha\vdash \alpha\sqcup (\alpha\sqcap \beta).\\
			
			6a~ \alpha\sqcap(\alpha\vee \beta)\vdash \alpha\sqcap \alpha.&
			6b~ \alpha\sqcup \alpha\vdash \alpha\sqcup (\alpha\wedge \beta).\\
			
			7a~ \bot \vdash \alpha\sqcap\neg \alpha.&
			7b~ \alpha\sqcup \lrcorner \alpha\vdash \top .\\
			
			8a~ \bot\vdash \neg \top. &
			8b~ \lrcorner \bot \vdash \top.\\
			
			%9~ **(\alpha\sqcap \alpha)\sqcup (\alpha\sqcap \alpha)\vdash (\alpha\sqcup \alpha)\sqcap (\alpha\sqcup \alpha).
		\end{array}$
	}
\end{theorem}
\begin{proof}
	The proofs are straightforward and one makes use of axioms 2a, 3a, 4a, Proposition \ref{drive rule1}  and the rule $(R4)$ in most cases. For instance, here is a proof for
	$1a$: 
	%-- it uses   axioms 2a, 3a, 4a, Proposition \ref{drive rule1}  and the rule $(R4)$. 	
	\begin{center}
		$\infer{(\alpha\sqcap\beta)\vdash(\beta\sqcap\alpha)~(R4)}{4a~(\alpha\sqcap\beta)\vdash(\alpha\sqcap\beta)\sqcap(\alpha\sqcap\beta) &\infer{(\alpha\sqcap\beta)\sqcap(\alpha\sqcap\beta)\vdash\beta\sqcap\alpha~(R6)}{3a~\alpha\sqcap\beta\vdash\beta & \alpha\sqcap\beta\vdash\alpha~2a}}$
	\end{center}
	
\noindent	Interchanging $\alpha$ and $\beta$ in the above, we get  $(\beta\sqcap\alpha)\vdash(\alpha\sqcap\beta)$. \\
	$(4a)$ follows from axiom 2a and $(R3)$.
	%Therefore $(\alpha\sqcap \beta)\dashv\vdash (\beta\sqcap \alpha)$.\\
	$(7a), (8a)$ follow from axiom 11a. The remaining proofs are given in the Appendix. Note that
	%and **$7b,8b$ follow from axiom (1b). 
	the proofs of $(ib), i=1,2,3,4,5,6,7,8,$ are obtained using the axioms and rules dual to those used to derive $(ia)$.  
	%Proposition \ref{drive rule1} is also used in some of the proofs. 
	%Let $\alpha,\beta,\gamma\in\mathfrak{F}$.\\
	 \end{proof}
Definitions of valuations on the algebras and satisfaction of sequents are as follows.

\begin{definition}
	\label{valution}
	{\rm 	 Let $\textbf{D}:=(D,\sqcup,\sqcap,\neg,\lrcorner,\top_{D},\bot_{D})$ be a  contextual dBa. A {\it valuation} $v:\mathfrak{F}\rightarrow D$ on $\textbf{D}$ is a map such that for all $\alpha,\beta\in \mathfrak{F}$  the following hold.}
	\begin{enumerate}
		
		$\begin{array}{ll}
			
			1.~v(\alpha\sqcup\beta):=v(\alpha)\sqcup v(\beta).&
			4.~ v(\alpha\sqcap\beta):=v(\alpha)\sqcap v(\beta).\\
			2.~v(\lrcorner \alpha):=\lrcorner v(\alpha).&
			5.~ v(\neg \alpha):=\neg v(\alpha). \\
			3.~v(\top):=\top_{D}.&
			6.~v(\bot):=\bot_{D}.
		\end{array}$
		
	\end{enumerate}
\end{definition}
%** you have used suffix D for top and bottom later -- be uniform**
\begin{definition}
	\label{satis-hyper-sequent}
	%\label{satis-sequent}
	{\rm  {\it A sequent} $\alpha\vdash\beta$ is said to be  {\it satisfied} by a valuation $v$ on   a contextual dBa $\textbf{D}$ if and only if  $v(\alpha)\sqsubseteq v(\beta)$.  
		%$\alpha\vdash\beta$ is  {\it true} in  $\textbf{D}$ if and only if  for  all valuations  $v$ in $\textbf{D}$, $v$ satisfies $\alpha\vdash\beta$. $\alpha\vdash\beta$ is {\it valid} in the class of all  pure dBas if and only if it is true in  every  pure dBa.
		$\alpha\vdash\beta$   is  {\it true} in  $\textbf{D}$ if and only if  for  all valuations  $v$ on $\textbf{D}$, $v$ satisfies  $\alpha\vdash\beta$. $\alpha\vdash\beta$  is {\it valid} in the class of all contextual  dBas if and only if it is true in  every contextual  dBa.
} \end{definition}

\begin{theorem}[Soundness]
	\label{sound1}
	{\rm If  a sequent $\alpha\vdash\beta$ is provable in  \textbf{CDBL}  then it is  valid in the class of all contextual dBas.
}	 \end{theorem}

\begin{proof}
	The proof that  all the axioms of  \textbf{CDBL} are valid in the class of all contextual dBas is straightforward and can be obtained using Proposition \ref{pro1.5} and Definition \ref{DBA}. 
	One then  needs to verify that the rules of inference preserve validity. 
	%The case for the external rules is easy. 
	%We argue for the other rules.
	Using Proposition  \ref{pro1.5}, one can show that $(R1), (R2),(R1)^{\prime}$ and $(R2)^{\prime}$ preserve validity. The cases for $(R3)$ and $(R3)^{\prime}$  follow from Proposition \ref{pro2}. 
	
	%For $(R5)$, one uses  *Definition \ref{DBA} and definition of contextual dBa*.
	%We first argue for the rule $(Sp)$.	\\
	To show $(R5)$ preserves validity, let the sequents $\alpha\sqcap \beta\vdash \alpha\sqcap\alpha, \alpha\sqcap \alpha\vdash \alpha\sqcap\beta, \alpha\sqcup \beta\vdash \beta\sqcup\beta$,  and $\beta\sqcup \beta\vdash\alpha\sqcup\beta$ be valid in the class of all contextual dBas. Let $\textbf{D}$ be a contextual dBa and $v$  a valuation in $\textbf{D}$. Then $v$ satisfies  each sequent, which implies that $v(\alpha\sqcap\beta)\sqsubseteq v(\alpha\sqcap\alpha), v(\alpha\sqcap\alpha)\sqsubseteq v(\beta\sqcap\alpha), v(\alpha\sqcup\beta)\sqsubseteq v(\beta\sqcap\beta)$ and $v(\beta\sqcup\beta)\sqsubseteq v(\alpha\sqcup\beta)$. So $v(\alpha\sqcap\beta)= v(\alpha\sqcap\alpha)$ and $ v(\alpha\sqcup\beta)=v(\beta\sqcup\beta)$, as $\textbf{D}$ is contextual. This gives $v(\alpha)\sqcap v(\beta)=v(\alpha)\sqcap v(\alpha)$ and $v(\alpha)\sqcup v(\beta)=v(\beta)\sqcup v(\beta)$. Thus $v(\alpha)\sqsubseteq v(\beta)$, whence $\alpha\vdash \beta$ is satisfied by $v$. 
\end{proof}

%	 \begin{theorem}[Completeness] 
%	 \label{complete}
%	 If a sequent $\alpha\vdash\beta$ is {\rm valid} in class of all {\rm (pure)dBa} then it is provable in {\rm \textbf{DBL}(\textbf{PDBL})}.
%	 
%	 \end{theorem}

As usual, the  completeness theorem is proved using the Lindenbaum-Tarski algebra of  \textbf{CDBL}, and
%We sketch the route taken by the proof.
%Now we define Lindenbaum-Tarski algebra of {\rm  \textbf{DBL}(\textbf{PDBL}) }.
the algebra is constructed in the standard way as follows.
A relation $\equiv_{\vdash}$ is defined  on $\mathfrak{F}$ by:  $\alpha\equiv_{\vdash}\beta$ if and only if $\alpha\dashv\vdash\beta$, for $\alpha,\beta\in \mathfrak{F}$. $\equiv_{\vdash}$ is    a congruence relation on $\mathfrak{F}$ with respect to  $\sqcup$, $\sqcap$, $\neg$, $\lrcorner$. The quotient set $\mathfrak{F}/\equiv_{\vdash}$  
%induced by the relation $\equiv_{\vdash}$ and 
with operations induced by the logical connectives, give the Lindenbaum-Tarski algebra $\mathcal{L}(\mathfrak{F}):= (\mathfrak{F}/\equiv_{\vdash},\sqcup,\sqcap,\neg,\lrcorner,[\top],[\bot])$. The axioms in {\bf CDBL} and Theorem \ref{other axiom of DBA} ensure that $\mathcal{L}(\mathfrak{F})$  is a dBa. One then has 

\begin{proposition}
	\label{lindn}
	{\rm For any formula $\alpha$ and $\beta$ the following are equivalent.
		\begin{enumerate}
			\item $\alpha\vdash\beta$ is provable in \textbf{CDBL}.
			\item $[\alpha]\sqsubseteq [\beta]$ in $\mathcal{L}(\mathfrak{F})$ of \textbf{CDBL}.
		\end{enumerate}
	}
\end{proposition}
\begin{proof}
	For $1\implies 2$, we make use of  $(R1)^{\prime}$, $(R4)$, axiom 2a and Theorem \ref{thempdbl}(2a, 3a). 
	% For $2\implies 1$,   (R5) is used.  
		\begin{prooftree}
			\AxiomC{$\alpha\vdash\beta$}
			\UnaryInfC{$\alpha\sqcap\alpha\vdash\alpha\sqcap\beta$}
			\UnaryInfC{$\alpha\sqcap\beta\vdash\alpha$}
			\UnaryInfC{$\alpha\sqcap(\alpha\sqcap\beta)\vdash\alpha\sqcap\alpha~~~\alpha\sqcap\beta\vdash\alpha\sqcap(\alpha\sqcap\beta)$}
			\UnaryInfC{$\alpha\sqcap\beta\vdash\alpha\sqcap\alpha$}
		\end{prooftree}
		So $\alpha\sqcap\alpha\dashv\vdash\alpha\sqcap\beta$, which implies that $[\alpha]\sqcap [\alpha]=[\alpha\sqcap\alpha]=[\alpha\sqcap\beta]=[\alpha]\sqcap [\beta]$. Dually we can show that $[\alpha]\sqcup [\beta]=[\beta]\sqcup [\beta]$. Therefore $[\alpha]\sqsubseteq [\beta]$.\\
		For $2\implies 1$, suppose $[\alpha]\sqsubseteq [\beta]$. Then $[\alpha]\sqcap [\beta]=[\alpha]\sqcap [\alpha]$. So $[\alpha\sqcap\beta]=[\alpha\sqcap\alpha]$. Similarly we can show that $[\alpha\sqcup\beta]=[\beta\sqcup\beta]$. Therefore $\alpha\sqcap\beta\dashv\vdash\alpha\sqcap\alpha$ and $\alpha\sqcup\beta\dashv\vdash\beta\sqcup\beta$. Now using (R5), $\alpha\vdash\beta$.
\end{proof}

It is worth noting that the  axioms of   \textbf{CDBL} are obtained by converting the dBa axioms into sequents.
Nonetheless, the system is complete with respect to the class of contextual dBas, because the relation $\equiv_{\vdash}$ provides a partial order  on the set $\mathfrak{F}/\equiv_{\vdash}$, which forces the Lindenbaum algebra $\mathcal{L}(\mathfrak{F})$ to become a contextual dBa.

\begin{theorem}
	\label{lindenbum-puredba}
	{\rm $\mathcal{L}(\mathfrak{F})$  is a contextual dBa.}
\end{theorem}
\begin{proof}
   Follows directly by  axiom 1, (R4) and Proposition \ref{lindn}.
\end{proof}
\noindent 	
%	\begin{proposition}
%	\label{pldbm}
%	{\rm The following statements are equivalent.
%	\begin{enumerate}
%	\item $\alpha\vdash\beta$ is provable in  \textbf{PDBL}.
%	\item $[\alpha]\sqsubseteq [\beta]$ in $\mathcal{L}(\mathfrak{F})$ of \textbf{PDBL}.
%	\end{enumerate}
%}	\end{proposition}
%	\begin{proof}[\rm \textbf{Proof of Theorem \ref{complete}}]
%	Let $\alpha,\beta\in \mathfrak{F}$ and the sequent $\alpha\vdash \beta$ is valid in class of all {\rm (pure)dBa}. If possible let us assume that  $\alpha\vdash \beta$ is not provable then by the Proposition \ref{ldbm} we have $[\alpha]\not\sqsubseteq [\beta]$. Now we define a map $v:\mathfrak{F}\rightarrow\mathfrak{F}/\equiv_{\vdash}$ by $v(\gamma)=[\gamma]$. Then it can be shown that $v$ is a valuation in the (pure)dBa, $\mathcal{L}(\mathfrak{F})$ and $v(\alpha)\not\sqsubseteq v(\beta)$ which is a contradiction as $\alpha\vdash\beta$ is valid. Hence our assumption was wrong and  the sequent $\alpha\vdash\beta$ is provable in {\rm \textbf{DBL}(\textbf{PDBL}) }.  
%	\end{proof}
%	\begin{theorem}
%	{\rm \textbf{DBL}(\textbf{PDBL}) } is sound and complete  with respect to class of all {\rm (pure)dBa}.
%\end{theorem}
\noindent The canonical map $v_{0}:\mathfrak{F}\rightarrow\mathfrak{F}/\equiv_{\vdash}$ defined by $v_{0}(\gamma):=[\gamma]$ for all $\gamma \in \mathfrak{F}$, can be shown to be a valuation on $\mathcal{L}(\mathfrak{F})$. 
\begin{theorem}[Completeness] 
	\label{complete1}
	{\rm  If a sequent $\alpha\vdash\beta$ is  valid in the class of all contextual dBas then it is provable in \textbf{CDBL}.	 
} \end{theorem}	
\begin{proof}
	If $\alpha\vdash\beta$ be valid in the class of all contextual dBas, 
	it is true in $\mathcal{L}(\mathfrak{F})$. Consider the canonical valuation $v_{0}$. Then  $v_{0}(\alpha)\sqsubseteq v_{0}(\beta)$ and so $[\alpha]\sqsubseteq [\beta]$. By  Proposition \ref{lindn}, it follows that $\alpha\vdash\beta$ is provable in \textbf{CDBL}. 
\end{proof}

\subsection{\textbf{MCDBL} and $\textbf{MCDBL4}$ }
\label{mpdbl}
The language $\mathfrak{L}_{1}$ of \textbf{MCDBL} adds two  unary modal connectives $\square$ and $\blacksquare$ to the language $\mathfrak{L}$ of \textbf{CDBL}. The formulae are given by the following scheme. 
\[\top\mid  \bot\mid  p\mid  \alpha\sqcup\beta\mid   \alpha\sqcap\beta\mid  \neg\alpha\mid  \lrcorner\alpha\mid   \square\alpha\mid  \blacksquare\alpha,\] 
where $p\in \textbf{PV}$. The set of formulae is denoted by $\mathfrak{F}_{1}$.
%**Change notation to F1 **
The axiom schema for \textbf{MCDBL} consists of all the axioms of  \textbf{CDBL} and the following.
\vskip 3pt
$\begin{array}{ll}
	15a~\square \alpha\sqcap\square\beta\dashv\vdash\square(\alpha\sqcap\beta)& 
	15b~ \blacksquare\alpha\sqcup\blacksquare\beta\dashv\vdash\blacksquare(\alpha\sqcup\beta)\\
	16a~ \square(\neg\bot)\dashv\vdash\neg\bot&
	16b~\blacksquare(\lrcorner\top)\dashv\vdash\lrcorner\top\\
	17a~\square(\alpha\sqcap\alpha)\dashv\vdash\square(\alpha)&
	17b~ \blacksquare(\alpha\sqcup\alpha)\dashv\vdash\blacksquare(\alpha)
\end{array}$
\vskip 3pt
\noindent {\it Rules of inference}: All the rules of \textbf{CDBL} and the following.

			$\infer[(R8)]{  \square\alpha\vdash\square\beta}{  \alpha\vdash\beta }
			$~~~
				$\infer[(R9)]{ \blacksquare\alpha\vdash\blacksquare\beta}{  \alpha\vdash\beta}
			$
	
Definable modal operators are $\lozenge, \blacklozenge$, given by $\lozenge\alpha:=\neg\square\neg\alpha$  and $\blacklozenge\alpha:=\lrcorner\blacksquare\lrcorner\alpha$.
%Provability of an s-hypersequent $G$ is defined in standard way as it is defined in case of \textbf{PDBL}. Then the following hold.
It is immediate that 
\begin{theorem}
	{\rm If a sequent $\alpha\vdash\beta$ is provable in $\textbf{CDBL}$ then it is provable in $\textbf{MCDBL}$.}
\end{theorem}
%\begin{proof}
%	As all the axioms of \textbf{PDBL} are also axioms of \textbf{MPDBL} and the rule of inference of \textbf{PDBL} are also rule of inference of \textbf{MPDBL}, a proof for the s-hypersequent $G$ in \textbf{PDBL} is also a proof for  $G$ in \textbf{MPDBL}.
%\end{proof}
A valuation $v$ on a contextual dBao $\mathfrak{O}:=(D,\sqcup,\sqcap,\neg,\lrcorner,\top_{D},\bot_{D},\textbf{I},\textbf{C})$, is a map  from $\mathfrak{F}_{1}$ to $D$   that satisfies the conditions in Definition \ref{valution}  and the following  for the modal operators:
\begin{definition}
	\label{algvaludbao}
	{\rm 
		%Let $\mathfrak{O}:=(D,\sqcup,\sqcap,\neg,\lrcorner,\top_{D},\bot_{D},\textbf{I},\textbf{C})$ be a pure dBao. A {\it valuation} $v$ in $\mathfrak{O}$ is a map from **\mathfrak{F}_{1} to $D$ satisfying 
		%the following.
		
		%		\hspace{2.5cm} $v(\top):=\top_{D}$ and $v(\bot):=\bot_{D}$.
		%		
		%		\hspace{1cm}$v(\alpha\sqcap\beta):=v(\alpha)\sqcap v(\beta)$ and $v(\alpha\sqcup\beta):=v(\alpha)\sqcup v(\beta)$.
		%		
		%		\hspace{2cm} $v(\neg\alpha):=\neg v(\alpha)$ and $v(\lrcorner\alpha):=\lrcorner v(\alpha)$.
		
		%\hspace{2cm} 
		$v(\square\alpha):=\textbf{I}(v(\alpha))$ and $v(\blacksquare\alpha):=\textbf{C}(v(\alpha))$.}
\end{definition}
Definitions of satisfaction, truth and validity of  sequents are given in a similar manner as before.
%\begin{definition}
%	{\rm An s-hypersequent $G$ is said to be {\it satisfied} by a valuation $v$ in a pure dBao $\mathfrak{O}$ if and only if $v(\alpha)\sqsubseteq v(\beta)$. An s-hypersequent $G$ is said to be {\it true} in a pure dBao $\mathfrak{O}$ if and only if  each valuation $v$ in  $\mathfrak{O}$ satisfies the s-hypersequent $G$. A sequent is said to be {\it valid} in the class of pure dBao if and only if it is true in each pure dBao in that class.}
%\end{definition}

\subsubsection{$\textbf{MCDBL}\Sigma$}
\label{mpdbl4}
 $\textbf{MCDBL4}$ is obtained as a special case of the logic $\textbf{MCDBL}\Sigma$ that is defined as follows. 
 %where $\Sigma$ is  any set of   sequents  in \textbf{MPDBL}. 
%and $V_{\Sigma}$ denote the class of those  pure dBaos in which the sequents of $\Sigma$ are valid. 
\begin{definition}
	{\rm Let $\Sigma$ be  any set of   sequents  in \textbf{MCDBL}. $\textbf{MCDBL}\Sigma$ is the logic obtained from \textbf{MCDBL} by adding all the sequents in $\Sigma$ as axioms.}
	\end{definition}
\noindent Note that if $\Sigma=\emptyset$ then $\textbf{MCDBL}\Sigma$ is the same as $\textbf{MCDBL}$.  The  set $\Sigma$ required to define $\textbf{MCDBL4}$ will be given at the end of this section. Let us briefly discuss some features of $\textbf{MCDBL}\Sigma$ for any $\Sigma$ -- these would  then apply to both $\textbf{MCDBL}$ and $\textbf{MCDBL4}$.
\vskip 2pt
Let $V_{\Sigma}$ denote the class of those  contextual dBaos in which the sequents of $\Sigma$ are valid. As a consequence of axioms 15a, 16a, 17a, 15a, 16b,  17b and  rules  $(R8)$,  $(R9)$, one can conclude that  if a sequent $\alpha\vdash\beta$ is provable in $\textbf{MCDBL}\Sigma$ then it is valid in the class $V_{\Sigma}$.
As before, one has the Lindenbaum-Tarski algebra $\mathcal{L}_{\Sigma}(\mathfrak{F}_{1})$   for $\textbf{MCDBL}\Sigma$; it has additional unary operators induced by the modal operators in $\mathfrak{L}_{1}$. More precisely,  
%\begin{definition}[Lindenbaum-Tarski Algebra]
%{\rm 
%The Lindenbaum-Tarski algebra of the logic $\textbf{MPDBL}\Sigma$ is the structure 
$\mathcal{L}_{\Sigma}(\mathfrak{F}_{1}):=(\mathfrak{F}_{1}/{\equiv_{\vdash}}, \sqcup,\sqcap,\neg,\lrcorner,[\top],[\bot],f_{\square},f_{\blacksquare})$, where $f_{\square},f_{\blacksquare}$ are defined as: 
$f_{\square}([\alpha]):=[\square\alpha]$,
$f_{\blacksquare}([\alpha]):=[\blacksquare\alpha].$
\vskip 3pt 
Proposition \ref{lindn} extends to this case. Using this proposition and rules $(R8)$, $(R9)$,	one shows that the operators $f_{\square}, f_{\blacksquare}$ are monotonic:
%\end{definition}
%\begin{proposition}
%	\label{mldbm}
%	{\rm For any $\alpha,\beta \in \mathfrak{F}_{1}$, the following are equivalent.
%		\begin{enumerate}
%			\item $\alpha\vdash\beta$ is provable in $\textbf{MPDBL}\Sigma$.
%			\item  $[\alpha]\sqsubseteq [\beta]$ in $\mathcal{L}_{\Sigma}(\mathfrak{F}_{1})$.
%	\end{enumerate}}
%\end{proposition}
%\begin{proof}
%	The proof is similar to that of  Proposition \ref{ldbm}.
%\end{proof}
%Moreover, the operators $f_{\square}, f_{\blacksquare}$ are monotonic -- this is  proved using the proposition and rules $(R8)$, $(R9)$:
\begin{lemma}
	\label{monotonicity of operators}
	{\rm For $\alpha,\beta \in \mathfrak{F}_{1}$, $[\alpha]\sqsubseteq [\beta]$ in $\mathcal{L}_{\Sigma}( \mathfrak{F}_{1})$ implies that $f_{\square}([\alpha])\sqsubseteq f_{\square}([\beta])$ and $f_{\blacksquare}([\alpha])\sqsubseteq f_{\blacksquare}([\beta])$. }
\end{lemma}
%\begin{proof}
%	The proof is obtained by applying Proposition  \ref{mldbm} and rules **?** $(R9)$, $(R10)$.
%	%Let $\alpha, \beta$ are two formula and $[\alpha]\sqsubseteq [\beta]$. $\alpha\vdash\beta$  by Proposition  \ref{mldbm}.  $\square\alpha\vdash\square\beta$ and $\textbf{C}\alpha\vdash\textbf{C}\beta$  by $(R9)$ and $(R10)$. Again using Proposition \ref{mldbm}, $[\square\alpha]\sqsubseteq [\square\beta]$ and $[\textbf{C}\alpha]\sqsubseteq [\textbf{C}\beta]$. Hence  $f_{\square}([\alpha])\sqsubseteq f_{\square}([\beta])$ and $f_{C}([\alpha])\sqsubseteq f_{C}([\beta])$. 
%\end{proof}
\noindent $(\mathfrak{F}_{1}/{\equiv_{\vdash}}, \sqcup,\sqcap,\neg,\lrcorner,[\top],[\bot])$ is a contextual dBa;   Lemma \ref{monotonicity of operators} 
%gives the monotonicity of the operators $f_{\square}, f_{\blacksquare}$. 
along with axioms  16a, 16b, 17a, 17b and the result corresponding to Proposition \ref{lindn} give
\begin{theorem}
\label{lbam}
	{\rm   $\mathcal{L}_{\Sigma}(\mathfrak{F}_{1}) \in V_\Sigma$.} 
\end{theorem}
%\begin{proof}
%	We know that $(FORM(\textbf{PV}, \square, \blacksquare)/\equiv, \sqcup,\sqcap,\neg,\lrcorner,[\top],[\bot])$ is a pure dBa.  From  Lemma \ref{monotonicity of operators} it follows that the operators $f_{\square}, f_{\blacksquare}$ are monotonic. Rest of the proof follows from axioms  16a), 16b), 17a), and 17b).
%Now let $[\alpha],[\beta]\in(From(\textbf{PV}, \circ, \textbf{C})/\equiv $. Then  $f_{\circ}([\alpha]\sqcap[\beta])=f_{\circ}([\alpha\sqcap\beta])=[\circ(\alpha\sqcap\beta)]=[\circ\alpha\sqcap\circ\beta]$ by axiom 15a), which implies that $f_{\circ}([\alpha]\sqcap[\beta])=[\circ\alpha]\sqcap[\circ\beta]=f_{\circ}([\alpha])\sqcap f_{\circ}([\beta])$. Dually, we can show that $f_{\textbf{C}}([\alpha]\sqcup [\beta])= f_{\textbf{C}}([\alpha])\sqcup f_{\textbf{C}}([\beta])$. The other axiom can be  proved directly using  axiom  16a), 16b), 17a), and 17b), which implies that $\mathcal{L}_{\Sigma}(\textbf{MPDBL})$ is a  dBao. From (Sp) and axiom 2a, 2b it follwos that for a formula $\alpha$,  either $[\alpha]\sqcap [\alpha]=[\alpha]$ or $[\alpha]\sqcup [\alpha]=[\alpha]$. So $\mathcal{L}_{\Sigma}(\textbf{MPDBL})$ is a pure dBao.
%\end{proof}
One then gets in the standard manner,
\begin{theorem}[Completeness]
	\label{algeb compe mdbl}
	{\rm If a sequent $\alpha\vdash\beta$ is valid in the class $V_{\Sigma}$ then it is provable in $\textbf{MCDBL}\Sigma$.}
\end{theorem}

\noindent   \textbf{MCDBL4} is defined as the logic $\textbf{MCDBL}\Sigma$ where  $\Sigma$ contains the following:
%We extend the logic \textbf{MPDBL} to  \textbf{MPDBL4} by adding the following sequents to \textbf{MPDBL}.\\
\begin{center}
	$\begin{array}{l l}
		18a~ \square\alpha\vdash\alpha &
		18b~ \alpha\vdash\blacksquare\alpha\\
		19a~ \square\square\alpha\dashv\vdash\square\alpha & 
		19b~ \blacksquare\blacksquare\alpha\dashv\vdash\blacksquare\alpha
	\end{array}$
\end{center}
We have thus obtained  
\begin{theorem}[Soundness and Completeness]
	{\rm \noindent 
\begin{enumerate}
\item $\alpha\vdash\beta$ is provable in \textbf{MCDBL} if and only if $\alpha\vdash\beta$ is valid in the class of all  contextual dBaos.
\item $\alpha\vdash\beta$ is provable in \textbf{MCDBL4} if and only if $\alpha\vdash\beta$ is valid in the class of all topological contextual dBas.
\end{enumerate}}
\end{theorem}

\subsection{Protoconcept-based semantics for the logics}
\label{semisemantics}
As a consequence of the representation result for contextual dBas (Corollary \ref{repcdBa}), we get another semantics for \textbf{CDBL} based on the sets of protoconcepts of contexts. The required  basic definitions are derivable from those given in Section \ref{LCA}. However, for the sake of completeness, these are mentioned here. 
%$\mathfrak{H}(\mathbb{K})$ of semiconcepts  of a context $\mathbb{K}$. 
We first define  valuations, models and satisfaction for a context $\mathbb{K}:=(G,M,I)$. \\Valuations associate  formulae with  protoconcepts of $\mathbb{K}$: 
%on the set $\mathfrak{H}(\mathbb{K})$ of semiconcepts  of a context $\mathbb{K}:=(G,M,I)$.

\noindent  A {\it valuation} is a map $v:\mathfrak{F}\rightarrow \mathfrak{P}(\mathbb{K})$ such that
% $v$ preserve the logical operation, that is 

$\begin{array}{ll}
	
	v(\alpha\sqcup\beta):=v(\alpha)\sqcup v(\beta).&
	v(\alpha\sqcap\beta):=v(\alpha)\sqcap v(\beta).\\
	v(\neg \alpha):=\neg v(\alpha).&
	v(\lrcorner \alpha):=\lrcorner v(\alpha).\\
	v(\top):= (G, \emptyset).&
	v(\bot):=(\emptyset, M).
\end{array}$

\noindent A {\it model} for \textbf{CDBL} based on the context  $\mathbb{K}$ is a pair $\mathbb{M}:=(\mathfrak{P}(\mathbb{K}), v)$, where $v$ is a valuation. 
\vskip 2pt
\noindent  Let $\mathcal{K}$ denote the collection of all contexts.
\begin{definition}
	\label{validitycontext}
	A sequent $\alpha\vdash \beta$ is said to be {\it satisfied} in a model $\mathbb{M}$ based on $\mathbb{K}$ if  $v(\alpha)\sqsubseteq v(\beta)$.   $\alpha\vdash \beta$ is  {\it true} in  $\mathbb{K}$ if  it is satisfied in  every model   based on $\mathbb{K}$. $\alpha\vdash \beta$  is {\it valid} in the class $\mathcal{K}$  if it is true in  every  context $\mathbb{K}\in \mathcal{K}$.
\end{definition} 
As for any context $\mathbb{K}$ the set  $\mathfrak{P}(\mathbb{K})$ of protoconcepts of $\mathbb{K}$  forms a contextual dBa (Theorem \ref{protconcept algebra}(1)), and for any model $\mathbb{M}:=(\mathfrak{P}(\mathbb{K}),v)$, $v$ is a valuation  according to  Definition \ref{valution}, Theorem \ref{sound1} gives us the soundness of \textbf{CDBL} with respect to the above semantics. In other words, if a sequent  is provable in \textbf{CDBL} then it is valid in the class $\mathcal{K}$.

For the completeness result, 
%To prove the completeness theorem for \textbf{PDBL}, 
we make use of the (Representation) Corollary \ref{repcdBa}   for contextual dBas and the fact that the Lindenbaum-Tarski algebra $\mathcal{L}(\mathfrak{F})$ is a contextual dBa (Theorem \ref{lindenbum-puredba}). 
% from Theorem \ref{lindenbum-puredba} and Theorem \ref{semiconceptembedding} 
From these  it follows that $h:\mathfrak{F}/\equiv_{\vdash} \rightarrow  \mathfrak{P}(\mathbb{K}(\mathcal{L}(\mathfrak{F})))$ defined as $h([\alpha]):= (F_{[\alpha]}, I_{[\alpha]})$ for all $[\alpha]\in \mathfrak{F}/\equiv_{\vdash}$, is  an embedding. Recall the canonical map $v_0: \mathfrak{F}\rightarrow\mathfrak{F}/\equiv_{\vdash} $ defined in Section \ref{LCA}. The composition $v_{1}:=h\circ v_{0}$ is then a valuation, which implies that $\mathbb{M}(\mathcal{L}(\mathfrak{F})):=( \mathfrak{P}(\mathbb{K}(\mathcal{L}(\mathfrak{F}))), v_{1})$ is a model for \textbf{CDBL}.

\begin{theorem}[Completeness]
	\label{compl-mcpdbl}
	{\rm If a sequent $\alpha\vdash\beta$  is valid in $\mathcal{K}$ then $\alpha\vdash\beta$ is provable in \textbf{CDBL}.}
\end{theorem}
\begin{proof}
If possible, suppose $\alpha\vdash\beta$ is not provable in \textbf{CDBL}. By  Proposition \ref{lindn},  $[\alpha]\not\sqsubseteq [\beta]$. By Proposition \ref{pro1}(3), either $[\alpha]\sqcap [\alpha]\not\sqsubseteq_{\sqcap} [\beta]\sqcap [\beta]$  or $[\alpha]\sqcup [\alpha]\not\sqsubseteq_{\sqcup} [\beta]\sqcup [\beta]$. 
	  Then there exists a prime filter $F_{0}$ in  $\mathcal{L}(\mathfrak{F})_{\sqcap}$ (a Boolean algebra by Proposition \ref{pro1}) such that $[\alpha]\sqcap[\alpha]\in F_{0}$ and $[\beta]\sqcap[\beta]\notin F_{0}$. By Lemma \ref{lema1}, there exists a  filter $F$ in $\mathcal{L}(\mathfrak{F})$ such that $F\cap \mathcal{L}(\mathfrak{F})_{\sqcap}=F_{0}$ and as $F_{0}$ is prime,  $F\in\mathcal{F}_{p}(\mathcal{L}(\mathfrak{F}))$. As $[\alpha]\sqcap[\alpha]\in F_{0}$, $[\alpha]\sqcap[\alpha]\in F$ and $[\beta]\sqcap[\beta]\notin F$, because  $[\beta]\sqcap[\beta]\notin F_{0}$ and $[\beta]\sqcap[\beta]\in  \mathcal{L}(\mathfrak{F})_{\sqcap}$. So $[\alpha]\in F$, as $[\alpha]\sqcap[\alpha]\sqsubseteq[\alpha]$, and $[\beta]\notin F$, otherwise $[\beta]\sqcap[\beta]\in F$. This gives $F\in F_{[\alpha]}$  and $F\notin F_{[\beta]}$, whence   $F_{[\alpha]} \cancel{\subseteq} F_{[\beta]}$.
	%Therefore $\mathbb{M}^{\textbf{CDBL}},F_{i}\models \alpha_{i}$, as $F_{i}\in ext(v(\alpha_{i}))=F_{[\alpha_{i}]}$ and $\mathbb{M}^{\textbf{CDBL}},F_{i}\not\models\beta_{i}$, as $F_{i}\notin  $.
	
	In case $[\alpha]\sqcup[\alpha]\not\sqsubseteq_{\sqcap}[\beta]\sqcup[\beta]$, we can dually show that there exists $I\in \mathcal{I}_{p}(\mathcal{L}(\mathfrak{F}))$ such that $[\alpha]\notin I $ and $[\beta]\in I$ giving $I_{[\beta]} \cancel{\subseteq} I_{[\alpha]}$.
	
	So $v_{1}(\alpha)=(F_{[\alpha]},I_{[\alpha]})~ \cancel{\sqsubseteq} ~(F_{[\beta]},I_{[\beta]}) =v_{1}(\beta)$,
	which implies that $\alpha\vdash\beta$ is not true in the model $\mathbb{M}(\mathcal{L}(\mathfrak{F}))$ --  a contradiction.
	 %  So our assumption was wrong, as $G$ is valid in the class $\mathcal{K}$ which implies that $G$ is provable in  \textbf{PDBL}.
\end{proof}

In case of \textbf{MCDBL} and \textbf{MCDBL4}, instead of a context $\mathbb{K}:=(G, M, I)$, we consider a Kripke context $\mathbb{KC}:=((G, R), (M, S), I)$ based on  $\mathbb{K}:=(G, M, I)$. A valuation $v: \mathfrak{F}_{1}\rightarrow \mathfrak{P}(\mathbb{K}) $ extends the one for \textbf{CDBL}  with the following definitions for the modal operators: 	$v(\square\alpha):= f_{R}(v(\alpha))$ and $v(\blacksquare\alpha):=f_{S}(v(\alpha))$. 
%A model $\mathbb{M}:=(\mathfrak{H}(\mathbb{K}), v)$  for \textbf{MPDBL} is defined based on a Kripke context $\mathbb{KC}$.  The satisfiction  of a sequent $\alpha\vdash\beta$ in a model $\mathbb{M}$ is defined as it is defined for the case \textbf{PDBL}. 
Let us denote the class of all Kripke contexts by $\mathcal{KC}$ and that  of all reflexive and transitive Kripke contexts by $\mathcal{KC}_{RT}$. Models, satisfaction of sequents is as for \textbf{CDBL}.
%The validity of an s-hypersequent is defined as it is defined in Definition \ref{validitycontext}. 
Then it is straightforward to show that \textbf{MCDBL} and \textbf{MCDBL4} are sound with respect to the classes $\mathcal{KC}$ and $\mathcal{KC}_{RT}$ respectively. 

Note that by Theorem \ref{lbam}, for \textbf{MCDBL} the Lindenbaum-Tarski algebra $ \mathcal{L}_{\Sigma}(\mathfrak{F}_{1})$   is a contextual dBao, while for  \textbf{MCDBL4}, it  is  a topological contextual dBa. The completeness of \textbf{MCDBL} with respect to the class $\mathcal{KC}$ is then proved in a similar manner as Theorem \ref{compl-mcpdbl}, the representation result given by Theorem \ref{rtdBao}(2) being used. In case of \textbf{MCDBL4}, 
%$\mathcal{L}_{\Sigma}(\mathfrak{F}_{1})$ is a topological pure dBa and 
as a consequence  of  Theorem \ref{rttdBa}, 
$\mathbb{KC}(\mathcal{L}_{\Sigma}(\mathfrak{F}_{1}))$ is a reflexive and transitive Kripke context. Using  the (Representation) Theorem \ref{rtdBa}(1), 
one gets completeness of \textbf{MCDBL4} with respect to the class $\mathcal{KC}_{RT}$.

\section{Conclusions}
\label{candfd}
%In this paper, we propose the definition of a Kripke context $\mathbb{KC}$. The definition of Kripke context extended to  reflexive, symmetric and transitive Kripke context. We discuss how Kripke context linking the context and concept of a context of FCA with approximations space and approximations operators of rough set theory.
%
%The new classes of algebras dBao, contextual dBao and pure dBao are proposed and extend to  topological dBa, contextual topological dBa and topological pure dBa respectively.  Representation results for the algebras is also studied. A proof system \textbf{PDBL}  for pure dBa based on hypersequent calculus   is proposed. \textbf{PDBL} extend to \textbf{MPDBL} and \textbf{MPDBL4} for dBao and topological pure dBa.
%
%This work is at early stage and 

In a pioneering work unifying FCA and rough set theory,  Yao, D{\"u}ntsch and Gediga  \cite{duntsch2002modal,yao2004comparative} proposed object oriented and property oriented concepts of a context. For a context $\mathbb{K}:=(G, M, I)$, its complement is the context  $\mathbb{K}^{c}:=(G, M, -R)$, where $-R:=G\times M\setminus R$. It has been shown that the lattice of concepts of $\mathbb{K}$ is dually isomorphic (isomorphic) to that of object oriented (property oriented) concepts of $\mathbb{K}^{c}$. In the line of Wille's work, negation was introduced into the study and object oriented semiconcepts and protoconcepts of a context were proposed in \cite{howlader2018algebras,howlader2020}. It was observed that  $(A, B)$ is a protoconcept of  $\mathbb{K}$, if and only if $(A^{c}, B)$ is an object oriented protoconcept of $\mathbb{K}^{c}$. The same holds for semiconcepts of a context. For a context $\mathbb{K}$, object oriented protoconcepts   form a dBa, while object oriented semiconcepts form a pure dBa. The entire study presented  here may also be done in terms of object oriented semiconcepts and protoconcepts. In particular,  one may derive representation results for the algebras introduced here, with the help of corresponding algebras of object oriented semiconcepts and protoconcepts. 

A complete \cite{vormbrock2005semiconcept} dBa $\textbf{D}$ is one for which  the Boolean algebras  $\textbf{D}_{\sqcap}$ and $\textbf{D}_{\sqcup}$ are complete. Vormbrock and  Wille \cite{vormbrock2005semiconcept} have shown that
 %\begin{itemize}
     %\item[-] 
      any complete fully contextual (pure) dBa $\textbf{D}$ for
which  $\textbf{D}_{\sqcap}$ and $\textbf{D}_{\sqcup}$ are atomic, is isomorphic to the algebra of protoconcepts (semiconcepts) of some context.      
%    the algebra of  semiconcepts up to isomorphism the complete pure double Boolean
%algebras $\textbf{D}$ whose Boolean algebras $\textbf{D}_{\sqcap}$ and $\textbf{D}_{\sqcup}$ are atomic.
%\item[-]  Protoconcept algebras are up to isomorphism the complete fully contextual double Boolean algebras $\textbf{D}$ whose Boolean algebras $\textbf{D}_{\sqcap}$ and $\textbf{D}_{\sqcup}$ are atomic.
% \end{itemize}
This result gives rise to the question of such a characterisation in case of a complete fully contextual dBao $\textbf{D}$ for which $\textbf{D}_{\sqcap}$ and $\textbf{D}_{\sqcup}$ are atomic. 
%a new abstraction  of  **?** the full complex algebra, which is fully contextual complete double Boolean algebra with operators. The structural properties of this class of algebras yet to be investigated. 
It appears that, using Vormbrock and Wille's results and  the representation  results obtained here for dBaos in terms of the full complex algebra of protoconcepts, one should be able to obtain  the desired  characterisation.
%epresentation results  for this class of algebras using the representation  results for dBaos and those obtained in **ref**.

%\begin{enumerate}[label={$\circ$}]
%\item B. Vormbrock and R. Wille \textit{Semiconcept and protoconcept algebras: the basic theorems.} In: B. Ganter, G. Stumme, R. Wille (eds.)  \textit{Formal Concept Analysis: Foundations and Applications}, pp. $34-48 (2005)$. Springer Berlin Heidelberg
%\end{enumerate}

Another direction of investigation one may pursue, is the duality between the class of all Kripke contexts and that of all dBaos. We have shown in this work that a dBao $\mathfrak{O}$ induces a Kripke context $\mathbb{KC}(\mathfrak{O})$, and on the other hand, a Kripke context $\mathbb{KC}$ induces a dBao $\underline{\mathfrak{P}}^{+}(\mathbb{KC})$. A natural question then would be:
%So  there is a correspondence between the class of dBaos and the class of Kripke contexts. %The study of correspondence characteristics is an exciting future direction.
% In particular, we are interested in the following question.
%\begin{itemize}
   % \item[-]
     is $\mathbb{KC}(\underline{\mathfrak{P}}^{+}(\mathbb{KC}))$ isomorphic to $\mathbb{KC}$?
%\end{itemize}

Topological representation results for algebras are well-studied in literature. This would serve as yet another immediate point of investigation for the algebras discussed in this work. 

Logics corresponding to dBas, pure dBas and their extensions with operators as defined here, remain an open question. The logic \textbf{MCDBL4} for topological contextual dBas is obtained as a special case of $\textbf{MCDBL}\Sigma$, where $\Sigma$ is any set of sequents in $\textbf{MCDBL}$. This gives a scheme of  obtaining  several other logics that may express   properties of  dBaos and corresponding classes of Kripke contexts besides the ones considered here. For topological contextual dBas and correspondingly, reflexive and transitive Kripke contexts,   \textbf{MCDBL4} with $\Sigma$ containing the modal axioms for  reflexivity and transitivity, serves the purpose. One may well include other axioms (such as  symmetry) in $\Sigma$, and investigate the resulting modal systems.

\backmatter

%\bmhead{Supplementary information}

%If your article has accompanying supplementary file/s please state so here. Authors reporting data from electrophoretic gels and blots should supply the full unprocessed scans for key as part of their Supplementary information. This may be requested by the editorial team/s if it is missing. Please refer to Journal-level guidance for any specific requirements.

%\bmhead{Acknowledgments}

%Acknowledgments are not compulsory. Where included they should be brief. Grant or contribution numbers may be acknowledged.

%Please refer to Journal-level guidance for any specific requirements.

%\section*{Declarations}

%Some journals require declarations to be submitted in a standardised format. Please check the Instructions for Authors of the journal to which you are submitting to see if you need to complete this section. If yes, your manuscript must contain the following sections under the heading `Declarations':

%\begin{itemize}
%	\item Funding
%	\item Conflict of interest/Competing interests (check journal-specific guidelines for which heading to use)
%	\item Ethics approval 
%	\item Consent to participate
%	\item Consent for publication
%	\item Availability of data and materials
%	\item Code availability 
%	\item Authors' contributions
%\end{itemize}

%% BioMed_Central_Bib_Style_v1.01

\newpage
\begin{appendices}
	
	\section{Proofs}\label{secA1}
	
	%\begin{proof} 
	{\it Proof in Theorem \ref{gen of dBa with oper}, that $(D,\sqcap,\sqcup,\neg,\top,\bot)$ is a Boolean algebra:}
		Let $\mathfrak{O}$ be a dBao such that for all $a\in D$, $\neg a=\lrcorner a$ and $\neg\neg a=a$. Let $x,y\in D$ such that  $x\sqsubseteq y$ and $y\sqsubseteq x$. By Proposition \ref{pro1.5}(4),  $x\sqcap x=y\sqcap y$ and $x\sqcup x=y\sqcup y$. Using Proposition \ref{pro2}(3), $\neg\neg x=\neg\neg y$ and so $x=y$. Therefore $(D,\sqsubseteq)$ is a partially ordered set. From Definition \ref{DBA}(2a and 2b) it follows that $\sqcap,\sqcup$ is commutative, while Definition \ref{DBA}(3a and 3b) gives that $\sqcap,\sqcup$ is associative. Using Definition \ref{DBA}(5a) and Proposition \ref{pro2}(3), $x\sqcap (x\sqcup y)=x\sqcap x=\neg\neg x$. So $x\sqcap (x\sqcup y)= x$. Again using Definition \ref{DBA}(5b) and Proposition \ref{pro2}(3), $x\sqcup (x\sqcap y)= x$. Therefore $(D,\sqcap,\sqcup,\neg,\top,\bot)$ is a bounded complemented lattice. To show it is a distributive lattice, we show that for all $x,y, \in D$ $x\sqcap y= x\wedge y $ and $x\vee y= x\sqcup y$. Rest of the proof follows from Definition \ref{DBA}(6a and 6b).
		
		Let $x,y\in D$.
		%	$\neg x\sqcap \neg y\sqsubseteq \neg x$ and $\neg x\sqcap \neg y\sqsubseteq \neg y.$ So by Proposition \ref{pro2}(2), $\neg\neg x\sqsubseteq \neg(\neg x\sqcap\neg y)$ and $\neg\neg y\sqsubseteq \neg(\neg x\sqcap\neg y).$ Then Proposition \ref{pro1.5}(6) gives $\neg\neg x\sqcap \neg\neg y\sqsubseteq \neg(\neg x\sqcap\neg y)\sqcap\neg\neg y$ and $\neg\neg y\sqcap \neg(\neg x\sqcap\neg y)\sqsubseteq\neg(\neg x\sqcap\neg y)\sqcap \neg(\neg x\sqcap\neg y).$ Therefore $\neg\neg x\sqcap \neg\neg y\sqsubseteq\neg(\neg x\sqcap\neg y)\sqcap \neg(\neg x\sqcap\neg y).$  By Proposition \ref{pro2}(1), $\neg\neg x\sqcap \neg\neg y\sqsubseteq\neg(\neg x\sqcap\neg y)$,  that is $(x\sqcap x)\sqcap (y\sqcap y)\sqsubseteq x\vee y.$ By axiom $(1a)$ and $(3a)$, $x\sqcap y\sqsubseteq x\vee y$.
		Then $x,y\sqsubseteq x\sqcup y.$ Proposition \ref{pro2}(2) gives $\neg(x\sqcup y)\sqsubseteq \neg x,\neg y.$ Therefore by  Proposition \ref{pro1.5}(6), $\neg (x\sqcup y)\sqcap\neg y\sqsubseteq \neg x\sqcap\neg y$ and $\neg (x\sqcup y)\sqcap\neg (x\sqcup y)\sqsubseteq \neg (x\sqcup y)\sqcap\neg y.$ So $\neg (x\sqcup y)\sqcap \neg (x\sqcup y)\sqsubseteq \neg x\sqcap\neg y.$ By  Proposition \ref{pro2}(1),  $\neg(x\sqcup y)\sqsubseteq \neg x\sqcap\neg y,$ and by Proposition \ref{pro2}(2),  $\neg(\neg x\sqcap\neg y)\sqsubseteq \neg\neg (x\sqcup y)=(x\sqcup y)\sqcap (x\sqcup y)\sqsubseteq x\sqcup y.$ Hence $x\vee y\sqsubseteq x\sqcup y.$
		Using Proposition \ref{pro1.5}(5) and Proposition \ref{pro2}(2),  $\neg x\sqcap\neg y\sqsubseteq \neg x,\neg y$. So $\neg\neg x\sqsubseteq \neg(\neg x\sqcap\neg y)$ and $\neg\neg y\sqsubseteq \neg (\neg x\sqcap \neg y)$. Therefore $x\sqsubseteq \neg(\neg x\sqcap\neg y)=x\vee y$ and $ y\sqsubseteq \neg (\neg x\sqcap \neg y)=x\vee y$. Proposition \ref{pro1.5}(6) gives $x\sqcup y \sqsubseteq x\vee y$, as $(x\vee y)\sqcup (x\vee y)=\lrcorner\lrcorner (x\vee y)=\neg\neg(x\vee y)=x\vee y$. So $x\sqcup y=x\vee y$. Dually we can show that $x\sqcap y=x\wedge y$.	
		% From Definition \ref{DBA}(6a and 6b) it follows that $(D,\sqcap,\sqcup,\neg,\top,\bot,\sqsubseteq)$ is a complemented distributive lattice and hence a Boolean algebra.\\
	%\end{proof}
	\vskip 4pt 
	%\begin{proof}
	\noindent {\it Proof of Theorem \ref{thempdbl}:}
	\vskip 3pt 
	%The proofs are straightforward and one makes use of axioms 2a, 3a, 4a, Proposition \ref{drive rule1}  and the rule $(R4)$ in most cases. 
	%**Give the axiom/rule used**	
	{\tiny
		%\noindent	$1a$: 	\begin{center}
		%		$\infer{(\alpha\sqcap\beta)\vdash(\beta\sqcap\alpha)}{(\alpha\sqcap\beta)\vdash(\alpha\sqcap\beta)\sqcap(\alpha\sqcap\beta) &\infer{(\alpha\sqcap\beta)\sqcap(\alpha\sqcap\beta)\vdash\beta\sqcap\alpha}{\alpha\sqcap\beta\vdash\beta & \alpha\sqcap\beta\vdash\alpha}}$
		%	\end{center}
		%	Interchanging $\alpha$ and $\beta$ in the above, we get  $(\beta\sqcap\alpha)\vdash(\alpha\sqcap\beta)$. 
		
		\noindent $2a.$
		
		$\infer{(\alpha\sqcap\beta)\sqcap\gamma\vdash\beta\sqcap\gamma~(R4)~\mbox{--~(I)}}{4a~(\alpha\sqcap\beta)\sqcap\gamma\vdash((\alpha\sqcap\beta)\sqcap\gamma)\sqcap((\alpha\sqcap\beta)\sqcap\gamma) & \infer{((\alpha\sqcap\beta)\sqcap\gamma)\sqcap((\alpha\sqcap\beta)\sqcap\gamma)\vdash\beta\sqcap\gamma~(R6)}{ \infer{(R4)~(\alpha\sqcap\beta)\sqcap\gamma\vdash\beta}{2a~(\alpha\sqcap\beta)\sqcap\gamma\vdash(\alpha\sqcap\beta) & \alpha\sqcap\beta\vdash\beta~3a}& (\alpha\sqcap\beta)\sqcap\gamma\vdash\gamma~3a}}$
		
	Now,\\
		$\infer{(\alpha\sqcap\beta)\sqcap\gamma\vdash\alpha\sqcap(\beta\sqcap\gamma)~(R4)}{4a~(\alpha\sqcap\beta)\sqcap\gamma\vdash((\alpha\sqcap\beta)\sqcap\gamma)\sqcap((\alpha\sqcap\beta)\sqcap\gamma)&\infer{((\alpha\sqcap\beta)\sqcap\gamma)\sqcap((\alpha\sqcap\beta)\sqcap\gamma)\vdash\alpha\sqcap(\beta\sqcap\gamma)~(R6)}{ \infer{(\alpha\sqcap\beta)\sqcap\gamma\vdash\alpha~(R4)}{2a~(\alpha\sqcap\beta)\sqcap\gamma\vdash \alpha\sqcap\beta & \alpha\sqcap\beta\vdash\alpha~2a}& (\alpha\sqcap\beta)\sqcap\gamma\vdash\beta\sqcap\gamma~\mbox{(from (I) above)} }}$
		
		Similarly we can show that $\alpha\sqcap(\beta\sqcap\gamma)\vdash(\alpha\sqcap\beta)\sqcap\gamma$. \\		%Therefore $\alpha\sqcap(\beta\sqcap\gamma)\dashv\vdash(\alpha\sqcap\beta)\sqcap\gamma$.\\
		
		\noindent $3a.$\\
		$\infer{(\alpha\sqcap\alpha)\sqcap\beta\vdash\alpha\sqcap\beta~(R4)}{4a~(\alpha\sqcap\alpha)\sqcap\beta\vdash((\alpha\sqcap\alpha)\sqcap\beta)\sqcap((\alpha\sqcap\alpha)\sqcap\beta)&\infer{((\alpha\sqcap\alpha)\sqcap\beta)\sqcap((\alpha\sqcap\alpha)\sqcap\beta)\vdash\alpha\sqcap\beta~(R6)}{\infer{(R4)~(\alpha\sqcap\alpha)\sqcap\beta\vdash\alpha}{2a~(\alpha\sqcap\alpha)\sqcap\beta\vdash\alpha\sqcap\alpha & \alpha\sqcap\alpha\vdash\alpha~2a}&(\alpha\sqcap\alpha)\sqcap\beta\vdash\beta~3a}}$\\
		
		$\infer{\alpha\sqcap\beta\vdash(\alpha\sqcap\alpha)\sqcap\beta~(R4)}{4a~\alpha\sqcap\beta\vdash(\alpha\sqcap\beta)\sqcap(\alpha\sqcap\beta)&\infer{(\alpha\sqcap\beta)\sqcap(\alpha\sqcap\beta)\vdash(\alpha\sqcap\alpha)\sqcap\beta~(R6)}{\infer{(R4)~\alpha\sqcap\beta\vdash\alpha\sqcap\alpha}{4a~\alpha\sqcap\beta\vdash(\alpha\sqcap\beta)\sqcap(\alpha\sqcap\beta)&\infer{(\alpha\sqcap\beta)\sqcap(\alpha\sqcap\beta)\vdash\alpha\sqcap\alpha~(R6)}{2a~(\alpha\sqcap\beta)\vdash\alpha & (\alpha\sqcap\beta)\vdash\alpha~2a}}& \alpha\sqcap\beta\vdash\beta~3a}}$\\
		%So  $(\alpha\sqcap \alpha)\sqcap \beta\dashv\vdash (\alpha\sqcap \beta)$.\\
		%$4a.$ Proof follows from axiom $(2a)$ and $(R3)$.\\
		$5a.$ \\
		$\infer{\alpha\sqcap(\alpha\sqcup\beta)\vdash\alpha\sqcap\alpha~(R4)}{4a~\alpha\sqcap(\alpha\sqcup\beta)\vdash(\alpha\sqcap(\alpha\sqcup\beta))\sqcap(\alpha\sqcap(\alpha\sqcup\beta))&\infer{(\alpha\sqcap(\alpha\sqcup\beta))\sqcap(\alpha\sqcap(\alpha\sqcup\beta))\vdash\alpha\sqcap\alpha~(R6)}{2a~\alpha\sqcap(\alpha\sqcup\beta)\vdash\alpha & \alpha\sqcap(\alpha\sqcup\beta)\vdash\alpha~2a}}$\\
		$6a.$ Proof is identical to that of $5a.$
%		$\infer{\alpha\sqcap(\alpha\vee\beta)\vdash\alpha\sqcap\alpha~(R4)}{4a~\alpha\sqcap(\alpha\vee\beta)\vdash(\alpha\sqcap(\alpha\vee\beta))\sqcap(\alpha\sqcap(\alpha\vee\beta))&\infer{(\alpha\sqcap(\alpha\vee\beta))\sqcap(\alpha\sqcap(\alpha\vee\beta))\vdash\alpha\sqcap\alpha~(R6)}{2a~\alpha\sqcap(\alpha\vee\beta)\vdash\alpha & \alpha\sqcap(\alpha\vee\beta)\vdash\alpha~2a}}$
	} 
%\end{proof}

	%%=============================================%%
	%% For submissions to Nature Portfolio Journals %%
	%% please use the heading ``Extended Data''.   %%
	%%=============================================%%
	
	%%=============================================================%%
	%% Sample for another appendix section			       %%
	%%=============================================================%%
	
	%% \section{Example of another appendix section}\label{secA2}%
	%% Appendices may be used for helpful, supporting or essential material that would otherwise 
	%% clutter, break up or be distracting to the text. Appendices can consist of sections, figures, 
	%% tables and equations etc.
	\end{appendices}

%%===========================================================================================%%
%% If you are submitting to one of the Nature Portfolio journals, using the eJP submission   %%
%% system, please include the references within the manuscript file itself. You may do this  %%
%% by copying the reference list from your .bbl file, paste it into the main manuscript .tex %%
%% file, and delete the associated \verb+\bibliography+ commands.                            %%
%%===========================================================================================%%

%\bibliographystyle{sn-mathphys}
%\bibliography{final}% common bib file
%% if required, the content of .bbl file can be included here once bbl is generated
%%\input sn-article.bbl

%% Default %%
%%\input sn-sample-bib.tex%
\end{document}